\documentclass[preprint,12pt]{elsarticle}

\usepackage[a4paper, total={7in, 9in}]{geometry}
\usepackage{mathrsfs}  
\usepackage{amsbsy}
\usepackage{amssymb}
\usepackage{amsthm}
\usepackage{amsmath}
 \usepackage[usenames,dvipsnames]{pstricks}
 \usepackage[utf8]{inputenc}
\usepackage{epsfig}
\usepackage{graphicx}
\usepackage{pst-grad} % For gradients
\usepackage{pst-plot} % For axes
\usepackage[space]{grffile} % For spaces in paths
\usepackage{etoolbox} % For spaces in paths
\usepackage{lineno}
\makeatletter % For spaces in paths
\patchcmd\Gread@eps{\@inputcheck#1 }{\@inputcheck"#1"\relax}{}{}
\makeatother 
 
 \newcommand\scalemath[2]{\scalebox{#1}{\mbox{\ensuremath{\displaystyle #2}}}}
  \newcommand\De[2]{D_n^s({#1},{#2})}
 
  \newcommand{\blobbed}{W^c_b(\atc )}
 \newcommand{\pos}{W^{+c+}(\atc )}
 \newcommand{\atc}{\tilde{C}_n}
 \newcommand{\tl}{TL\tilde{C}_n(q,Q)}
 \newcommand{\twob}{2BTL_n(q,Q)}
 \newcommand{\simp}{SB_n(q,Q,\kappa)}
 \newcommand\s{\mathfrak{s}}
\journal{ }

\numberwithin{equation}{section}

\theoremstyle{plain}
\newtheorem{teo}{Theorem}[section]
\newtheorem{coro}[teo]{Corollary}%[section]

\newtheorem{defi}[teo]{Definition}%[section]
\newtheorem{exa}[teo]{Example}
\newtheorem{lem}[teo]{Lemma}
\newtheorem{propos}[teo]{Proposition}%[section]
\newtheorem{rem}[teo]{Remark}

\newtheorem{thmx}{Theorem}

\usepackage{hyperref}
\usepackage{lipsum}

\newenvironment{demo}{\noindent \textit{Proof:}  \rm}{\quad \hfill $\square$}

\def\remarknotcontained{% \usepackage[usenames,dvipsnames]{pstricks}
\psscalebox{.5} % Change this value to rescale the drawing.
{
\begin{pspicture}[shift=-1.675](0,-1.675)(4.6141686,1.675)
\psdots[linecolor=black, dotsize=0.2](2.30721,-0.4625)
\psdots[linecolor=black, dotsize=0.2](2.30721,0.3375)
\psdots[linecolor=black, dotsize=0.2](2.30721,1.1375)
\psdots[linecolor=black, dotsize=0.2](3.1072102,-1.2625)
\psdots[linecolor=black, dotsize=0.2](3.1072102,-0.4625)
\psdots[linecolor=black, dotsize=0.2](3.1072102,0.3375)
\rput{-53.370266}(1.1095182,2.157458){\psellipse[linecolor=black, linewidth=0.04, dimen=outer](2.7009602,-0.025)(0.93125,0.35)}
\psline[linecolor=black, linewidth=0.04, arrowsize=0.05291667cm 2.0,arrowlength=1.4,arrowinset=0.0]{->}(3.09471,-0.05)(4.70721,-0.05)
\pscircle[linecolor=black, linewidth=0.04, dimen=outer](2.6697102,0.0){1.675}
\psline[linecolor=black, linewidth=0.04, arrowsize=0.05291667cm 2.0,arrowlength=1.4,arrowinset=0.0]{->}(0.9822101,-0.05)(-0.09278992,-0.0625)
\rput(-1.3,0){\scalebox{2}{$G(w)$}}
\rput(6,0){\scalebox{2}{$G(w')$}}
\end{pspicture}
}
}

\def\obliquesIJIJIJIJIJI{% \usepackage[usenames,dvipsnames]{pstricks}
\psscalebox{1.0 1.0} % Change this value to rescale the drawing.
{
\begin{pspicture}[shift=-3](0,-2.9006946)(20.4,2.9006946)
\definecolor{colour0}{rgb}{1.0,0.2,0.2}
\definecolor{colour1}{rgb}{0.2,0.2,1.0}
\psdots[linecolor=black, dotsize=0.2](7.6,2.4)
\psdots[linecolor=black, dotsize=0.2](8.0,1.6)
\psdots[linecolor=black, dotsize=0.2](8.4,0.8000001)
\psdots[linecolor=black, dotsize=0.2](8.8,0.0)
\psdots[linecolor=black, dotsize=0.2](9.2,-0.79999995)
\psdots[linecolor=black, dotsize=0.2](9.6,-1.5999999)
\psdots[linecolor=black, dotsize=0.2](10.0,-2.3999999)
\psdots[linecolor=black, dotsize=0.2](10.4,2.4)
\psdots[linecolor=black, dotsize=0.2](10.8,1.6)
\psdots[linecolor=black, dotsize=0.2](11.2,0.8000001)
\psdots[linecolor=black, dotsize=0.2](11.6,0.0)
\psdots[linecolor=black, dotsize=0.2](12.0,-0.79999995)
\psdots[linecolor=black, dotsize=0.2](12.4,-1.5999999)
\psdots[linecolor=black, dotsize=0.2](12.8,-2.3999999)
\psdots[linecolor=black, dotstyle=o, dotsize=0.2, fillcolor=white](8.0,2.4)
\psdots[linecolor=black, dotstyle=o, dotsize=0.2, fillcolor=white](8.4,2.4)
\psdots[linecolor=black, dotstyle=o, dotsize=0.2, fillcolor=white](8.8,2.4)
\psdots[linecolor=black, dotstyle=o, dotsize=0.2, fillcolor=white](9.2,2.4)
\psdots[linecolor=black, dotstyle=o, dotsize=0.2, fillcolor=white](9.6,2.4)
\psdots[linecolor=black, dotstyle=o, dotsize=0.2, fillcolor=white](10.0,2.4)
\psdots[linecolor=black, dotstyle=o, dotsize=0.2, fillcolor=white](10.0,2.0)
\psdots[linecolor=black, dotstyle=o, dotsize=0.2, fillcolor=white](10.4,2.0)
\psdots[linecolor=black, dotstyle=o, dotsize=0.2, fillcolor=white](10.4,1.6)
\psdots[linecolor=black, dotstyle=o, dotsize=0.2, fillcolor=white](10.4,1.2)
\psdots[linecolor=black, dotstyle=o, dotsize=0.2, fillcolor=white](10.8,1.2)
\psdots[linecolor=black, dotstyle=o, dotsize=0.2, fillcolor=white](9.6,2.0)
\psdots[linecolor=black, dotstyle=o, dotsize=0.2, fillcolor=white](9.2,2.0)
\psdots[linecolor=black, dotstyle=o, dotsize=0.2, fillcolor=white](8.8,2.0)
\psdots[linecolor=black, dotstyle=o, dotsize=0.2, fillcolor=white](8.4,2.0)
\psdots[linecolor=black, dotstyle=o, dotsize=0.2, fillcolor=white](8.0,2.0)
\psdots[linecolor=black, dotstyle=o, dotsize=0.2, fillcolor=white](8.4,1.6)
\psdots[linecolor=black, dotstyle=o, dotsize=0.2, fillcolor=white](8.8,1.6)
\psdots[linecolor=black, dotstyle=o, dotsize=0.2, fillcolor=white](9.2,1.6)
\psdots[linecolor=black, dotstyle=o, dotsize=0.2, fillcolor=white](9.6,1.6)
\psdots[linecolor=black, dotstyle=o, dotsize=0.2, fillcolor=white](10.0,1.6)
\psdots[linecolor=black, dotstyle=o, dotsize=0.2, fillcolor=white](10.0,1.2)
\psdots[linecolor=black, dotstyle=o, dotsize=0.2, fillcolor=white](9.6,1.2)
\psdots[linecolor=black, dotstyle=o, dotsize=0.2, fillcolor=white](9.2,1.2)
\psdots[linecolor=black, dotstyle=o, dotsize=0.2, fillcolor=white](8.8,1.2)
\psdots[linecolor=black, dotstyle=o, dotsize=0.2, fillcolor=white](8.4,1.2)
\psdots[linecolor=black, dotstyle=o, dotsize=0.2, fillcolor=white](8.8,0.8000001)
\psdots[linecolor=black, dotstyle=o, dotsize=0.2, fillcolor=white](9.2,0.8000001)
\psdots[linecolor=black, dotstyle=o, dotsize=0.2, fillcolor=white](9.6,0.8000001)
\psdots[linecolor=black, dotstyle=o, dotsize=0.2, fillcolor=white](9.6,0.8000001)
\psdots[linecolor=black, dotstyle=o, dotsize=0.2, fillcolor=white](10.0,0.8000001)
\psdots[linecolor=black, dotstyle=o, dotsize=0.2, fillcolor=white](10.4,0.8000001)
\psdots[linecolor=black, dotstyle=o, dotsize=0.2, fillcolor=white](10.8,0.8000001)
\psdots[linecolor=black, dotstyle=o, dotsize=0.2, fillcolor=white](10.8,0.40000007)
\psdots[linecolor=black, dotstyle=o, dotsize=0.2, fillcolor=white](11.2,0.40000007)
\psdots[linecolor=black, dotstyle=o, dotsize=0.2, fillcolor=white](11.2,0.0)
\psdots[linecolor=black, dotstyle=o, dotsize=0.2, fillcolor=white](10.8,0.0)
\psdots[linecolor=black, dotstyle=o, dotsize=0.2, fillcolor=white](10.4,0.40000007)
\psdots[linecolor=black, dotstyle=o, dotsize=0.2, fillcolor=white](10.4,0.0)
\psdots[linecolor=black, dotstyle=o, dotsize=0.2, fillcolor=white](10.0,0.40000007)
\psdots[linecolor=black, dotstyle=o, dotsize=0.2, fillcolor=white](10.0,0.0)
\psdots[linecolor=black, dotstyle=o, dotsize=0.2, fillcolor=white](9.6,0.40000007)
\psdots[linecolor=black, dotstyle=o, dotsize=0.2, fillcolor=white](9.6,0.0)
\psdots[linecolor=black, dotstyle=o, dotsize=0.2, fillcolor=white](9.2,0.40000007)
\psdots[linecolor=black, dotstyle=o, dotsize=0.2, fillcolor=white](9.2,0.0)
\psdots[linecolor=black, dotstyle=o, dotsize=0.2, fillcolor=white](8.8,0.40000007)
\psdots[linecolor=black, dotstyle=o, dotsize=0.2, fillcolor=white](9.2,-0.39999992)
\psdots[linecolor=black, dotstyle=o, dotsize=0.2, fillcolor=white](9.6,-0.39999992)
\psdots[linecolor=black, dotstyle=o, dotsize=0.2, fillcolor=white](10.0,-0.39999992)
\psdots[linecolor=black, dotstyle=o, dotsize=0.2, fillcolor=white](10.4,-0.39999992)
\psdots[linecolor=black, dotstyle=o, dotsize=0.2, fillcolor=white](10.8,-0.39999992)
\psdots[linecolor=black, dotstyle=o, dotsize=0.2, fillcolor=white](11.2,-0.39999992)
\psdots[linecolor=black, dotstyle=o, dotsize=0.2, fillcolor=white](11.6,-0.39999992)
\psdots[linecolor=black, dotstyle=o, dotsize=0.2, fillcolor=white](11.6,-0.79999995)
\psdots[linecolor=black, dotstyle=o, dotsize=0.2, fillcolor=white](11.2,-0.79999995)
\psdots[linecolor=black, dotstyle=o, dotsize=0.2, fillcolor=white](10.8,-0.79999995)
\psdots[linecolor=black, dotstyle=o, dotsize=0.2, fillcolor=white](10.4,-0.79999995)
\psdots[linecolor=black, dotstyle=o, dotsize=0.2, fillcolor=white](10.0,-0.79999995)
\psdots[linecolor=black, dotstyle=o, dotsize=0.2, fillcolor=white](9.6,-0.79999995)
\psdots[linecolor=black, dotstyle=o, dotsize=0.2, fillcolor=white](9.6,-1.1999999)
\psdots[linecolor=black, dotstyle=o, dotsize=0.2, fillcolor=white](10.0,-1.1999999)
\psdots[linecolor=black, dotstyle=o, dotsize=0.2, fillcolor=white](10.4,-1.1999999)
\psdots[linecolor=black, dotstyle=o, dotsize=0.2, fillcolor=white](10.8,-1.1999999)
\psdots[linecolor=black, dotstyle=o, dotsize=0.2, fillcolor=white](11.2,-1.1999999)
\psdots[linecolor=black, dotstyle=o, dotsize=0.2, fillcolor=white](11.6,-1.1999999)
\psdots[linecolor=black, dotstyle=o, dotsize=0.2, fillcolor=white](12.0,-1.1999999)
\psdots[linecolor=black, dotstyle=o, dotsize=0.2, fillcolor=white](12.0,-1.5999999)
\psdots[linecolor=black, dotstyle=o, dotsize=0.2, fillcolor=white](11.6,-1.5999999)
\psdots[linecolor=black, dotstyle=o, dotsize=0.2, fillcolor=white](11.2,-1.5999999)
\psdots[linecolor=black, dotstyle=o, dotsize=0.2, fillcolor=white](10.8,-1.5999999)
\psdots[linecolor=black, dotstyle=o, dotsize=0.2, fillcolor=white](10.4,-1.5999999)
\psdots[linecolor=black, dotstyle=o, dotsize=0.2, fillcolor=white](10.0,-1.5999999)
\psdots[linecolor=black, dotstyle=o, dotsize=0.2, fillcolor=white](10.0,-1.9999999)
\psdots[linecolor=black, dotstyle=o, dotsize=0.2, fillcolor=white](10.4,-1.9999999)
\psdots[linecolor=black, dotstyle=o, dotsize=0.2, fillcolor=white](10.4,-2.3999999)
\psdots[linecolor=black, dotstyle=o, dotsize=0.2, fillcolor=white](10.8,-1.9999999)
\psdots[linecolor=black, dotstyle=o, dotsize=0.2, fillcolor=white](10.8,-2.3999999)
\psdots[linecolor=black, dotstyle=o, dotsize=0.2, fillcolor=white](11.2,-1.9999999)
\psdots[linecolor=black, dotstyle=o, dotsize=0.2, fillcolor=white](11.2,-2.3999999)
\psdots[linecolor=black, dotstyle=o, dotsize=0.2, fillcolor=white](11.6,-1.9999999)
\psdots[linecolor=black, dotstyle=o, dotsize=0.2, fillcolor=white](11.6,-2.3999999)
\psdots[linecolor=black, dotstyle=o, dotsize=0.2, fillcolor=white](12.0,-1.9999999)
\psdots[linecolor=black, dotstyle=o, dotsize=0.2, fillcolor=white](12.0,-2.3999999)
\psdots[linecolor=black, dotstyle=o, dotsize=0.2, fillcolor=white](12.4,-1.9999999)
\psdots[linecolor=black, dotstyle=o, dotsize=0.2, fillcolor=white](12.4,-2.3999999)
\psdots[linecolor=black, dotstyle=o, dotsize=0.2, fillcolor=white](10.4,-2.8)
\psdots[linecolor=black, dotstyle=o, dotsize=0.2, fillcolor=white](10.8,-2.8)
\psdots[linecolor=black, dotstyle=o, dotsize=0.2, fillcolor=white](11.2,-2.8)
\psdots[linecolor=black, dotstyle=o, dotsize=0.2, fillcolor=white](11.6,-2.8)
\psdots[linecolor=black, dotstyle=o, dotsize=0.2, fillcolor=white](12.0,-2.8)
\psdots[linecolor=black, dotstyle=o, dotsize=0.2, fillcolor=white](12.4,-2.8)
\psdots[linecolor=black, dotstyle=o, dotsize=0.2, fillcolor=white](12.8,-2.8)
\psdots[linecolor=black, dotstyle=o, dotsize=0.2, fillcolor=white](7.6,2.8000002)
\psdots[linecolor=black, dotstyle=o, dotsize=0.2, fillcolor=white](8.0,2.8000002)
\psdots[linecolor=black, dotstyle=o, dotsize=0.2, fillcolor=white](8.0,2.8000002)
\psdots[linecolor=black, dotstyle=o, dotsize=0.2, fillcolor=white](8.4,2.8000002)
\psdots[linecolor=black, dotstyle=o, dotsize=0.2, fillcolor=white](8.8,2.8000002)
\psdots[linecolor=black, dotstyle=o, dotsize=0.2, fillcolor=white](9.2,2.8000002)
\psdots[linecolor=black, dotstyle=o, dotsize=0.2, fillcolor=white](9.6,2.8000002)
\psdots[linecolor=black, dotstyle=o, dotsize=0.2, fillcolor=white](10.0,2.8000002)
\psdots[linecolor=black, dotsize=0.2](1.2,2.4)
\psdots[linecolor=black, dotsize=0.2](4.0,2.4)
\psdots[linecolor=black, dotsize=0.2](1.6,1.6)
\psdots[linecolor=black, dotsize=0.2](2.0,0.8000001)
\psdots[linecolor=black, dotsize=0.2](2.4,0.0)
\psdots[linecolor=black, dotsize=0.2](2.8,-0.79999995)
\psdots[linecolor=black, dotsize=0.2](3.2,-1.5999999)
\psdots[linecolor=black, dotsize=0.2](3.6,-2.3999999)
\psdots[linecolor=black, dotsize=0.2](4.4,1.6)
\psdots[linecolor=black, dotsize=0.2](4.8,0.8000001)
\psdots[linecolor=black, dotsize=0.2](5.2,0.0)
\psdots[linecolor=black, dotsize=0.2](5.6,-0.79999995)
\psdots[linecolor=black, dotsize=0.2](6.0,-1.5999999)
\psdots[linecolor=black, dotsize=0.2](6.4,-2.3999999)
\psline[linecolor=black, linewidth=0.04, arrowsize=0.05291667cm 2.0,arrowlength=1.4,arrowinset=0.0]{->}(6.0,0.40000007)(7.6,0.40000007)
\psline[linecolor=black, linewidth=0.04, arrowsize=0.05291667cm 2.0,arrowlength=1.4,arrowinset=0.0]{->}(12.8,0.40000007)(14.4,0.40000007)
\psdots[linecolor=black, dotsize=0.2](14.4,2.4)
\psdots[linecolor=black, dotsize=0.2](14.8,1.6)
\psdots[linecolor=black, dotsize=0.2](15.2,0.8000001)
\psdots[linecolor=black, dotsize=0.2](15.6,0.0)
\psdots[linecolor=black, dotsize=0.2](16.0,-0.79999995)
\psdots[linecolor=black, dotsize=0.2](16.4,-1.5999999)
\psdots[linecolor=black, dotsize=0.2](16.8,-2.3999999)
\psdots[linecolor=black, dotsize=0.2](17.2,2.4)
\psdots[linecolor=black, dotsize=0.2](17.6,1.6)
\psdots[linecolor=black, dotsize=0.2](18.0,0.8000001)
\psdots[linecolor=black, dotsize=0.2](18.4,0.0)
\psdots[linecolor=black, dotsize=0.2](18.8,-0.79999995)
\psdots[linecolor=black, dotsize=0.2](19.2,-1.5999999)
\psdots[linecolor=black, dotsize=0.2](19.6,-2.3999999)
\psdots[linecolor=black, dotsize=0.2](14.8,2.8000002)
\psdots[linecolor=black, dotsize=0.2](15.2,2.8000002)
\psdots[linecolor=black, dotsize=0.2](15.6,2.8000002)
\psdots[linecolor=black, dotsize=0.2](16.0,2.8000002)
\psdots[linecolor=black, dotsize=0.2](16.4,2.8000002)
\psdots[linecolor=black, dotsize=0.2](16.8,2.8000002)
\psdots[linecolor=black, dotsize=0.2](16.8,2.4)
\psdots[linecolor=black, dotsize=0.2](16.4,2.4)
\psdots[linecolor=black, dotsize=0.2](14.4,2.8000002)
\psdots[linecolor=black, dotsize=0.2](14.8,2.4)
\psdots[linecolor=black, dotsize=0.2](15.2,2.4)
\psdots[linecolor=black, dotsize=0.2](15.6,2.4)
\psdots[linecolor=black, dotsize=0.2](16.0,2.4)
\psdots[linecolor=black, dotsize=0.2](14.8,2.0)
\psdots[linecolor=black, dotsize=0.2](15.2,2.0)
\psdots[linecolor=black, dotsize=0.2](15.6,2.0)
\psdots[linecolor=black, dotsize=0.2](16.0,2.0)
\psdots[linecolor=black, dotsize=0.2](16.4,2.0)
\psdots[linecolor=black, dotsize=0.2](16.8,2.0)
\psdots[linecolor=black, dotsize=0.2](16.8,2.0)
\psdots[linecolor=black, dotsize=0.2](17.2,2.0)
\psdots[linecolor=black, dotsize=0.2](17.2,1.6)
\psdots[linecolor=black, dotsize=0.2](16.8,1.6)
\psdots[linecolor=black, dotsize=0.2](16.4,1.6)
\psdots[linecolor=black, dotsize=0.2](16.0,1.6)
\psdots[linecolor=black, dotsize=0.2](15.6,1.6)
\psdots[linecolor=black, dotsize=0.2](15.2,1.6)
\psdots[linecolor=black, dotsize=0.2](15.2,1.2)
\psdots[linecolor=black, dotsize=0.2](15.6,1.2)
\psdots[linecolor=black, dotsize=0.2](16.0,1.2)
\psdots[linecolor=black, dotsize=0.2](16.4,1.2)
\psdots[linecolor=black, dotsize=0.2](16.8,1.2)
\psdots[linecolor=black, dotsize=0.2](17.2,1.2)
\psdots[linecolor=black, dotsize=0.2](17.6,1.2)
\psdots[linecolor=black, dotsize=0.2](17.6,0.8000001)
\psdots[linecolor=black, dotsize=0.2](17.2,0.8000001)
\psdots[linecolor=black, dotsize=0.2](16.8,0.8000001)
\psdots[linecolor=black, dotsize=0.2](16.4,0.8000001)
\psdots[linecolor=black, dotsize=0.2](16.0,0.8000001)
\psdots[linecolor=black, dotsize=0.2](15.6,0.8000001)
\psdots[linecolor=black, dotsize=0.2](15.6,0.40000007)
\psdots[linecolor=black, dotsize=0.2](16.0,0.40000007)
\psdots[linecolor=black, dotsize=0.2](16.0,0.0)
\psdots[linecolor=black, dotsize=0.2](16.0,-0.39999992)
\psdots[linecolor=black, dotsize=0.2](16.4,0.40000007)
\psdots[linecolor=black, dotsize=0.2](16.4,0.0)
\psdots[linecolor=black, dotsize=0.2](16.4,-0.39999992)
\psdots[linecolor=black, dotsize=0.2](16.4,-0.79999995)
\psdots[linecolor=black, dotsize=0.2](16.4,-0.79999995)
\psdots[linecolor=black, dotsize=0.2](16.4,-1.1999999)
\psdots[linecolor=black, dotsize=0.2](16.8,0.40000007)
\psdots[linecolor=black, dotsize=0.2](16.8,0.0)
\psdots[linecolor=black, dotsize=0.2](16.8,-0.39999992)
\psdots[linecolor=black, dotsize=0.2](16.8,-0.79999995)
\psdots[linecolor=black, dotsize=0.2](16.8,-1.1999999)
\psdots[linecolor=black, dotsize=0.2](16.8,-1.5999999)
\psdots[linecolor=black, dotsize=0.2](16.8,-1.9999999)
\psdots[linecolor=black, dotsize=0.2](17.2,-2.3999999)
\psdots[linecolor=black, dotsize=0.2](17.2,-1.9999999)
\psdots[linecolor=black, dotsize=0.2](17.2,-1.5999999)
\psdots[linecolor=black, dotsize=0.2](17.2,-1.1999999)
\psdots[linecolor=black, dotsize=0.2](17.2,-0.79999995)
\psdots[linecolor=black, dotsize=0.2](17.2,-0.39999992)
\psdots[linecolor=black, dotsize=0.2](17.2,0.0)
\psdots[linecolor=black, dotsize=0.2](17.2,0.40000007)
\psdots[linecolor=black, dotsize=0.2](17.6,0.40000007)
\psdots[linecolor=black, dotsize=0.2](18.0,0.40000007)
\psdots[linecolor=black, dotsize=0.2](18.0,0.0)
\psdots[linecolor=black, dotsize=0.2](17.6,0.0)
\psdots[linecolor=black, dotsize=0.2](17.6,-0.39999992)
\psdots[linecolor=black, dotsize=0.2](18.0,-0.39999992)
\psdots[linecolor=black, dotsize=0.2](18.4,-0.39999992)
\psdots[linecolor=black, dotsize=0.2](18.4,-0.79999995)
\psdots[linecolor=black, dotsize=0.2](18.0,-0.79999995)
\psdots[linecolor=black, dotsize=0.2](17.6,-0.79999995)
\psdots[linecolor=black, dotsize=0.2](17.6,-1.1999999)
\psdots[linecolor=black, dotsize=0.2](18.0,-1.1999999)
\psdots[linecolor=black, dotsize=0.2](18.4,-1.1999999)
\psdots[linecolor=black, dotsize=0.2](18.8,-1.1999999)
\psdots[linecolor=black, dotsize=0.2](18.8,-1.5999999)
\psdots[linecolor=black, dotsize=0.2](18.4,-1.5999999)
\psdots[linecolor=black, dotsize=0.2](18.0,-1.5999999)
\psdots[linecolor=black, dotsize=0.2](17.6,-1.5999999)
\psdots[linecolor=black, dotsize=0.2](17.6,-1.9999999)
\psdots[linecolor=black, dotsize=0.2](17.6,-2.3999999)
\psdots[linecolor=black, dotsize=0.2](18.0,-2.3999999)
\psdots[linecolor=black, dotsize=0.2](18.0,-1.9999999)
\psdots[linecolor=black, dotsize=0.2](18.4,-1.9999999)
\psdots[linecolor=black, dotsize=0.2](18.4,-2.3999999)
\psdots[linecolor=black, dotsize=0.2](18.8,-1.9999999)
\psdots[linecolor=black, dotsize=0.2](18.8,-2.3999999)
\psdots[linecolor=black, dotsize=0.2](19.2,-1.9999999)
\psdots[linecolor=black, dotsize=0.2](19.2,-2.3999999)
\psdots[linecolor=black, dotsize=0.2](17.2,-2.8)
\psdots[linecolor=black, dotsize=0.2](17.6,-2.8)
\psdots[linecolor=black, dotsize=0.2](17.6,-2.8)
\psdots[linecolor=black, dotsize=0.2](18.0,-2.8)
\psdots[linecolor=black, dotsize=0.2](18.4,-2.8)
\psdots[linecolor=black, dotsize=0.2](18.8,-2.8)
\psdots[linecolor=black, dotsize=0.2](19.2,-2.8)
\psdots[linecolor=black, dotsize=0.2](19.6,-2.8)
\psdots[linecolor=black, dotsize=0.2](7.6,2.8000002)
\psdots[linecolor=black, dotsize=0.2](8.0,2.8000002)
\psdots[linecolor=black, dotsize=0.2](8.0,2.4)
\psdots[linecolor=black, dotsize=0.2](8.4,2.8000002)
\psdots[linecolor=black, dotsize=0.2](8.4,2.4)
\psdots[linecolor=black, dotsize=0.2](8.4,2.0)
\psdots[linecolor=black, dotsize=0.2](8.0,2.0)
\psdots[linecolor=black, dotsize=0.2](8.4,1.6)
\psdots[linecolor=black, dotsize=0.2](8.4,1.2)
\psdots[linecolor=black, dotsize=0.2](8.8,2.8000002)
\psdots[linecolor=black, dotsize=0.2](8.8,2.4)
\psdots[linecolor=black, dotsize=0.2](8.8,2.0)
\psdots[linecolor=black, dotsize=0.2](8.8,1.6)
\psdots[linecolor=black, dotsize=0.2](8.8,1.2)
\psdots[linecolor=black, dotsize=0.2](8.8,0.8000001)
\psdots[linecolor=black, dotsize=0.2](8.8,0.40000007)
\psdots[linecolor=black, dotsize=0.2](9.2,2.8000002)
\psdots[linecolor=black, dotsize=0.2](9.2,2.4)
\psdots[linecolor=black, dotsize=0.2](9.2,2.0)
\psdots[linecolor=black, dotsize=0.2](9.2,1.6)
\psdots[linecolor=black, dotsize=0.2](9.2,1.6)
\psdots[linecolor=black, dotsize=0.2](9.2,1.2)
\psdots[linecolor=black, dotsize=0.2](9.2,0.8000001)
\psdots[linecolor=black, dotsize=0.2](9.2,0.8000001)
\psdots[linecolor=black, dotsize=0.2](9.2,0.40000007)
\psdots[linecolor=black, dotsize=0.2](9.2,0.40000007)
\psdots[linecolor=black, dotsize=0.2](9.2,0.0)
\psdots[linecolor=black, dotsize=0.2](9.2,0.0)
\psdots[linecolor=black, dotsize=0.2](9.2,-0.39999992)
\psdots[linecolor=black, dotsize=0.2](9.6,2.8000002)
\psdots[linecolor=black, dotsize=0.2](9.6,2.4)
\psdots[linecolor=black, dotsize=0.2](10.0,2.8000002)
\psdots[linecolor=black, dotsize=0.2](10.0,2.8000002)
\psdots[linecolor=black, dotsize=0.2](10.0,2.4)
\psdots[linecolor=black, dotsize=0.2](10.4,2.0)
\psdots[linecolor=black, dotsize=0.2](10.0,2.0)
\psdots[linecolor=black, dotsize=0.2](9.6,2.0)
\psdots[linecolor=black, dotsize=0.2](9.6,1.6)
\psdots[linecolor=black, dotsize=0.2](10.0,1.6)
\psdots[linecolor=black, dotsize=0.2](10.4,1.6)
\psdots[linecolor=black, dotsize=0.2](10.4,1.2)
\psdots[linecolor=black, dotsize=0.2](10.0,1.2)
\psdots[linecolor=black, dotsize=0.2](9.6,1.2)
\psdots[linecolor=black, dotsize=0.2](10.8,1.2)
\psdots[linecolor=black, dotsize=0.2](10.8,0.8000001)
\psdots[linecolor=black, dotsize=0.2](10.4,0.8000001)
\psdots[linecolor=black, dotsize=0.2](10.0,0.8000001)
\psdots[linecolor=black, dotsize=0.2](9.6,0.8000001)
\psdots[linecolor=black, dotsize=0.2](9.6,0.0)
\psdots[linecolor=black, dotsize=0.2](9.6,0.40000007)
\psdots[linecolor=black, dotsize=0.2](10.0,0.40000007)
\psdots[linecolor=black, dotsize=0.2](10.0,0.0)
\psdots[linecolor=black, dotsize=0.2](10.4,0.40000007)
\psdots[linecolor=black, dotsize=0.2](10.4,0.0)
\psdots[linecolor=black, dotsize=0.2](10.8,0.40000007)
\psdots[linecolor=black, dotsize=0.2](10.8,0.0)
\psdots[linecolor=black, dotsize=0.2](11.2,0.40000007)
\psdots[linecolor=black, dotsize=0.2](11.2,0.0)
\psdots[linecolor=black, dotsize=0.2](11.2,-0.39999992)
\psdots[linecolor=black, dotsize=0.2](11.2,-0.39999992)
\psdots[linecolor=black, dotsize=0.2](11.6,-0.39999992)
\psdots[linecolor=black, dotsize=0.2](10.4,-0.39999992)
\psdots[linecolor=black, dotsize=0.2](10.8,-0.39999992)
\psdots[linecolor=black, dotsize=0.2](10.0,-0.39999992)
\psdots[linecolor=black, dotsize=0.2](9.6,-0.39999992)
\psdots[linecolor=black, dotsize=0.2](9.6,-0.79999995)
\psdots[linecolor=black, dotsize=0.2](10.0,-0.79999995)
\psdots[linecolor=black, dotsize=0.2](10.0,-1.1999999)
\psdots[linecolor=black, dotsize=0.2](9.6,-1.1999999)
\psdots[linecolor=black, dotsize=0.2](10.0,-1.5999999)
\psdots[linecolor=black, dotsize=0.2](10.4,-1.5999999)
\psdots[linecolor=black, dotsize=0.2](10.4,-1.1999999)
\psdots[linecolor=black, dotsize=0.2](10.4,-0.79999995)
\psdots[linecolor=black, dotsize=0.2](10.8,-0.79999995)
\psdots[linecolor=black, dotsize=0.2](10.8,-1.1999999)
\psdots[linecolor=black, dotsize=0.2](10.8,-1.1999999)
\psdots[linecolor=black, dotsize=0.2](10.8,-1.5999999)
\psdots[linecolor=black, dotsize=0.2](11.2,-1.5999999)
\psdots[linecolor=black, dotsize=0.2](11.2,-1.1999999)
\psdots[linecolor=black, dotsize=0.2](11.2,-0.79999995)
\psdots[linecolor=black, dotsize=0.2](11.6,-0.79999995)
\psdots[linecolor=black, dotsize=0.2](11.6,-1.1999999)
\psdots[linecolor=black, dotsize=0.2](11.6,-1.5999999)
\psdots[linecolor=black, dotsize=0.2](12.0,-1.1999999)
\psdots[linecolor=black, dotsize=0.2](12.0,-1.5999999)
\psdots[linecolor=black, dotsize=0.2](12.0,-1.9999999)
\psdots[linecolor=black, dotsize=0.2](12.0,-1.9999999)
\psdots[linecolor=black, dotsize=0.2](12.4,-1.9999999)
\psdots[linecolor=black, dotsize=0.2](12.4,-2.3999999)
\psdots[linecolor=black, dotsize=0.2](12.0,-2.3999999)
\psdots[linecolor=black, dotsize=0.2](12.0,-2.8)
\psdots[linecolor=black, dotsize=0.2](12.0,-2.8)
\psdots[linecolor=black, dotsize=0.2](12.4,-2.8)
\psdots[linecolor=black, dotsize=0.2](12.8,-2.8)
\psdots[linecolor=black, dotsize=0.2](11.6,-2.8)
\psdots[linecolor=black, dotsize=0.2](11.6,-2.3999999)
\psdots[linecolor=black, dotsize=0.2](11.6,-1.9999999)
\psdots[linecolor=black, dotsize=0.2](11.2,-1.9999999)
\psdots[linecolor=black, dotsize=0.2](11.2,-2.3999999)
\psdots[linecolor=black, dotsize=0.2](11.2,-2.8)
\psdots[linecolor=black, dotsize=0.2](10.8,-2.8)
\psdots[linecolor=black, dotsize=0.2](10.8,-2.3999999)
\psdots[linecolor=black, dotsize=0.2](10.8,-1.9999999)
\psdots[linecolor=black, dotsize=0.2](10.4,-1.9999999)
\psdots[linecolor=black, dotsize=0.2](10.0,-1.9999999)
\psdots[linecolor=black, dotsize=0.2](10.4,-2.3999999)
\psdots[linecolor=black, dotsize=0.2](10.4,-2.8)
\psline[linecolor=black, linewidth=0.04](0.0,2.8000002)(20.0,2.8000002)
\psline[linecolor=black, linewidth=0.04](0.0,-2.8)(20.4,-2.8)
\psline[linecolor=black, linewidth=0.04](20.0,2.8000002)(20.4,2.8000002)
\psline[linecolor=colour0, linewidth=0.04](16.8,-2.3999999)(14.4,2.4)
\psline[linecolor=colour0, linewidth=0.04](14.8,2.4)(17.2,-2.3999999)
\psline[linecolor=colour0, linewidth=0.04](17.6,-2.3999999)(15.2,2.4)
\psline[linecolor=colour0, linewidth=0.04](18.0,-2.3999999)(15.6,2.4)
\psline[linecolor=colour0, linewidth=0.04](16.0,2.4)(18.4,-2.3999999)
\psline[linecolor=colour0, linewidth=0.04](18.8,-2.3999999)(16.4,2.4)
\psline[linecolor=colour0, linewidth=0.04](16.8,2.4)(19.2,-2.3999999)
\psline[linecolor=colour0, linewidth=0.04](19.6,-2.3999999)(17.2,2.4)
\psline[linecolor=colour1, linewidth=0.04](17.2,-2.8)(14.4,2.8000002)
\psline[linecolor=colour1, linewidth=0.04](14.8,2.8000002)(17.6,-2.8)
\psline[linecolor=colour1, linewidth=0.04](15.2,2.8000002)(18.0,-2.8)
\psline[linecolor=colour1, linewidth=0.04](18.4,-2.8)(15.6,2.8000002)
\psline[linecolor=colour1, linewidth=0.04](16.0,2.8000002)(18.8,-2.8)
\psline[linecolor=colour1, linewidth=0.04](19.2,-2.8)(16.4,2.8000002)
\psline[linecolor=colour1, linewidth=0.04](16.8,2.8000002)(19.6,-2.8)
\psdots[linecolor=black, dotsize=0.2](14.4,2.4)
\psdots[linecolor=black, dotsize=0.2](14.8,2.4)
\psdots[linecolor=black, dotsize=0.2](15.2,2.4)
\psdots[linecolor=black, dotsize=0.2](15.6,2.4)
\psdots[linecolor=black, dotsize=0.2](16.0,2.4)
\psdots[linecolor=black, dotsize=0.2](16.0,2.8000002)
\psdots[linecolor=black, dotsize=0.2](15.6,2.8000002)
\psdots[linecolor=black, dotsize=0.2](15.2,2.8000002)
\psdots[linecolor=black, dotsize=0.2](14.8,2.8000002)
\psdots[linecolor=black, dotsize=0.2](14.4,2.8000002)
\psdots[linecolor=black, dotsize=0.2](16.4,2.8000002)
\psdots[linecolor=black, dotsize=0.2](16.8,2.8000002)
\psdots[linecolor=black, dotsize=0.2](17.2,2.4)
\psdots[linecolor=black, dotsize=0.2](16.8,2.4)
\psdots[linecolor=black, dotsize=0.2](17.2,2.0)
\psdots[linecolor=black, dotsize=0.2](16.8,2.0)
\psdots[linecolor=black, dotsize=0.2](16.4,2.4)
\psdots[linecolor=black, dotsize=0.2](16.4,2.0)
\psdots[linecolor=black, dotsize=0.2](16.0,2.0)
\psdots[linecolor=black, dotsize=0.2](15.6,2.0)
\psdots[linecolor=black, dotsize=0.2](15.2,2.0)
\psdots[linecolor=black, dotsize=0.2](14.8,2.0)
\psdots[linecolor=black, dotsize=0.2](15.2,1.6)
\psdots[linecolor=black, dotsize=0.2](14.8,1.6)
\psdots[linecolor=black, dotsize=0.2](15.6,1.6)
\psdots[linecolor=black, dotsize=0.2](16.0,1.6)
\psdots[linecolor=black, dotsize=0.2](16.4,1.6)
\psdots[linecolor=black, dotsize=0.2](16.8,1.6)
\psdots[linecolor=black, dotsize=0.2](17.2,1.6)
\psdots[linecolor=black, dotsize=0.2](17.6,1.6)
\psdots[linecolor=black, dotsize=0.2](17.6,1.2)
\psdots[linecolor=black, dotsize=0.2](17.2,1.2)
\psdots[linecolor=black, dotsize=0.2](16.8,1.2)
\psdots[linecolor=black, dotsize=0.2](16.4,1.2)
\psdots[linecolor=black, dotsize=0.2](16.0,1.2)
\psdots[linecolor=black, dotsize=0.2](15.6,1.2)
\psdots[linecolor=black, dotsize=0.2](15.2,1.2)
\psdots[linecolor=black, dotsize=0.2](15.2,0.8000001)
\psdots[linecolor=black, dotsize=0.2](15.6,0.8000001)
\psdots[linecolor=black, dotsize=0.2](16.0,0.8000001)
\psdots[linecolor=black, dotsize=0.2](16.4,0.8000001)
\psdots[linecolor=black, dotsize=0.2](16.8,0.8000001)
\psdots[linecolor=black, dotsize=0.2](17.2,0.8000001)
\psdots[linecolor=black, dotsize=0.2](17.6,0.8000001)
\psdots[linecolor=black, dotsize=0.2](18.0,0.8000001)
\psdots[linecolor=black, dotsize=0.2](18.0,0.40000007)
\psdots[linecolor=black, dotsize=0.2](17.6,0.40000007)
\psdots[linecolor=black, dotsize=0.2](17.2,0.40000007)
\psdots[linecolor=black, dotsize=0.2](16.8,0.40000007)
\psdots[linecolor=black, dotsize=0.2](16.4,0.40000007)
\psdots[linecolor=black, dotsize=0.2](16.0,0.40000007)
\psdots[linecolor=black, dotsize=0.2](15.6,0.40000007)
\psdots[linecolor=black, dotsize=0.2](15.6,0.0)
\psdots[linecolor=black, dotsize=0.2](16.0,0.0)
\psdots[linecolor=black, dotsize=0.2](16.4,0.0)
\psdots[linecolor=black, dotsize=0.2](16.8,0.0)
\psdots[linecolor=black, dotsize=0.2](17.2,0.0)
\psdots[linecolor=black, dotsize=0.2](17.6,0.0)
\psdots[linecolor=black, dotsize=0.2](18.0,0.0)
\psdots[linecolor=black, dotsize=0.2](18.4,0.0)
\psdots[linecolor=black, dotsize=0.2](18.4,-0.39999992)
\psdots[linecolor=black, dotsize=0.2](18.0,-0.39999992)
\psdots[linecolor=black, dotsize=0.2](17.6,-0.39999992)
\psdots[linecolor=black, dotsize=0.2](17.2,-0.39999992)
\psdots[linecolor=black, dotsize=0.2](16.8,-0.39999992)
\psdots[linecolor=black, dotsize=0.2](16.4,-0.39999992)
\psdots[linecolor=black, dotsize=0.2](16.0,-0.39999992)
\psdots[linecolor=black, dotsize=0.2](16.0,-0.79999995)
\psdots[linecolor=black, dotsize=0.2](16.4,-0.79999995)
\psdots[linecolor=black, dotsize=0.2](16.8,-0.79999995)
\psdots[linecolor=black, dotsize=0.2](17.2,-0.79999995)
\psdots[linecolor=black, dotsize=0.2](17.6,-0.79999995)
\psdots[linecolor=black, dotsize=0.2](18.0,-0.79999995)
\psdots[linecolor=black, dotsize=0.2](18.4,-0.79999995)
\psdots[linecolor=black, dotsize=0.2](18.8,-0.79999995)
\psdots[linecolor=black, dotsize=0.2](18.8,-1.1999999)
\psdots[linecolor=black, dotsize=0.2](18.4,-1.1999999)
\psdots[linecolor=black, dotsize=0.2](18.0,-1.1999999)
\psdots[linecolor=black, dotsize=0.2](17.6,-1.1999999)
\psdots[linecolor=black, dotsize=0.2](17.2,-1.1999999)
\psdots[linecolor=black, dotsize=0.2](16.8,-1.1999999)
\psdots[linecolor=black, dotsize=0.2](16.4,-1.1999999)
\psdots[linecolor=black, dotsize=0.2](16.4,-1.5999999)
\psdots[linecolor=black, dotsize=0.2](16.8,-1.5999999)
\psdots[linecolor=black, dotsize=0.2](17.2,-1.5999999)
\psdots[linecolor=black, dotsize=0.2](17.6,-1.5999999)
\psdots[linecolor=black, dotsize=0.2](18.0,-1.5999999)
\psdots[linecolor=black, dotsize=0.2](18.4,-1.5999999)
\psdots[linecolor=black, dotsize=0.2](18.8,-1.5999999)
\psdots[linecolor=black, dotsize=0.2](19.2,-1.5999999)
\psdots[linecolor=black, dotsize=0.2](19.2,-1.9999999)
\psdots[linecolor=black, dotsize=0.2](18.8,-1.9999999)
\psdots[linecolor=black, dotsize=0.2](18.4,-1.9999999)
\psdots[linecolor=black, dotsize=0.2](18.0,-1.9999999)
\psdots[linecolor=black, dotsize=0.2](17.6,-1.9999999)
\psdots[linecolor=black, dotsize=0.2](17.2,-1.9999999)
\psdots[linecolor=black, dotsize=0.2](16.8,-1.9999999)
\psdots[linecolor=black, dotsize=0.2](16.8,-2.3999999)
\psdots[linecolor=black, dotsize=0.2](17.2,-2.3999999)
\psdots[linecolor=black, dotsize=0.2](17.6,-2.3999999)
\psdots[linecolor=black, dotsize=0.2](18.0,-2.3999999)
\psdots[linecolor=black, dotsize=0.2](18.4,-2.3999999)
\psdots[linecolor=black, dotsize=0.2](18.8,-2.3999999)
\psdots[linecolor=black, dotsize=0.2](19.2,-2.3999999)
\psdots[linecolor=black, dotsize=0.2](19.6,-2.3999999)
\psdots[linecolor=black, dotsize=0.2](19.6,-2.8)
\psdots[linecolor=black, dotsize=0.2](19.2,-2.8)
\psdots[linecolor=black, dotsize=0.2](18.8,-2.8)
\psdots[linecolor=black, dotsize=0.2](18.4,-2.8)
\psdots[linecolor=black, dotsize=0.2](18.0,-2.8)
\psdots[linecolor=black, dotsize=0.2](17.6,-2.8)
\psdots[linecolor=black, dotsize=0.2](17.2,-2.8)
\end{pspicture}
}
}

\def\oblique{% \usepackage[usenames,dvipsnames]{pstricks}
\psscalebox{1.0 1.0} % Change this value to rescale the drawing.
{
\begin{pspicture}[shift=-2.4](0,-2.4089375)(3.2357836,2.4089375)
\definecolor{colour0}{rgb}{1.0,0.0,0.2}
\psdots[linecolor=black, dotsize=0.2](0.81789184,1.2)
\psdots[linecolor=black, dotsize=0.2](0.81789184,1.6)
\psdots[linecolor=black, dotsize=0.2](1.2178918,1.6)
\psdots[linecolor=black, dotsize=0.2](1.2178918,1.2)
\psdots[linecolor=black, dotsize=0.2](1.2178918,0.8)
\psdots[linecolor=black, dotsize=0.2](1.2178918,0.4)
\psdots[linecolor=black, dotsize=0.2](1.2178918,0.0)
\psdots[linecolor=black, dotsize=0.2](1.6178918,1.2)
\psdots[linecolor=black, dotsize=0.2](1.6178918,0.8)
\psdots[linecolor=black, dotsize=0.2](1.6178918,0.4)
\psdots[linecolor=black, dotsize=0.2](1.6178918,0.0)
\psdots[linecolor=black, dotsize=0.2](1.6178918,-0.4)
\psdots[linecolor=black, dotsize=0.2](2.017892,-1.2)
\psdots[linecolor=black, dotsize=0.2](2.017892,-0.8)
\psdots[linecolor=black, dotsize=0.2](2.017892,-0.4)
\psdots[linecolor=black, dotsize=0.2](2.017892,0.0)
\psdots[linecolor=black, dotsize=0.2](2.4178917,-1.2)
\psdots[linecolor=black, dotsize=0.2](2.4178917,-1.6)
\psdots[linecolor=black, dotsize=0.2](2.8178918,-1.6)
\psline[linecolor=colour0, linewidth=0.04](0.017891847,2.4)(2.4178917,-2.4)
\psline[linecolor=colour0, linewidth=0.04](0.41789186,2.4)(2.8178918,-2.4)
\psline[linecolor=colour0, linewidth=0.04](0.81789184,2.4)(3.217892,-2.4)
\psline[linecolor=blue, linewidth=0.04](0.41789186,2.0)(2.4178917,-2.0)
\psline[linecolor=blue, linewidth=0.04](0.81789184,2.0)(2.8178918,-2.0)
\psline[linecolor=blue, linewidth=0.04](1.2178918,2.0)(3.217892,-2.0)
\psdots[linecolor=black, dotsize=0.2](0.81789184,1.2)
\psdots[linecolor=black, dotsize=0.2](0.81789184,1.6)
\psdots[linecolor=black, dotsize=0.2](1.2178918,1.2)
\psdots[linecolor=black, dotsize=0.2](1.2178918,1.6)
\psdots[linecolor=black, dotsize=0.2](1.6178918,1.2)
\psdots[linecolor=black, dotsize=0.2](1.6178918,0.8)
\psdots[linecolor=black, dotsize=0.2](1.2178918,0.8)
\psdots[linecolor=black, dotsize=0.2](1.6178918,0.4)
\psdots[linecolor=black, dotsize=0.2](1.2178918,0.4)
\psdots[linecolor=black, dotsize=0.2](1.2178918,0.0)
\psdots[linecolor=black, dotsize=0.2](1.6178918,-0.4)
\psdots[linecolor=black, dotsize=0.2](1.6178918,0.0)
\psdots[linecolor=black, dotsize=0.2](2.017892,-0.4)
\psdots[linecolor=black, dotsize=0.2](2.017892,0.0)
\psdots[linecolor=black, dotsize=0.2](2.4178917,-1.2)
\psdots[linecolor=black, dotsize=0.2](2.017892,-0.8)
\psdots[linecolor=black, dotsize=0.2](2.017892,-1.2)
\psdots[linecolor=black, dotsize=0.2](2.4178917,-1.6)
\psdots[linecolor=black, dotsize=0.2](2.8178918,-1.6)
\end{pspicture}
}
}

\def\DynkintypeC{% \usepackage[usenames,dvipsnames]{pstricks}
\psscalebox{1.0 1.0} % Change this value to rescale the drawing.
{
\begin{pspicture}[shift=-1](0,-1)(8.597115,0.09927719)
\rput(0.1,-.35){$\sigma_0$}
\rput(1.4,-.35){$\sigma_1$}
\rput(2.6,-.35){$\sigma_2$}
\rput(6,-.35){$\sigma_{n-2}$}
\rput(7.3,-.35){$\sigma_{n-1}$}
\rput(8.7,-.35){$\sigma_n=t_n $}
\psdots[linecolor=black, dotsize=0.2](0.09855774,0.0)
\psdots[linecolor=black, dotsize=0.2](1.2985578,0.0)
\psdots[linecolor=black, dotsize=0.2](2.4985578,0.0)
\psdots[linecolor=black, dotsize=0.2](6.098558,0.0)
\psdots[linecolor=black, dotsize=0.2](7.2985578,0.0)
\psdots[linecolor=black, dotsize=0.2](8.498558,0.0)
\psline[linecolor=black, linewidth=0.04](1.3099864,-0.023576595)(2.4528434,-0.012148024)
\psline[linecolor=black, linewidth=0.04](0.09855774,0.056423407)(1.3099864,0.07928055)
\psline[linecolor=black, linewidth=0.04](0.13284345,-0.057862308)(1.264272,-0.035005167)
\psline[linecolor=black, linewidth=0.04](2.544272,0.0)(3.1957006,0.01070912)
\psline[linecolor=black, linewidth=0.04](6.098558,0.0)(5.4471292,-0.012148024)
\psline[linecolor=black, linewidth=0.04](6.1328435,-0.023576595)(7.2871294,0.0)
\psline[linecolor=black, linewidth=0.04](7.321415,0.067851976)(8.487129,0.067851976)
\psline[linecolor=black, linewidth=0.04](7.321415,-0.069290884)(8.4757,-0.04643374)
\rput(4.3,0){$\cdots $}
\end{pspicture}
}
}

\def\realcountingB{% \usepackage[usenames,dvipsnames]{pstricks}
\psscalebox{1.0 1.0} % Change this value to rescale the drawing.
{
\begin{pspicture}[shift=-1.9](0,-1.8592789)(4.9971156,2.3)
\definecolor{colour0}{rgb}{1.0,0.0,0.2}
\rput(0.9,2.1){\scalebox{.8}{$P_3 $}}
\rput(1.3,2.1){\scalebox{.8}{$P_4 $}}
\rput(2.5,2.1){\scalebox{.8}{$P_7 $}}
\rput(-0.1,0.10218454){$ w_n^J$}
\rput(3.6,0.49242845){$w_n^I$}
\psdots[linecolor=black, dotsize=0.2](0.09855774,1.7607211)
\psdots[linecolor=black, dotsize=0.2](0.49855775,1.7607211)
\psdots[linecolor=black, dotsize=0.2](0.8985577,1.7607211)
\psdots[linecolor=black, dotsize=0.2](1.2985578,1.7607211)
\psdots[linecolor=black, dotsize=0.2](1.6985577,1.7607211)
\psdots[linecolor=black, dotsize=0.2](2.0985577,1.7607211)
\psdots[linecolor=black, dotsize=0.2](2.4985578,1.7607211)
\psdots[linecolor=black, dotsize=0.2](2.8985577,1.7607211)
\psdots[linecolor=black, dotsize=0.2](3.2985578,1.7607211)
\psdots[linecolor=black, dotsize=0.2](3.6985579,1.7607211)
\psdots[linecolor=black, dotsize=0.2](4.098558,1.7607211)
\psdots[linecolor=black, dotsize=0.2](4.4985576,1.7607211)
\psdots[linecolor=black, dotsize=0.2](4.8985577,1.7607211)
\psline[linecolor=black, linewidth=0.04](0.8985577,1.7607211)(1.2985578,1.7607211)(1.2985578,1.3607211)(1.2985578,0.96072114)(1.6985577,0.96072114)(1.6985577,0.16072112)(2.0985577,0.16072112)(2.0985577,-0.6392789)(2.4985578,-0.6392789)(2.4985578,-1.4392788)(2.8985577,-1.4392788)(2.8985577,-1.8392788)
\psdots[linecolor=black, dotsize=0.2](1.2985578,1.3607211)
\psdots[linecolor=black, dotsize=0.2](1.2985578,0.96072114)
\psdots[linecolor=black, dotsize=0.2](1.6985577,0.96072114)
\psdots[linecolor=black, dotsize=0.2](1.6985577,0.56072116)
\psdots[linecolor=black, dotsize=0.2](1.6985577,0.16072112)
\psdots[linecolor=black, dotsize=0.2](2.0985577,0.16072112)
\psdots[linecolor=black, dotsize=0.2](2.0985577,-0.23927887)
\psdots[linecolor=black, dotsize=0.2](2.0985577,-0.6392789)
\psdots[linecolor=black, dotsize=0.2](2.4985578,-0.6392789)
\psdots[linecolor=black, dotsize=0.2](2.4985578,-1.0392789)
\psdots[linecolor=black, dotsize=0.2](2.4985578,-1.4392788)
\psdots[linecolor=black, dotsize=0.2](2.8985577,-1.4392788)
\psline[linecolor=colour0, linewidth=0.04](1.3362936,1.7682683)(1.3362936,1.3607211)(1.7362936,1.3607211)(1.7362936,0.95317394)(1.7362936,0.5380796)(2.1331158,0.5380796)(2.1336455,-0.25331396)(2.5336454,-0.25331396)(2.5336454,-1.0392789)(2.940663,-1.0392789)(2.9336455,-1.8392788)(3.2985578,-1.8392788)(3.2985578,-1.8392788)
\psdots[linecolor=black, dotsize=0.2](1.2985578,1.7607211)
\psdots[linecolor=black, dotsize=0.2](1.2985578,1.3607211)
\psdots[linecolor=black, dotsize=0.2](1.2985578,1.3607211)
\psdots[linecolor=black, dotsize=0.2](1.6985577,0.96072114)
\psdots[linecolor=black, dotsize=0.2](1.6985577,0.56072116)
\psdots[linecolor=black, dotsize=0.2](1.6985577,0.16072112)
\psdots[linecolor=black, dotsize=0.2](2.0985577,0.16072112)
\psdots[linecolor=black, dotsize=0.2](2.0985577,-0.23927887)
\psdots[linecolor=black, dotsize=0.2](2.0985577,-0.6392789)
\psdots[linecolor=black, dotsize=0.2](2.4985578,-0.6392789)
\psdots[linecolor=black, dotsize=0.2](2.4985578,-1.0392789)
\psdots[linecolor=black, dotsize=0.2](2.4985578,-1.4392788)
\psdots[linecolor=black, dotsize=0.2](2.8985577,-1.4392788)
\psline[linecolor=black, linewidth=0.04, arrowsize=0.05291667cm 2.0,arrowlength=1.4,arrowinset=0.0]{->}(0.18636262,0.10218454)(1.3570943,0.10218454)
\psline[linecolor=black, linewidth=0.04, arrowsize=0.05291667cm 2.0,arrowlength=1.4,arrowinset=0.0]{->}(3.3083138,0.49242845)(2.332704,0.49242845)
\psdots[linecolor=black, dotsize=0.2](1.6985577,1.3607211)
\psdots[linecolor=black, dotsize=0.2](2.0985577,0.56072116)
\psdots[linecolor=black, dotsize=0.2](2.4985578,-0.23927887)
\psdots[linecolor=black, dotsize=0.2](2.8985577,-1.0392789)
\end{pspicture}
}
}

\def\realcounting{\psscalebox{1.0 1.0} % Change this value to rescale the drawing.
{
\begin{pspicture}[shift=-2](0,-2)(4.4985576,2)
\rput(4.8,0.2){$w_n^I$}
\rput(2.8,-1.6){$\ddots $}
\rput(0.1,1.7){\scalebox{.8}{$P_1 $}}
\rput(0.5,1.7){\scalebox{.8}{$P_2 $}}
\rput(0.9,1.7){\scalebox{.8}{$P_3 $}}
\rput(1.3,1.7){\scalebox{.8}{$P_4 $}}
\rput(1.7,1.7){\scalebox{.8}{$P_5 $}}
\rput(2.1,1.7){\scalebox{.8}{$P_6 $}}
\rput(2.5,1.7){\scalebox{.8}{$P_7 $}}

\psdots[linecolor=black, dotsize=0.2](0.09855774,1.3507211)
\psdots[linecolor=black, dotsize=0.2](0.49855775,1.3507211)
\psdots[linecolor=black, dotsize=0.2](0.8985577,1.3507211)
\psdots[linecolor=black, dotsize=0.2](1.2985578,1.3507211)
\psdots[linecolor=black, dotsize=0.2](1.6985577,1.3507211)
\psdots[linecolor=black, dotsize=0.2](2.0985577,1.3507211)
\psdots[linecolor=black, dotsize=0.2](2.4985578,1.3507211)
\psline[linecolor=black, linewidth=0.04](1.2985578,1.3507211)(1.2985578,0.95072114)(1.6985577,0.95072114)(1.6985577,0.5507211)(1.6985577,0.15072113)(2.0985577,0.15072113)(2.0985577,-0.24927887)(2.0985577,-0.6492789)(2.4985578,-0.6492789)(2.4985578,-1.0492789)(2.4985578,-1.0492789)
\psdots[linecolor=black, dotsize=0.2](1.2985578,0.95072114)
\psdots[linecolor=black, dotsize=0.2](1.6985577,0.95072114)
\psdots[linecolor=black, dotsize=0.2](1.6985577,0.5507211)
\psdots[linecolor=black, dotsize=0.2](1.6985577,0.15072113)
\psdots[linecolor=black, dotsize=0.2](2.0985577,0.15072113)
\psdots[linecolor=black, dotsize=0.2](2.0985577,-0.24927887)
\psdots[linecolor=black, dotsize=0.2](2.0985577,-0.6492789)
\psdots[linecolor=black, dotsize=0.2](2.4985578,-0.6492789)
\psdots[linecolor=black, dotsize=0.2](2.4985578,-1.0492789)
\psline[linecolor=black, linewidth=0.04](2.4985578,-1.0492789)(2.4985578,-1.4492788)
\psline[linecolor=black, linewidth=0.04, arrowsize=0.05291667cm 2.0,arrowlength=1.4,arrowinset=0.0]{->}(4.4985576,0.15072113)(2.4985578,0.15072113)
\end{pspicture}
}
}

\def\nodd{% \usepackage[usenames,dvipsnames]{pstricks}
\psscalebox{1.0 1.0} % Change this value to rescale the drawing.
{
\begin{pspicture}[shift=-2.4](0,-2.4038246)(5.1636367,2.4038246)
\rput(-0.4,2.3){$n$}
\rput(-0.6,1.9){$n-1$}
\psdots[linecolor=black, dotsize=0.2](4.8,-2.0947332)
\psdots[linecolor=black, dotsize=0.2](4.8,-1.6947333)
\psdots[linecolor=black, dotsize=0.2](4.4,-1.6947333)
\psdots[linecolor=black, dotsize=0.2](4.4,-1.2947333)
\psdots[linecolor=black, dotsize=0.2](4.4,-0.89473325)
\psdots[linecolor=black, dotsize=0.2](4.0,-0.89473325)
\psdots[linecolor=black, dotsize=0.2](4.0,-0.49473327)
\psdots[linecolor=black, dotsize=0.2](4.0,-0.094733275)
\psdots[linecolor=black, dotsize=0.2](3.6,-0.094733275)
\psdots[linecolor=black, dotsize=0.2](3.6,0.30526674)
\psdots[linecolor=black, dotsize=0.2](3.6,0.7052667)
\psdots[linecolor=black, dotsize=0.2](3.2,0.7052667)
\psdots[linecolor=black, dotsize=0.2](3.2,1.1052667)
\psdots[linecolor=black, dotsize=0.2](3.2,1.5052667)
\psdots[linecolor=black, dotsize=0.2](2.8,1.5052667)
\psline[linecolor=black, linewidth=0.04](2.8,1.5052667)(3.2,1.5052667)(3.2,0.7052667)(3.6,0.7052667)(3.6,-0.094733275)(4.0,-0.094733275)(4.0,-0.89473325)(4.4,-0.89473325)(4.4,-1.6947333)(4.8,-1.6947333)(4.8,-2.0947332)
\psdots[linecolor=black, dotsize=0.2](2.8,1.9052668)
\psdots[linecolor=black, dotsize=0.2](2.8,2.3052666)
\psdots[linecolor=black, dotsize=0.2](2.4,2.3052666)
\psline[linecolor=black, linewidth=0.04](2.4,2.3052666)(2.8,2.3052666)(2.8,1.5052667)
\psdots[linecolor=black, dotsize=0.2](2.0,2.3052666)
\psdots[linecolor=black, dotsize=0.2](1.6,2.3052666)
\psdots[linecolor=black, dotsize=0.2](1.2,2.3052666)
\psdots[linecolor=black, dotsize=0.2](2.4,1.9052668)
\psdots[linecolor=black, dotsize=0.2](2.0,1.9052668)
\psdots[linecolor=black, dotsize=0.2](1.6,1.9052668)
\psdots[linecolor=black, dotsize=0.2](2.0,1.5052667)
\psdots[linecolor=black, dotsize=0.2](2.4,1.5052667)
\psdots[linecolor=black, dotsize=0.2](2.4,1.1052667)
\psdots[linecolor=black, dotsize=0.2](2.8,1.1052667)
\psdots[linecolor=black, dotsize=0.2](2.8,0.7052667)
\psdots[linecolor=black, dotsize=0.2](3.2,0.30526674)
\psdots[linecolor=blue, dotsize=0.2](1.2,1.9052668)
\psdots[linecolor=black, dotstyle=o, dotsize=0.2, fillcolor=white](1.6,1.5052667)
\psdots[linecolor=black, dotstyle=o, dotsize=0.2, fillcolor=white](1.2,1.5052667)
\psdots[linecolor=black, dotstyle=o, dotsize=0.2, fillcolor=white](1.2,1.1052667)
\psdots[linecolor=black, dotstyle=o, dotsize=0.2, fillcolor=white](1.6,1.1052667)
\psdots[linecolor=black, dotstyle=o, dotsize=0.2, fillcolor=white](2.0,1.1052667)
\psdots[linecolor=black, dotstyle=o, dotsize=0.2, fillcolor=white](2.0,0.7052667)
\psdots[linecolor=black, dotstyle=o, dotsize=0.2, fillcolor=white](1.6,0.7052667)
\psdots[linecolor=black, dotstyle=o, dotsize=0.2, fillcolor=white](2.4,0.7052667)
\psdots[linecolor=black, dotstyle=o, dotsize=0.2, fillcolor=white](2.4,0.30526674)
\psdots[linecolor=black, dotstyle=o, dotsize=0.2, fillcolor=white](2.8,0.30526674)
\psdots[linecolor=black, dotstyle=o, dotsize=0.2, fillcolor=white](2.8,-0.094733275)
\psdots[linecolor=black, dotstyle=o, dotsize=0.2, fillcolor=white](3.2,-0.094733275)
\psdots[linecolor=black, dotstyle=o, dotsize=0.2, fillcolor=white](3.2,-0.094733275)
\psdots[linecolor=black, dotstyle=o, dotsize=0.2, fillcolor=white](3.6,-0.49473327)
\psdots[linecolor=black, dotstyle=o, dotsize=0.2, fillcolor=white](2.8,-0.49473327)
\psdots[linecolor=black, dotstyle=o, dotsize=0.2, fillcolor=white](3.2,-0.49473327)
\psdots[linecolor=black, dotstyle=o, dotsize=0.2, fillcolor=white](2.4,-0.49473327)
\psdots[linecolor=black, dotstyle=o, dotsize=0.2, fillcolor=white](2.4,-0.094733275)
\psdots[linecolor=black, dotstyle=o, dotsize=0.2, fillcolor=white](2.0,-0.094733275)
\psdots[linecolor=black, dotstyle=o, dotsize=0.2, fillcolor=white](2.0,0.30526674)
\psdots[linecolor=black, dotstyle=o, dotsize=0.2, fillcolor=white](1.6,0.30526674)
\psdots[linecolor=black, dotstyle=o, dotsize=0.2, fillcolor=white](1.2,0.7052667)
\psdots[linecolor=black, dotstyle=o, dotsize=0.2, fillcolor=white](1.2,0.30526674)
\psdots[linecolor=black, dotstyle=o, dotsize=0.2, fillcolor=white](1.2,-0.094733275)
\psdots[linecolor=black, dotstyle=o, dotsize=0.2, fillcolor=white](1.6,-0.094733275)
\psdots[linecolor=black, dotstyle=o, dotsize=0.2, fillcolor=white](1.6,-0.49473327)
\psdots[linecolor=black, dotstyle=o, dotsize=0.2, fillcolor=white](1.2,-0.49473327)
\psdots[linecolor=black, dotstyle=o, dotsize=0.2, fillcolor=white](2.0,-0.49473327)
\psdots[linecolor=black, dotstyle=o, dotsize=0.2, fillcolor=white](2.0,-0.89473325)
\psdots[linecolor=black, dotstyle=o, dotsize=0.2, fillcolor=white](1.6,-0.89473325)
\psdots[linecolor=black, dotstyle=o, dotsize=0.2, fillcolor=white](1.2,-0.89473325)
\psdots[linecolor=black, dotstyle=o, dotsize=0.2, fillcolor=white](1.2,-1.2947333)
\psdots[linecolor=black, dotstyle=o, dotsize=0.2, fillcolor=white](1.6,-1.2947333)
\psdots[linecolor=black, dotstyle=o, dotsize=0.2, fillcolor=white](1.6,-1.6947333)
\psdots[linecolor=black, dotstyle=o, dotsize=0.2, fillcolor=white](1.2,-1.6947333)
\psdots[linecolor=black, dotstyle=o, dotsize=0.2, fillcolor=white](1.2,-2.0947332)
\psdots[linecolor=black, dotstyle=o, dotsize=0.2, fillcolor=white](1.6,-2.0947332)
\psdots[linecolor=black, dotstyle=o, dotsize=0.2, fillcolor=white](2.0,-2.0947332)
\psdots[linecolor=black, dotstyle=o, dotsize=0.2, fillcolor=white](2.4,-2.0947332)
\psdots[linecolor=black, dotstyle=o, dotsize=0.2, fillcolor=white](2.4,-1.6947333)
\psdots[linecolor=black, dotstyle=o, dotsize=0.2, fillcolor=white](2.0,-1.6947333)
\psdots[linecolor=black, dotstyle=o, dotsize=0.2, fillcolor=white](2.0,-1.2947333)
\psdots[linecolor=black, dotstyle=o, dotsize=0.2, fillcolor=white](2.4,-1.2947333)
\psdots[linecolor=black, dotstyle=o, dotsize=0.2, fillcolor=white](2.4,-0.89473325)
\psdots[linecolor=black, dotstyle=o, dotsize=0.2, fillcolor=white](2.8,-0.89473325)
\psdots[linecolor=black, dotstyle=o, dotsize=0.2, fillcolor=white](2.8,-0.89473325)
\psdots[linecolor=black, dotstyle=o, dotsize=0.2, fillcolor=white](3.2,-0.89473325)
\psdots[linecolor=black, dotstyle=o, dotsize=0.2, fillcolor=white](3.6,-0.89473325)
\psdots[linecolor=black, dotstyle=o, dotsize=0.2, fillcolor=white](3.6,-0.89473325)
\psdots[linecolor=black, dotstyle=o, dotsize=0.2, fillcolor=white](4.0,-1.2947333)
\psdots[linecolor=black, dotstyle=o, dotsize=0.2, fillcolor=white](3.2,-1.2947333)
\psdots[linecolor=black, dotstyle=o, dotsize=0.2, fillcolor=white](3.6,-1.2947333)
\psdots[linecolor=black, dotstyle=o, dotsize=0.2, fillcolor=white](2.8,-1.2947333)
\psdots[linecolor=black, dotstyle=o, dotsize=0.2, fillcolor=white](2.8,-1.6947333)
\psdots[linecolor=black, dotstyle=o, dotsize=0.2, fillcolor=white](2.8,-1.6947333)
\psdots[linecolor=black, dotstyle=o, dotsize=0.2, fillcolor=white](3.2,-1.6947333)
\psdots[linecolor=black, dotstyle=o, dotsize=0.2, fillcolor=white](3.2,-2.0947332)
\psdots[linecolor=black, dotstyle=o, dotsize=0.2, fillcolor=white](2.8,-2.0947332)
\psdots[linecolor=black, dotstyle=o, dotsize=0.2, fillcolor=white](3.6,-1.6947333)
\psdots[linecolor=black, dotstyle=o, dotsize=0.2, fillcolor=white](3.6,-2.0947332)
\psdots[linecolor=black, dotstyle=o, dotsize=0.2, fillcolor=white](4.0,-1.6947333)
\psdots[linecolor=black, dotstyle=o, dotsize=0.2, fillcolor=white](4.0,-2.0947332)
\psdots[linecolor=black, dotstyle=o, dotsize=0.2, fillcolor=white](4.4,-2.0947332)
\psline[linecolor=black, linewidth=0.04, arrowsize=0.05291667cm 2.0,arrowlength=1.4,arrowinset=0.0]{->}(0.0,2.3052666)(0.8,2.3052666)
\psline[linecolor=black, linewidth=0.04, arrowsize=0.05291667cm 2.0,arrowlength=1.4,arrowinset=0.0]{->}(0.0,1.9052668)(0.8,1.9052668)
\end{pspicture}
}
}

\def\extending{% \usepackage[usenames,dvipsnames]{pstricks}
\psscalebox{1.0 1.0} % Change this value to rescale the drawing.
{
\begin{pspicture}[shift=-3.44](0,-3.4492788)(4.388889,3.4492788)
\definecolor{colour3}{rgb}{0.0,0.2,0.8}
\psdots[linecolor=black, dotsize=0.2](1.7888888,0.95072114)
\psdots[linecolor=black, dotsize=0.2](1.7888888,0.5507211)
\psdots[linecolor=black, dotsize=0.2](2.1888888,0.5507211)
\psdots[linecolor=black, dotsize=0.2](2.1888888,0.15072113)
\psdots[linecolor=black, dotsize=0.2](2.1888888,-0.24927887)
\psdots[linecolor=black, dotsize=0.2](2.588889,-0.24927887)
\psdots[linecolor=black, dotsize=0.2](2.588889,-0.6492789)
\psdots[linecolor=black, dotsize=0.2](2.588889,-1.0492789)
\psdots[linecolor=black, dotsize=0.2](2.9888887,-1.0492789)
\psdots[linecolor=black, dotsize=0.2](2.9888887,-1.4492788)
\psdots[linecolor=black, dotsize=0.2](2.9888887,-1.8492789)
\psdots[linecolor=black, dotsize=0.2](3.3888888,-1.8492789)
\psdots[linecolor=black, dotsize=0.2](3.3888888,-2.2492788)
\psdots[linecolor=black, dotsize=0.2](3.3888888,-2.6492789)
\psdots[linecolor=black, dotsize=0.2](3.788889,-2.6492789)
\psdots[linecolor=black, dotsize=0.2](3.788889,-3.049279)
\psdots[linecolor=colour3, dotsize=0.2](1.7888888,1.3507211)
\psdots[linecolor=colour3, dotsize=0.2](1.3888888,1.3507211)
\psdots[linecolor=colour3, dotsize=0.2](1.3888888,1.7507211)
\psdots[linecolor=colour3, dotsize=0.2](1.3888888,2.150721)
\psdots[linecolor=colour3, dotsize=0.2](0.98888886,2.150721)
\psdots[linecolor=colour3, dotsize=0.2](0.98888886,2.5507212)
\psdots[linecolor=colour3, dotsize=0.2](0.98888886,2.950721)
\psdots[linecolor=colour3, dotsize=0.2](0.5888889,2.950721)
\psdots[linecolor=colour3, dotsize=0.2](0.5888889,3.3507211)
\psdots[linecolor=black, dotsize=0.2](1.3888888,0.95072114)
\psdots[linecolor=black, dotsize=0.2](0.98888886,0.95072114)
\psdots[linecolor=black, dotsize=0.2](0.5888889,0.95072114)
\psline[linecolor=black, linewidth=0.04](1.7888888,0.95072114)(1.7888888,0.5507211)(2.1888888,0.5507211)(2.1888888,-0.24927887)(2.588889,-0.24927887)(2.588889,-1.0492789)(2.9888887,-1.0492789)(2.9888887,-1.8492789)(3.3888888,-1.8492789)(3.3888888,-2.6492789)(3.788889,-2.6492789)(3.788889,-3.049279)
\psline[linecolor=colour3, linewidth=0.04](0.5888889,3.3507211)(0.5888889,2.950721)(0.98888886,2.950721)(0.98888886,2.150721)(1.3888888,2.150721)(1.3888888,1.3507211)(1.7888888,1.3507211)(1.7888888,0.95072114)
\psdots[linecolor=black, dotsize=0.2](1.7888888,0.95072114)
\psdots[linecolor=black, dotsize=0.2](0.5888889,1.3507211)
\psdots[linecolor=black, dotsize=0.2](0.5888889,1.7507211)
\psdots[linecolor=black, dotsize=0.2](0.5888889,2.150721)
\psdots[linecolor=black, dotsize=0.2](0.5888889,2.5507212)
\psdots[linecolor=black, dotsize=0.2](0.98888886,0.5507211)
\psdots[linecolor=black, dotsize=0.2](1.3888888,0.5507211)
\psdots[linecolor=black, dotsize=0.2](1.7888888,0.15072113)
\psdots[linecolor=black, dotsize=0.2](1.3888888,0.15072113)
\psdots[linecolor=black, dotsize=0.2](1.7888888,-0.24927887)
\psdots[linecolor=black, dotsize=0.2](2.1888888,-0.6492789)
\psdots[linecolor=black, dotsize=0.2](0.98888886,1.3507211)
\psdots[linecolor=black, dotsize=0.2](0.98888886,1.7507211)
\psframe[linecolor=black, linewidth=0.04, dimen=outer](4.388889,1.1284989)(0.0,-3.4492788)
\psdots[linecolor=black, dotstyle=o, dotsize=0.2, fillcolor=white](0.5888889,0.5507211)
\psdots[linecolor=black, dotstyle=o, dotsize=0.2, fillcolor=white](0.5888889,0.15072113)
\psdots[linecolor=black, dotstyle=o, dotsize=0.2, fillcolor=white](0.98888886,0.15072113)
\psdots[linecolor=black, dotstyle=o, dotsize=0.2, fillcolor=white](0.98888886,-0.24927887)
\psdots[linecolor=black, dotstyle=o, dotsize=0.2, fillcolor=white](0.5888889,-0.24927887)
\psdots[linecolor=black, dotstyle=o, dotsize=0.2, fillcolor=white](0.5888889,-0.6492789)
\psdots[linecolor=black, dotstyle=o, dotsize=0.2, fillcolor=white](0.98888886,-0.6492789)
\psdots[linecolor=black, dotstyle=o, dotsize=0.2, fillcolor=white](1.3888888,-0.6492789)
\psdots[linecolor=black, dotstyle=o, dotsize=0.2, fillcolor=white](1.3888888,-0.24927887)
\psdots[linecolor=black, dotstyle=o, dotsize=0.2, fillcolor=white](1.7888888,-0.6492789)
\psdots[linecolor=black, dotstyle=o, dotsize=0.2, fillcolor=white](1.7888888,-1.0492789)
\psdots[linecolor=black, dotstyle=o, dotsize=0.2, fillcolor=white](2.1888888,-1.0492789)
\psdots[linecolor=black, dotstyle=o, dotsize=0.2, fillcolor=white](2.1888888,-1.0492789)
\psdots[linecolor=black, dotstyle=o, dotsize=0.2, fillcolor=white](2.1888888,-1.4492788)
\psdots[linecolor=black, dotstyle=o, dotsize=0.2, fillcolor=white](2.588889,-1.4492788)
\psdots[linecolor=black, dotstyle=o, dotsize=0.2, fillcolor=white](2.588889,-1.8492789)
\psdots[linecolor=black, dotstyle=o, dotsize=0.2, fillcolor=white](2.1888888,-1.8492789)
\psdots[linecolor=black, dotstyle=o, dotsize=0.2, fillcolor=white](2.9888887,-2.2492788)
\psdots[linecolor=black, dotstyle=o, dotsize=0.2, fillcolor=white](2.9888887,-2.2492788)
\psdots[linecolor=black, dotstyle=o, dotsize=0.2, fillcolor=white](2.9888887,-2.6492789)
\psdots[linecolor=black, dotstyle=o, dotsize=0.2, fillcolor=white](2.588889,-2.2492788)
\psdots[linecolor=black, dotstyle=o, dotsize=0.2, fillcolor=white](2.588889,-2.6492789)
\psdots[linecolor=black, dotstyle=o, dotsize=0.2, fillcolor=white](2.1888888,-2.2492788)
\psdots[linecolor=black, dotstyle=o, dotsize=0.2, fillcolor=white](2.1888888,-2.6492789)
\psdots[linecolor=black, dotstyle=o, dotsize=0.2, fillcolor=white](3.3888888,-3.049279)
\psdots[linecolor=black, dotstyle=o, dotsize=0.2, fillcolor=white](2.9888887,-3.049279)
\psdots[linecolor=black, dotstyle=o, dotsize=0.2, fillcolor=white](2.588889,-3.049279)
\psdots[linecolor=black, dotstyle=o, dotsize=0.2, fillcolor=white](2.1888888,-3.049279)
\psdots[linecolor=black, dotstyle=o, dotsize=0.2, fillcolor=white](1.7888888,-3.049279)
\psdots[linecolor=black, dotstyle=o, dotsize=0.2, fillcolor=white](1.7888888,-2.6492789)
\psdots[linecolor=black, dotstyle=o, dotsize=0.2, fillcolor=white](1.7888888,-2.2492788)
\psdots[linecolor=black, dotstyle=o, dotsize=0.2, fillcolor=white](1.7888888,-1.8492789)
\psdots[linecolor=black, dotstyle=o, dotsize=0.2, fillcolor=white](1.7888888,-1.4492788)
\psdots[linecolor=black, dotstyle=o, dotsize=0.2, fillcolor=white](1.3888888,-1.0492789)
\psdots[linecolor=black, dotstyle=o, dotsize=0.2, fillcolor=white](1.3888888,-1.4492788)
\psdots[linecolor=black, dotstyle=o, dotsize=0.2, fillcolor=white](1.3888888,-1.8492789)
\psdots[linecolor=black, dotstyle=o, dotsize=0.2, fillcolor=white](1.3888888,-2.2492788)
\psdots[linecolor=black, dotstyle=o, dotsize=0.2, fillcolor=white](1.3888888,-2.6492789)
\psdots[linecolor=black, dotstyle=o, dotsize=0.2, fillcolor=white](1.3888888,-3.049279)
\psdots[linecolor=black, dotstyle=o, dotsize=0.2, fillcolor=white](0.98888886,-3.049279)
\psdots[linecolor=black, dotstyle=o, dotsize=0.2, fillcolor=white](0.98888886,-2.6492789)
\psdots[linecolor=black, dotstyle=o, dotsize=0.2, fillcolor=white](0.98888886,-2.2492788)
\psdots[linecolor=black, dotstyle=o, dotsize=0.2, fillcolor=white](0.98888886,-1.8492789)
\psdots[linecolor=black, dotstyle=o, dotsize=0.2, fillcolor=white](0.98888886,-1.4492788)
\psdots[linecolor=black, dotstyle=o, dotsize=0.2, fillcolor=white](0.98888886,-1.0492789)
\psdots[linecolor=black, dotstyle=o, dotsize=0.2, fillcolor=white](0.5888889,-1.0492789)
\psdots[linecolor=black, dotstyle=o, dotsize=0.2, fillcolor=white](0.5888889,-1.4492788)
\psdots[linecolor=black, dotstyle=o, dotsize=0.2, fillcolor=white](0.5888889,-1.8492789)
\psdots[linecolor=black, dotstyle=o, dotsize=0.2, fillcolor=white](0.5888889,-1.8492789)
\psdots[linecolor=black, dotstyle=o, dotsize=0.2, fillcolor=white](0.5888889,-2.2492788)
\psdots[linecolor=black, dotstyle=o, dotsize=0.2, fillcolor=white](0.5888889,-2.6492789)
\psdots[linecolor=black, dotstyle=o, dotsize=0.2, fillcolor=white](0.5888889,-3.049279)
\end{pspicture}
}
}

\def\exbijection{% \usepackage[usenames,dvipsnames]{pstricks}
\psscalebox{.5 .5} % Change this value to rescale the drawing.
{
\begin{pspicture}[shift=-2.1](0,-2.0985577)(20.597115,2.0985577)
\definecolor{colour1}{rgb}{0.2,0.2,1.0}
\definecolor{colour2}{rgb}{0.0,0.2,1.0}
\psdots[linecolor=black, dotsize=0.2](0.09855774,2.0)
\psdots[linecolor=black, dotsize=0.2](0.09855774,1.6)
\psdots[linecolor=black, dotsize=0.2](0.49855775,1.6)
\psdots[linecolor=black, dotsize=0.2](0.49855775,1.2)
\psdots[linecolor=black, dotsize=0.2](0.49855775,1.2)
\psdots[linecolor=black, dotsize=0.2](0.49855775,0.8)
\psdots[linecolor=black, dotsize=0.2](0.8985577,0.8)
\psdots[linecolor=black, dotsize=0.2](0.8985577,0.4)
\psdots[linecolor=black, dotsize=0.2](0.8985577,0.0)
\psdots[linecolor=black, dotsize=0.2](1.2985578,0.0)
\psdots[linecolor=black, dotsize=0.2](1.2985578,-0.4)
\psdots[linecolor=black, dotsize=0.2](1.2985578,-0.8)
\psdots[linecolor=black, dotsize=0.2](1.6985577,-0.8)
\psdots[linecolor=black, dotsize=0.2](1.6985577,-1.2)
\psdots[linecolor=black, dotsize=0.2](1.6985577,-1.6)
\psdots[linecolor=black, dotsize=0.2](2.0985577,-1.6)
\psdots[linecolor=black, dotsize=0.2](2.0985577,-2.0)
\psline[linecolor=black, linewidth=0.04](0.09855774,2.0)(0.09855774,1.6)(0.49855775,1.6)(0.49855775,0.8)(0.8985577,0.8)(0.8985577,0.0)(1.2985578,0.0)(1.2985578,-0.8)(1.6985577,-0.8)(1.6985577,-1.6)(2.0985577,-1.6)(2.0985577,-2.0)
\psdots[linecolor=black, dotsize=0.2](0.49855775,2.0)
\psdots[linecolor=black, dotsize=0.2](0.8985577,2.0)
\psdots[linecolor=black, dotsize=0.2](1.2985578,2.0)
\psdots[linecolor=black, dotsize=0.2](0.8985577,1.6)
\psdots[linecolor=black, dotsize=0.2](1.2985578,1.6)
\psdots[linecolor=black, dotsize=0.2](0.8985577,1.2)
\psdots[linecolor=black, dotsize=0.2](1.2985578,1.2)
\psdots[linecolor=black, dotsize=0.2](1.2985578,0.8)
\psdots[linecolor=black, dotsize=0.2](1.2985578,0.4)
\psdots[linecolor=colour1, dotsize=0.2](1.6985577,1.2)
\psdots[linecolor=colour1, dotsize=0.2](1.6985577,0.8)
\psdots[linecolor=colour1, dotsize=0.2](1.6985577,0.4)
\psdots[linecolor=colour1, dotsize=0.2](1.6985577,0.0)
\psdots[linecolor=colour1, dotsize=0.2](1.6985577,-0.4)
\psdots[linecolor=colour1, dotsize=0.2](2.0985577,0.0)
\psdots[linecolor=colour1, dotsize=0.2](2.0985577,-0.4)
\psdots[linecolor=colour1, dotsize=0.2](2.0985577,-0.8)
\psdots[linecolor=colour1, dotsize=0.2](2.0985577,-1.2)
\psdots[linecolor=colour1, dotsize=0.2](2.4985578,-0.4)
\psdots[linecolor=colour1, dotsize=0.2](2.4985578,-0.8)
\psdots[linecolor=colour1, dotsize=0.2](2.4985578,-1.2)
\psdots[linecolor=colour1, dotsize=0.2](2.4985578,-1.6)
\psdots[linecolor=colour1, dotsize=0.2](2.4985578,-2.0)
\psdots[linecolor=colour1, dotsize=0.2](2.8985577,-2.0)
\psdots[linecolor=colour1, dotsize=0.2](2.8985577,-1.6)
\psdots[linecolor=colour1, dotsize=0.2](2.8985577,-1.2)
\psdots[linecolor=colour1, dotsize=0.2](3.2985578,-2.0)
\psdots[linecolor=colour1, dotsize=0.2](1.6985577,1.6)
\psline[linecolor=black, linewidth=0.04](2.8985577,-2.0)(2.0985577,-1.2)
\psline[linecolor=black, linewidth=0.04](3.2985578,-2.0)(1.6985577,-0.4)
\psline[linecolor=black, linewidth=0.04](2.8985577,-1.2)(1.2985578,0.4)
\psline[linecolor=black, linewidth=0.04](2.4985578,-0.4)(0.8985577,1.2)
\psline[linecolor=black, linewidth=0.04](1.6985577,0.8)(0.49855775,2.0)
\psline[linecolor=black, linewidth=0.04](1.6985577,1.2)(0.8985577,2.0)
\psline[linecolor=black, linewidth=0.04](1.6985577,1.6)(1.2985578,2.0)
\psline[linecolor=black, linewidth=0.04](2.8985577,-1.2)(3.6985579,-2.0)
\psline[linecolor=black, linewidth=0.04](2.4985578,-0.4)(4.098558,-2.0)
\psline[linecolor=black, linewidth=0.04](1.6985577,0.8)(4.4985576,-2.0)
\psline[linecolor=black, linewidth=0.04](1.6985577,1.2)(4.8985577,-2.0)
\psline[linecolor=black, linewidth=0.04](1.6985577,1.6)(5.2985578,-2.0)
\psline[linecolor=black, linewidth=0.04, arrowsize=0.05291667cm 2.0,arrowlength=1.4,arrowinset=0.0]{->}(4.8985577,0.4)(6.098558,0.4)
\psdots[linecolor=colour1, dotsize=0.2](1.6985577,1.6)
\psdots[linecolor=colour1, dotsize=0.2](1.6985577,1.2)
\psdots[linecolor=colour1, dotsize=0.2](1.6985577,0.8)
\psdots[linecolor=colour1, dotsize=0.2](1.6985577,0.4)
\psdots[linecolor=colour1, dotsize=0.2](1.6985577,0.0)
\psdots[linecolor=colour1, dotsize=0.2](2.0985577,0.0)
\psdots[linecolor=colour1, dotsize=0.2](2.4985578,-0.4)
\psdots[linecolor=colour1, dotsize=0.2](2.0985577,-0.4)
\psdots[linecolor=colour1, dotsize=0.2](1.6985577,-0.4)
\psdots[linecolor=colour1, dotsize=0.2](2.0985577,-0.8)
\psdots[linecolor=colour1, dotsize=0.2](2.4985578,-0.8)
\psdots[linecolor=colour1, dotsize=0.2](2.8985577,-1.2)
\psdots[linecolor=colour1, dotsize=0.2](2.4985578,-1.2)
\psdots[linecolor=colour1, dotsize=0.2](2.0985577,-1.2)
\psdots[linecolor=colour1, dotsize=0.2](2.4985578,-1.6)
\psdots[linecolor=colour1, dotsize=0.2](2.8985577,-2.0)
\psdots[linecolor=colour1, dotsize=0.2](2.8985577,-1.6)
\psdots[linecolor=colour1, dotsize=0.2](3.2985578,-2.0)
\psdots[linecolor=colour1, dotsize=0.2](7.2985578,-2.0)
\psline[linecolor=black, linewidth=0.04](7.698558,-1.2)(7.698558,-2.0)
\psline[linecolor=black, linewidth=0.04](8.098557,-0.4)(8.098557,-2.0)
\psline[linecolor=black, linewidth=0.04](8.498558,0.4)(8.498558,-2.0)
\psline[linecolor=black, linewidth=0.04](8.898558,1.2)(8.898558,-2.0)
\psline[linecolor=black, linewidth=0.04](9.298557,2.0)(9.298557,-2.0)
\psline[linecolor=black, linewidth=0.04](9.698558,2.0)(9.698558,-2.0)
\psline[linecolor=black, linewidth=0.04](10.098557,2.0)(10.098557,-2.0)
\psdots[linecolor=colour2, dotsize=0.2](7.698558,-2.0)
\psdots[linecolor=colour2, dotsize=0.2](7.698558,-1.6)
\psdots[linecolor=colour2, dotsize=0.2](7.698558,-1.2)
\psdots[linecolor=colour2, dotsize=0.2](8.098557,-2.0)
\psdots[linecolor=colour2, dotsize=0.2](8.098557,-1.6)
\psdots[linecolor=colour2, dotsize=0.2](8.098557,-1.2)
\psdots[linecolor=colour2, dotsize=0.2](8.098557,-0.8)
\psdots[linecolor=colour2, dotsize=0.2](8.098557,-0.4)
\psdots[linecolor=colour2, dotsize=0.2](8.498558,0.0)
\psdots[linecolor=colour2, dotsize=0.2](8.498558,-0.4)
\psdots[linecolor=colour2, dotsize=0.2](8.498558,-0.8)
\psdots[linecolor=colour2, dotsize=0.2](8.498558,-1.2)
\psdots[linecolor=colour2, dotsize=0.2](8.898558,0.4)
\psdots[linecolor=colour2, dotsize=0.2](8.898558,0.0)
\psdots[linecolor=colour2, dotsize=0.2](8.898558,-0.4)
\psdots[linecolor=colour2, dotsize=0.2](9.298557,0.8)
\psdots[linecolor=colour2, dotsize=0.2](9.698558,1.2)
\psdots[linecolor=colour2, dotsize=0.2](10.098557,1.6)
\psdots[linecolor=black, dotsize=0.2](8.498558,0.4)
\psdots[linecolor=black, dotsize=0.2](8.898558,1.2)
\psdots[linecolor=black, dotsize=0.2](8.898558,0.8)
\psdots[linecolor=black, dotsize=0.2](9.298557,1.2)
\psdots[linecolor=black, dotsize=0.2](9.298557,1.6)
\psdots[linecolor=black, dotsize=0.2](9.298557,2.0)
\psdots[linecolor=black, dotsize=0.2](9.698558,2.0)
\psdots[linecolor=black, dotsize=0.2](9.698558,1.6)
\psdots[linecolor=black, dotsize=0.2](10.098557,2.0)
\psline[linecolor=black, linewidth=0.04, arrowsize=0.05291667cm 2.0,arrowlength=1.4,arrowinset=0.0]{->}(10.898558,0.4)(12.098557,0.4)
\psline[linecolor=black, linewidth=0.04](12.898558,-2.0)(12.898558,2.0)
\psline[linecolor=black, linewidth=0.04](13.298557,-2.0)(13.298557,2.0)
\psline[linecolor=black, linewidth=0.04](13.698558,-2.0)(13.698558,2.0)
\psline[linecolor=black, linewidth=0.04](14.098557,1.2)(14.098557,-2.0)
\psline[linecolor=black, linewidth=0.04](14.498558,0.4)(14.498558,-2.0)
\psline[linecolor=black, linewidth=0.04](14.898558,-0.4)(14.898558,-2.0)
\psline[linecolor=black, linewidth=0.04](15.298557,-1.2)(15.298557,-2.0)
\psdots[linecolor=black, dotsize=0.2](12.898558,2.0)
\psdots[linecolor=black, dotsize=0.2](13.298557,2.0)
\psdots[linecolor=black, dotsize=0.2](13.298557,1.6)
\psdots[linecolor=black, dotsize=0.2](13.698558,2.0)
\psdots[linecolor=black, dotsize=0.2](13.698558,1.6)
\psdots[linecolor=black, dotsize=0.2](13.698558,1.2)
\psdots[linecolor=black, dotsize=0.2](14.098557,1.2)
\psdots[linecolor=black, dotsize=0.2](14.098557,0.8)
\psdots[linecolor=black, dotsize=0.2](14.498558,0.4)
\psdots[linecolor=blue, dotsize=0.2](12.898558,1.6)
\psdots[linecolor=blue, dotsize=0.2](13.298557,1.2)
\psdots[linecolor=blue, dotsize=0.2](13.698558,0.8)
\psdots[linecolor=blue, dotsize=0.2](14.098557,0.4)
\psdots[linecolor=blue, dotsize=0.2](14.098557,0.0)
\psdots[linecolor=blue, dotsize=0.2](14.098557,-0.4)
\psdots[linecolor=blue, dotsize=0.2](14.498558,0.0)
\psdots[linecolor=blue, dotsize=0.2](14.498558,-0.4)
\psdots[linecolor=blue, dotsize=0.2](14.498558,-0.4)
\psdots[linecolor=blue, dotsize=0.2](14.498558,-0.8)
\psdots[linecolor=blue, dotsize=0.2](14.498558,-1.2)
\psdots[linecolor=blue, dotsize=0.2](14.898558,-0.4)
\psdots[linecolor=blue, dotsize=0.2](14.898558,-0.8)
\psdots[linecolor=blue, dotsize=0.2](14.898558,-1.2)
\psdots[linecolor=blue, dotsize=0.2](14.898558,-1.6)
\psdots[linecolor=blue, dotsize=0.2](14.898558,-1.6)
\psdots[linecolor=blue, dotsize=0.2](14.898558,-2.0)
\psdots[linecolor=blue, dotsize=0.2](15.298557,-1.2)
\psdots[linecolor=blue, dotsize=0.2](15.298557,-1.6)
\psdots[linecolor=blue, dotsize=0.2](15.298557,-2.0)
\psdots[linecolor=blue, dotsize=0.2](15.698558,-2.0)
\psline[linecolor=black, linewidth=0.04, arrowsize=0.05291667cm 2.0,arrowlength=1.4,arrowinset=0.0]{->}(15.698558,0.4)(16.898558,0.4)
\psdots[linecolor=black, dotsize=0.2](17.298557,2.0)
\psdots[linecolor=black, dotsize=0.2](17.698557,2.0)
\psdots[linecolor=black, dotsize=0.2](18.098558,2.0)
\psdots[linecolor=black, dotsize=0.2](17.698557,1.6)
\psdots[linecolor=black, dotsize=0.2](18.098558,1.6)
\psdots[linecolor=black, dotsize=0.2](18.098558,1.2)
\psdots[linecolor=black, dotsize=0.2](18.498558,1.2)
\psdots[linecolor=black, dotsize=0.2](18.498558,0.8)
\psdots[linecolor=black, dotsize=0.2](18.898558,0.4)
\psdots[linecolor=blue, dotsize=0.2](17.298557,1.6)
\psdots[linecolor=blue, dotsize=0.2](17.698557,1.2)
\psdots[linecolor=blue, dotsize=0.2](18.098558,0.8)
\psdots[linecolor=blue, dotsize=0.2](18.498558,0.4)
\psdots[linecolor=blue, dotsize=0.2](18.498558,0.0)
\psdots[linecolor=blue, dotsize=0.2](18.498558,-0.4)
\psdots[linecolor=blue, dotsize=0.2](18.898558,0.0)
\psdots[linecolor=blue, dotsize=0.2](18.898558,-0.4)
\psdots[linecolor=blue, dotsize=0.2](18.898558,-0.8)
\psdots[linecolor=blue, dotsize=0.2](18.898558,-1.2)
\psdots[linecolor=blue, dotsize=0.2](19.298557,-0.4)
\psdots[linecolor=blue, dotsize=0.2](19.298557,-0.8)
\psdots[linecolor=blue, dotsize=0.2](19.298557,-1.2)
\psdots[linecolor=blue, dotsize=0.2](19.298557,-1.6)
\psdots[linecolor=blue, dotsize=0.2](19.298557,-2.0)
\psdots[linecolor=blue, dotsize=0.2](19.698557,-1.2)
\psdots[linecolor=blue, dotsize=0.2](19.698557,-1.6)
\psdots[linecolor=blue, dotsize=0.2](19.698557,-2.0)
\psdots[linecolor=blue, dotsize=0.2](20.098558,-2.0)
\psdots[linecolor=black, dotsize=0.2](18.498558,2.0)
\psdots[linecolor=black, dotsize=0.2](18.498558,1.6)
\psdots[linecolor=black, dotsize=0.2](18.898558,1.6)
\psdots[linecolor=black, dotsize=0.2](18.898558,1.2)
\psdots[linecolor=black, dotsize=0.2](18.898558,0.8)
\psdots[linecolor=black, dotsize=0.2](19.298557,0.8)
\psdots[linecolor=black, dotsize=0.2](19.298557,0.4)
\psdots[linecolor=black, dotsize=0.2](19.298557,0.0)
\psdots[linecolor=black, dotsize=0.2](19.698557,0.0)
\psdots[linecolor=black, dotsize=0.2](19.698557,-0.4)
\psdots[linecolor=black, dotsize=0.2](19.698557,-0.8)
\psdots[linecolor=black, dotsize=0.2](20.098558,-0.8)
\psdots[linecolor=black, dotsize=0.2](20.098558,-1.2)
\psdots[linecolor=black, dotsize=0.2](20.098558,-1.6)
\psdots[linecolor=black, dotsize=0.2](20.498558,-1.6)
\psdots[linecolor=black, dotsize=0.2](20.498558,-2.0)
\psline[linecolor=black, linewidth=0.04](18.498558,2.0)(18.498558,1.6)(18.898558,1.6)(18.898558,1.2)(18.898558,0.8)(19.298557,0.8)(19.298557,0.4)(19.298557,0.0)(19.698557,0.0)(19.698557,-0.4)(19.698557,-0.8)(20.098558,-0.8)(20.098558,-1.2)(20.098558,-1.6)(20.498558,-1.6)(20.498558,-2.0)
\end{pspicture}
}
}

\def\rotatingthediagonals{% \usepackage[usenames,dvipsnames]{pstricks}
\psscalebox{1.0 1.0} % Change this value to rescale the drawing.
{
\begin{pspicture}[shift=-1.8](0,-1.7006946)(12.199252,1.7006946)
\psdots[linecolor=black, dotstyle=o, dotsize=0.2, fillcolor=white](0.10069458,1.6)
\psdots[linecolor=black, dotsize=0.2](0.10069458,1.6)
\psdots[linecolor=black, dotsize=0.2](0.10069458,1.2)
\psdots[linecolor=black, dotsize=0.2](0.5006946,1.2)
\psdots[linecolor=black, dotsize=0.2](0.9006946,0.40000007)
\psdots[linecolor=black, dotsize=0.2](0.9006946,0.0)
\psdots[linecolor=black, dotsize=0.2](0.9006946,-0.39999992)
\psdots[linecolor=black, dotsize=0.2](1.3006946,-0.39999992)
\psdots[linecolor=black, dotsize=0.2](1.3006946,-0.79999995)
\psdots[linecolor=black, dotsize=0.2](1.3006946,-1.1999999)
\psdots[linecolor=black, dotsize=0.2](1.7006946,-1.1999999)
\psdots[linecolor=black, dotsize=0.2](1.7006946,-1.5999999)
\psdots[linecolor=black, dotsize=0.2](0.5006946,0.40000007)
\psdots[linecolor=black, dotsize=0.2](0.5006946,0.8000001)
\psdots[linecolor=black, dotstyle=o, dotsize=0.2, fillcolor=white](0.9006946,0.8000001)
\psdots[linecolor=black, dotstyle=o, dotsize=0.2, fillcolor=white](1.3006946,0.40000007)
\psdots[linecolor=black, dotstyle=o, dotsize=0.2, fillcolor=white](1.3006946,0.0)
\psdots[linecolor=black, dotstyle=o, dotsize=0.2, fillcolor=white](1.7006946,0.0)
\psdots[linecolor=black, dotstyle=o, dotsize=0.2, fillcolor=white](1.7006946,-0.39999992)
\psdots[linecolor=black, dotstyle=o, dotsize=0.2, fillcolor=white](1.7006946,-0.79999995)
\psdots[linecolor=black, dotstyle=o, dotsize=0.2, fillcolor=white](2.1006947,-0.39999992)
\psdots[linecolor=black, dotstyle=o, dotsize=0.2, fillcolor=white](2.1006947,-0.79999995)
\psdots[linecolor=black, dotstyle=o, dotsize=0.2, fillcolor=white](2.1006947,-1.1999999)
\psdots[linecolor=black, dotstyle=o, dotsize=0.2, fillcolor=white](2.1006947,-1.5999999)
\psdots[linecolor=black, dotstyle=o, dotsize=0.2, fillcolor=white](2.5006945,-0.79999995)
\psdots[linecolor=black, dotstyle=o, dotsize=0.2, fillcolor=white](2.5006945,-1.1999999)
\psdots[linecolor=black, dotstyle=o, dotsize=0.2, fillcolor=white](2.5006945,-1.5999999)
\psdots[linecolor=black, dotstyle=o, dotsize=0.2, fillcolor=white](2.9006946,-1.1999999)
\psdots[linecolor=black, dotstyle=o, dotsize=0.2, fillcolor=white](2.9006946,-1.5999999)
\psdots[linecolor=black, dotstyle=o, dotsize=0.2, fillcolor=white](3.3006945,-1.5999999)
\psline[linecolor=black, linewidth=0.02](0.9006946,0.8000001)(3.3006945,-1.5999999)
\psline[linecolor=black, linewidth=0.02](1.3006946,0.0)(2.9006946,-1.5999999)
\psline[linecolor=black, linewidth=0.02](1.7006946,-0.79999995)(2.5006945,-1.5999999)
\psline[linecolor=black, linewidth=0.02, arrowsize=0.05291667cm 2.0,arrowlength=1.4,arrowinset=0.0]{->}(3.7006946,0.0)(5.3006945,0.0)
\psdots[linecolor=black, dotstyle=o, dotsize=0.2, fillcolor=white](6.1006947,-1.5999999)
\psdots[linecolor=black, dotstyle=o, dotsize=0.2, fillcolor=white](6.5006948,-1.5999999)
\psdots[linecolor=black, dotstyle=o, dotsize=0.2, fillcolor=white](6.5006948,-1.1999999)
\psdots[linecolor=black, dotstyle=o, dotsize=0.2, fillcolor=white](6.5006948,-0.79999995)
\psdots[linecolor=black, dotstyle=o, dotsize=0.2, fillcolor=white](6.9006944,-1.5999999)
\psdots[linecolor=black, dotstyle=o, dotsize=0.2, fillcolor=white](6.9006944,-1.1999999)
\psdots[linecolor=black, dotstyle=o, dotsize=0.2, fillcolor=white](6.9006944,-0.79999995)
\psdots[linecolor=black, dotstyle=o, dotsize=0.2, fillcolor=white](6.9006944,-0.39999992)
\psdots[linecolor=black, dotstyle=o, dotsize=0.2, fillcolor=white](6.9006944,0.0)
\psdots[linecolor=black, dotstyle=o, dotsize=0.2, fillcolor=white](7.3006945,-1.5999999)
\psdots[linecolor=black, dotstyle=o, dotsize=0.2, fillcolor=white](7.3006945,-1.1999999)
\psdots[linecolor=black, dotstyle=o, dotsize=0.2, fillcolor=white](7.3006945,-0.79999995)
\psdots[linecolor=black, dotstyle=o, dotsize=0.2, fillcolor=white](7.3006945,-0.39999992)
\psdots[linecolor=black, dotstyle=o, dotsize=0.2, fillcolor=white](7.3006945,0.0)
\psdots[linecolor=black, dotstyle=o, dotsize=0.2, fillcolor=white](7.3006945,0.40000007)
\psdots[linecolor=black, dotstyle=o, dotsize=0.2, fillcolor=white](7.3006945,0.8000001)
\psdots[linecolor=black, dotsize=0.2](7.3006945,1.2)
\psdots[linecolor=black, dotsize=0.2](7.3006945,1.6)
\psdots[linecolor=black, dotsize=0.2](6.9006944,0.40000007)
\psdots[linecolor=black, dotsize=0.2](6.9006944,0.8000001)
\psdots[linecolor=black, dotsize=0.2](6.9006944,1.2)
\psdots[linecolor=black, dotsize=0.2](6.5006948,-0.39999992)
\psdots[linecolor=black, dotsize=0.2](6.5006948,0.0)
\psdots[linecolor=black, dotsize=0.2](6.5006948,0.40000007)
\psdots[linecolor=black, dotsize=0.2](6.1006947,-1.1999999)
\psdots[linecolor=black, dotsize=0.2](6.1006947,-0.79999995)
\psdots[linecolor=black, dotsize=0.2](6.1006947,-0.39999992)
\psdots[linecolor=black, dotsize=0.2](5.7006946,-1.5999999)
\psdots[linecolor=black, dotsize=0.2](5.7006946,-1.1999999)
\psline[linecolor=black, linewidth=0.02](0.9006946,0.8000001)(0.10069458,1.6)
\psline[linecolor=black, linewidth=0.02](1.3006946,0.0)(0.10069458,1.2)
\psline[linecolor=black, linewidth=0.02](1.7006946,-0.79999995)(0.5006946,0.40000007)
\psline[linecolor=black, linewidth=0.02](2.1006947,-1.5999999)(0.9006946,-0.39999992)
\psline[linecolor=black, linewidth=0.02](1.7006946,-1.5999999)(1.3006946,-1.1999999)
\psline[linecolor=black, linewidth=0.02](7.3006945,-1.5999999)(7.3006945,1.6)
\psline[linecolor=black, linewidth=0.02](6.9006944,1.2)(6.9006944,-1.5999999)
\psline[linecolor=black, linewidth=0.02](6.5006948,0.40000007)(6.5006948,-1.5999999)
\psline[linecolor=black, linewidth=0.02](6.1006947,-0.39999992)(6.1006947,-1.5999999)
\psline[linecolor=black, linewidth=0.02](5.7006946,-1.1999999)(5.7006946,-1.5999999)
\psline[linecolor=black, linewidth=0.02, arrowsize=0.05291667cm 2.0,arrowlength=1.4,arrowinset=0.0]{->}(8.500694,0.0)(10.100695,0.0)
\psdots[linecolor=black, dotsize=0.2](10.500694,1.6)
\psdots[linecolor=black, dotsize=0.2](10.500694,1.2)
\psdots[linecolor=black, dotsize=0.2](10.900695,1.2)
\psdots[linecolor=black, dotsize=0.2](10.900695,0.8000001)
\psdots[linecolor=black, dotsize=0.2](10.900695,0.40000007)
\psdots[linecolor=black, dotsize=0.2](11.300694,0.40000007)
\psdots[linecolor=black, dotsize=0.2](11.300694,0.0)
\psdots[linecolor=black, dotsize=0.2](11.300694,-0.39999992)
\psdots[linecolor=black, dotsize=0.2](11.700695,-0.39999992)
\psdots[linecolor=black, dotsize=0.2](11.700695,-0.79999995)
\psdots[linecolor=black, dotsize=0.2](11.700695,-1.1999999)
\psdots[linecolor=black, dotsize=0.2](12.100695,-1.1999999)
\psdots[linecolor=black, dotsize=0.2](12.100695,-1.5999999)
\psdots[linecolor=black, dotstyle=o, dotsize=0.2, fillcolor=white](10.500694,0.8000001)
\psdots[linecolor=black, dotstyle=o, dotsize=0.2, fillcolor=white](10.500694,0.40000007)
\psdots[linecolor=black, dotstyle=o, dotsize=0.2, fillcolor=white](10.500694,0.0)
\psdots[linecolor=black, dotstyle=o, dotsize=0.2, fillcolor=white](10.900695,0.0)
\psdots[linecolor=black, dotstyle=o, dotsize=0.2, fillcolor=white](10.900695,-0.39999992)
\psdots[linecolor=black, dotstyle=o, dotsize=0.2, fillcolor=white](10.500694,-0.39999992)
\psdots[linecolor=black, dotstyle=o, dotsize=0.2, fillcolor=white](10.500694,-0.79999995)
\psdots[linecolor=black, dotstyle=o, dotsize=0.2, fillcolor=white](10.900695,-0.79999995)
\psdots[linecolor=black, dotstyle=o, dotsize=0.2, fillcolor=white](11.300694,-0.79999995)
\psdots[linecolor=black, dotstyle=o, dotsize=0.2, fillcolor=white](11.300694,-1.1999999)
\psdots[linecolor=black, dotstyle=o, dotsize=0.2, fillcolor=white](10.900695,-1.1999999)
\psdots[linecolor=black, dotstyle=o, dotsize=0.2, fillcolor=white](10.500694,-1.1999999)
\psdots[linecolor=black, dotstyle=o, dotsize=0.2, fillcolor=white](10.500694,-1.1999999)
\psdots[linecolor=black, dotstyle=o, dotsize=0.2, fillcolor=white](10.500694,-1.5999999)
\psdots[linecolor=black, dotstyle=o, dotsize=0.2, fillcolor=white](10.900695,-1.5999999)
\psdots[linecolor=black, dotstyle=o, dotsize=0.2, fillcolor=white](11.300694,-1.5999999)
\psdots[linecolor=black, dotstyle=o, dotsize=0.2, fillcolor=white](11.700695,-1.5999999)
\psline[linecolor=black, linewidth=0.02](10.500694,1.6)(10.500694,-1.5999999)
\psline[linecolor=black, linewidth=0.02](10.900695,1.2)(10.900695,-1.5999999)
\psline[linecolor=black, linewidth=0.02](11.300694,0.40000007)(11.300694,-1.5999999)
\psline[linecolor=black, linewidth=0.02](11.700695,-0.39999992)(11.700695,-1.5999999)
\psline[linecolor=black, linewidth=0.02](12.100695,-1.1999999)(12.100695,-1.5999999)
\end{pspicture}
}
}

\def\heapinotherheap{% \usepackage[usenames,dvipsnames]{pstricks}
\psscalebox{1.0 1.0} % Change this value to rescale the drawing.
{
\begin{pspicture}[shift=-1.8](0,-1.7836804)(1.9229163,1.7836804)
\psdots[linecolor=black, dotstyle=o, dotsize=0.2, fillcolor=white](0.10069458,1.682986)
\psdots[linecolor=black, dotsize=0.2](0.10069458,1.682986)
\psdots[linecolor=black, dotsize=0.2](0.10069458,1.282986)
\psdots[linecolor=black, dotsize=0.2](0.5006946,1.282986)
\psdots[linecolor=black, dotsize=0.2](0.9006946,0.48298606)
\psdots[linecolor=black, dotsize=0.2](0.9006946,0.08298607)
\psdots[linecolor=black, dotsize=0.2](0.9006946,-0.31701392)
\psdots[linecolor=black, dotsize=0.2](1.3006946,-0.31701392)
\psdots[linecolor=black, dotsize=0.2](1.3006946,-0.71701396)
\psdots[linecolor=black, dotsize=0.2](1.3006946,-1.1170139)
\psdots[linecolor=black, dotsize=0.2](1.7006946,-1.1170139)
\psdots[linecolor=black, dotsize=0.2](1.7006946,-1.5170139)
\psdots[linecolor=black, dotstyle=o, dotsize=0.2, fillcolor=white](0.10069458,0.88298607)
\psdots[linecolor=black, dotstyle=o, dotsize=0.2, fillcolor=white](0.10069458,0.48298606)
\psdots[linecolor=black, dotstyle=o, dotsize=0.2, fillcolor=white](0.10069458,0.08298607)
\psdots[linecolor=black, dotstyle=o, dotsize=0.2, fillcolor=white](0.5006946,0.08298607)
\psdots[linecolor=black, dotstyle=o, dotsize=0.2, fillcolor=white](0.5006946,-0.31701392)
\psdots[linecolor=black, dotstyle=o, dotsize=0.2, fillcolor=white](0.10069458,-0.31701392)
\psdots[linecolor=black, dotstyle=o, dotsize=0.2, fillcolor=white](0.10069458,-0.71701396)
\psdots[linecolor=black, dotstyle=o, dotsize=0.2, fillcolor=white](0.5006946,-0.71701396)
\psdots[linecolor=black, dotstyle=o, dotsize=0.2, fillcolor=white](0.9006946,-0.71701396)
\psdots[linecolor=black, dotstyle=o, dotsize=0.2, fillcolor=white](0.9006946,-1.1170139)
\psdots[linecolor=black, dotstyle=o, dotsize=0.2, fillcolor=white](0.5006946,-1.1170139)
\psdots[linecolor=black, dotstyle=o, dotsize=0.2, fillcolor=white](0.10069458,-1.5170139)
\psdots[linecolor=black, dotstyle=o, dotsize=0.2, fillcolor=white](0.10069458,-1.1170139)
\psdots[linecolor=black, dotstyle=o, dotsize=0.2, fillcolor=white](0.5006946,-1.5170139)
\psdots[linecolor=black, dotstyle=o, dotsize=0.2, fillcolor=white](0.9006946,-1.5170139)
\psdots[linecolor=black, dotstyle=o, dotsize=0.2, fillcolor=white](1.3006946,-1.5170139)
\psframe[linecolor=black, linewidth=0.04, dimen=outer](1.9229168,1.1052083)(0.26736125,-1.7836806)
\psdots[linecolor=black, dotsize=0.2](0.5006946,0.48298606)
\psdots[linecolor=black, dotsize=0.2](0.5006946,0.88298607)
\end{pspicture}
}
}

\def\manyheaps{% \usepackage[usenames,dvipsnames]{pstricks}
\psscalebox{.8 .8} % Change this value to rescale the drawing.
{
\begin{pspicture}[shift=-1.7](0,-1.7006946)(13.001389,1.7006946)
\psdots[linecolor=black, dotsize=0.2](0.100694425,1.6)
\psdots[linecolor=black, dotsize=0.2](0.100694425,1.2)
\psdots[linecolor=black, dotsize=0.2](0.50069445,1.2)
\psdots[linecolor=black, dotsize=0.2](0.50069445,0.8000001)
\psdots[linecolor=black, dotsize=0.2](2.5006945,1.6)
\psdots[linecolor=black, dotsize=0.2](2.5006945,1.2)
\psdots[linecolor=black, dotsize=0.2](2.9006944,1.2)
\psdots[linecolor=black, dotsize=0.2](2.9006944,0.8000001)
\psdots[linecolor=black, dotsize=0.2](2.9006944,0.40000007)
\psdots[linecolor=black, dotsize=0.2](3.3006945,0.40000007)
\psdots[linecolor=black, dotsize=0.2](3.3006945,0.0)
\psdots[linecolor=black, dotstyle=o, dotsize=0.2, fillcolor=white](0.100694425,0.8000001)
\psdots[linecolor=black, dotstyle=o, dotsize=0.2, fillcolor=white](0.90069443,0.8000001)
\psdots[linecolor=black, dotstyle=o, dotsize=0.2, fillcolor=white](2.5006945,0.8000001)
\psdots[linecolor=black, dotstyle=o, dotsize=0.2, fillcolor=white](2.5006945,0.40000007)
\psdots[linecolor=black, dotstyle=o, dotsize=0.2, fillcolor=white](2.5006945,0.0)
\psdots[linecolor=black, dotstyle=o, dotsize=0.2, fillcolor=white](2.9006944,0.0)
\psdots[linecolor=black, dotstyle=o, dotsize=0.2, fillcolor=white](3.3006945,0.8000001)
\psdots[linecolor=black, dotstyle=o, dotsize=0.2, fillcolor=white](3.7006943,0.40000007)
\psdots[linecolor=black, dotstyle=o, dotsize=0.2, fillcolor=white](3.7006943,0.0)
\psdots[linecolor=black, dotstyle=o, dotsize=0.2, fillcolor=white](4.1006947,0.0)
\psdots[linecolor=black, dotstyle=o, dotsize=0.2, fillcolor=white](5.7006946,1.6)
\psdots[linecolor=black, dotsize=0.2](5.7006946,1.6)
\psdots[linecolor=black, dotsize=0.2](5.7006946,1.2)
\psdots[linecolor=black, dotsize=0.2](6.1006947,1.2)
\psdots[linecolor=black, dotsize=0.2](6.1006947,0.8000001)
\psdots[linecolor=black, dotsize=0.2](6.1006947,0.40000007)
\psdots[linecolor=black, dotsize=0.2](6.5006943,0.40000007)
\psdots[linecolor=black, dotsize=0.2](6.5006943,0.0)
\psdots[linecolor=black, dotsize=0.2](6.5006943,-0.39999992)
\psdots[linecolor=black, dotsize=0.2](6.9006944,-0.39999992)
\psdots[linecolor=black, dotsize=0.2](6.9006944,-0.79999995)
\psdots[linecolor=black, dotstyle=o, dotsize=0.2, fillcolor=white](5.7006946,0.8000001)
\psdots[linecolor=black, dotstyle=o, dotsize=0.2, fillcolor=white](5.7006946,0.40000007)
\psdots[linecolor=black, dotstyle=o, dotsize=0.2, fillcolor=white](5.7006946,0.0)
\psdots[linecolor=black, dotstyle=o, dotsize=0.2, fillcolor=white](6.1006947,0.0)
\psdots[linecolor=black, dotstyle=o, dotsize=0.2, fillcolor=white](6.1006947,-0.39999992)
\psdots[linecolor=black, dotstyle=o, dotsize=0.2, fillcolor=white](5.7006946,-0.39999992)
\psdots[linecolor=black, dotstyle=o, dotsize=0.2, fillcolor=white](5.7006946,-0.79999995)
\psdots[linecolor=black, dotstyle=o, dotsize=0.2, fillcolor=white](6.1006947,-0.79999995)
\psdots[linecolor=black, dotstyle=o, dotsize=0.2, fillcolor=white](6.1006947,-0.79999995)
\psdots[linecolor=black, dotstyle=o, dotsize=0.2, fillcolor=white](6.5006943,-0.79999995)
\psdots[linecolor=black, dotstyle=o, dotsize=0.2, fillcolor=white](6.5006943,0.8000001)
\psdots[linecolor=black, dotstyle=o, dotsize=0.2, fillcolor=white](6.9006944,0.40000007)
\psdots[linecolor=black, dotstyle=o, dotsize=0.2, fillcolor=white](6.9006944,0.0)
\psdots[linecolor=black, dotstyle=o, dotsize=0.2, fillcolor=white](7.3006945,0.0)
\psdots[linecolor=black, dotstyle=o, dotsize=0.2, fillcolor=white](7.3006945,-0.39999992)
\psdots[linecolor=black, dotstyle=o, dotsize=0.2, fillcolor=white](7.3006945,-0.79999995)
\psdots[linecolor=black, dotstyle=o, dotsize=0.2, fillcolor=white](7.7006946,-0.39999992)
\psdots[linecolor=black, dotstyle=o, dotsize=0.2, fillcolor=white](7.7006946,-0.79999995)
\psdots[linecolor=black, dotstyle=o, dotsize=0.2, fillcolor=white](8.100695,-0.79999995)
\psdots[linecolor=black, dotstyle=o, dotsize=0.2, fillcolor=white](9.700694,1.6)
\psdots[linecolor=black, dotsize=0.2](9.700694,1.6)
\psdots[linecolor=black, dotsize=0.2](9.700694,1.2)
\psdots[linecolor=black, dotsize=0.2](10.100695,1.2)
\psdots[linecolor=black, dotsize=0.2](10.100695,0.8000001)
\psdots[linecolor=black, dotsize=0.2](10.100695,0.40000007)
\psdots[linecolor=black, dotsize=0.2](10.500694,0.40000007)
\psdots[linecolor=black, dotsize=0.2](10.500694,0.0)
\psdots[linecolor=black, dotsize=0.2](10.500694,-0.39999992)
\psdots[linecolor=black, dotsize=0.2](10.900695,-0.39999992)
\psdots[linecolor=black, dotsize=0.2](10.900695,-0.79999995)
\psdots[linecolor=black, dotsize=0.2](10.900695,-1.1999999)
\psdots[linecolor=black, dotsize=0.2](11.300694,-1.1999999)
\psdots[linecolor=black, dotsize=0.2](11.300694,-1.5999999)
\psdots[linecolor=black, dotstyle=o, dotsize=0.2, fillcolor=white](9.700694,0.8000001)
\psdots[linecolor=black, dotstyle=o, dotsize=0.2, fillcolor=white](9.700694,0.40000007)
\psdots[linecolor=black, dotstyle=o, dotsize=0.2, fillcolor=white](9.700694,0.0)
\psdots[linecolor=black, dotstyle=o, dotsize=0.2, fillcolor=white](10.100695,0.0)
\psdots[linecolor=black, dotstyle=o, dotsize=0.2, fillcolor=white](10.100695,-0.39999992)
\psdots[linecolor=black, dotstyle=o, dotsize=0.2, fillcolor=white](9.700694,-0.39999992)
\psdots[linecolor=black, dotstyle=o, dotsize=0.2, fillcolor=white](9.700694,-0.79999995)
\psdots[linecolor=black, dotstyle=o, dotsize=0.2, fillcolor=white](10.100695,-0.79999995)
\psdots[linecolor=black, dotstyle=o, dotsize=0.2, fillcolor=white](10.500694,-0.79999995)
\psdots[linecolor=black, dotstyle=o, dotsize=0.2, fillcolor=white](10.500694,-1.1999999)
\psdots[linecolor=black, dotstyle=o, dotsize=0.2, fillcolor=white](10.100695,-1.1999999)
\psdots[linecolor=black, dotstyle=o, dotsize=0.2, fillcolor=white](9.700694,-1.5999999)
\psdots[linecolor=black, dotstyle=o, dotsize=0.2, fillcolor=white](9.700694,-1.1999999)
\psdots[linecolor=black, dotstyle=o, dotsize=0.2, fillcolor=white](10.100695,-1.5999999)
\psdots[linecolor=black, dotstyle=o, dotsize=0.2, fillcolor=white](10.500694,-1.5999999)
\psdots[linecolor=black, dotstyle=o, dotsize=0.2, fillcolor=white](10.900695,-1.5999999)
\psdots[linecolor=black, dotstyle=o, dotsize=0.2, fillcolor=white](10.500694,0.8000001)
\psdots[linecolor=black, dotstyle=o, dotsize=0.2, fillcolor=white](10.900695,0.40000007)
\psdots[linecolor=black, dotstyle=o, dotsize=0.2, fillcolor=white](10.900695,0.0)
\psdots[linecolor=black, dotstyle=o, dotsize=0.2, fillcolor=white](11.300694,0.0)
\psdots[linecolor=black, dotstyle=o, dotsize=0.2, fillcolor=white](11.300694,-0.39999992)
\psdots[linecolor=black, dotstyle=o, dotsize=0.2, fillcolor=white](11.700694,-0.39999992)
\psdots[linecolor=black, dotstyle=o, dotsize=0.2, fillcolor=white](11.700694,-0.79999995)
\psdots[linecolor=black, dotstyle=o, dotsize=0.2, fillcolor=white](11.300694,-0.79999995)
\psdots[linecolor=black, dotstyle=o, dotsize=0.2, fillcolor=white](11.700694,-1.1999999)
\psdots[linecolor=black, dotstyle=o, dotsize=0.2, fillcolor=white](11.700694,-1.5999999)
\psdots[linecolor=black, dotstyle=o, dotsize=0.2, fillcolor=white](12.100695,-0.79999995)
\psdots[linecolor=black, dotstyle=o, dotsize=0.2, fillcolor=white](12.100695,-1.1999999)
\psdots[linecolor=black, dotstyle=o, dotsize=0.2, fillcolor=white](12.100695,-1.5999999)
\psdots[linecolor=black, dotstyle=o, dotsize=0.2, fillcolor=white](12.500694,-1.1999999)
\psdots[linecolor=black, dotstyle=o, dotsize=0.2, fillcolor=white](12.500694,-1.5999999)
\psdots[linecolor=black, dotstyle=o, dotsize=0.2, fillcolor=white](12.900695,-1.5999999)
\end{pspicture}
}
}

\def\sbigerthann{% \usepackage[usenames,dvipsnames]{pstricks}
\psscalebox{1.0 1.0} % Change this value to rescale the drawing.
{
\begin{pspicture}[shift=-1.5](0,-1.4985577)(10.098557,1.4985577)
\rput(-.2,1.4){6}
\rput(-.2,-1){0}
\rput(5.8,1.4){7}
\rput(5.8,-1.4){0}
\definecolor{colour0}{rgb}{0.6,0.6,0.6}
\psdots[linecolor=colour0, dotsize=0.2](1.2,1.4000001)
\psdots[linecolor=colour0, dotsize=0.2](1.6,1.4000001)
\psdots[linecolor=colour0, dotsize=0.2](2.0,1.4000001)
\psdots[linecolor=colour0, dotsize=0.2](2.4,1.4000001)
\psdots[linecolor=colour0, dotsize=0.2](2.8,1.4000001)
\psdots[linecolor=colour0, dotsize=0.2](3.2,1.4000001)
\psdots[linecolor=colour0, dotsize=0.2](3.6,1.4000001)
\psdots[linecolor=colour0, dotsize=0.2](1.6,1.0)
\psdots[linecolor=colour0, dotsize=0.2](2.0,1.0)
\psdots[linecolor=colour0, dotsize=0.2](2.4,1.0)
\psdots[linecolor=colour0, dotsize=0.2](2.8,1.0)
\psdots[linecolor=colour0, dotsize=0.2](3.2,1.0)
\psdots[linecolor=colour0, dotsize=0.2](3.6,1.0)
\psdots[linecolor=colour0, dotsize=0.2](3.6,0.6)
\psdots[linecolor=colour0, dotsize=0.2](3.2,0.6)
\psdots[linecolor=colour0, dotsize=0.2](2.8,0.6)
\psdots[linecolor=colour0, dotsize=0.2](2.4,0.6)
\psdots[linecolor=colour0, dotsize=0.2](2.0,0.6)
\psdots[linecolor=colour0, dotsize=0.2](2.4,0.20000003)
\psdots[linecolor=colour0, dotsize=0.2](2.8,0.20000003)
\psdots[linecolor=colour0, dotsize=0.2](3.2,0.20000003)
\psdots[linecolor=colour0, dotsize=0.2](3.6,0.20000003)
\psdots[linecolor=colour0, dotsize=0.2](3.6,-0.19999996)
\psdots[linecolor=colour0, dotsize=0.2](3.2,-0.19999996)
\psdots[linecolor=colour0, dotsize=0.2](2.8,-0.19999996)
\psdots[linecolor=colour0, dotsize=0.2](3.2,-0.59999996)
\psdots[linecolor=colour0, dotsize=0.2](3.6,-0.59999996)
\psdots[linecolor=colour0, dotsize=0.2](3.6,-0.99999994)
\psdots[linecolor=colour0, dotsize=0.2](7.2,1.4000001)
\psdots[linecolor=colour0, dotsize=0.2](7.6,1.4000001)
\psdots[linecolor=colour0, dotsize=0.2](8.0,1.4000001)
\psdots[linecolor=colour0, dotsize=0.2](8.4,1.4000001)
\psdots[linecolor=colour0, dotsize=0.2](8.8,1.4000001)
\psdots[linecolor=colour0, dotsize=0.2](9.2,1.4000001)
\psdots[linecolor=colour0, dotsize=0.2](9.6,1.4000001)
\psdots[linecolor=colour0, dotsize=0.2](10.0,1.4000001)
\psdots[linecolor=colour0, dotsize=0.2](10.0,1.0)
\psdots[linecolor=colour0, dotsize=0.2](9.6,1.0)
\psdots[linecolor=colour0, dotsize=0.2](9.2,1.0)
\psdots[linecolor=colour0, dotsize=0.2](8.8,1.0)
\psdots[linecolor=colour0, dotsize=0.2](8.4,1.0)
\psdots[linecolor=colour0, dotsize=0.2](8.0,1.0)
\psdots[linecolor=colour0, dotsize=0.2](7.6,1.0)
\psdots[linecolor=colour0, dotsize=0.2](8.0,0.6)
\psdots[linecolor=colour0, dotsize=0.2](8.4,0.6)
\psdots[linecolor=colour0, dotsize=0.2](8.8,0.6)
\psdots[linecolor=colour0, dotsize=0.2](9.2,0.6)
\psdots[linecolor=colour0, dotsize=0.2](9.6,0.6)
\psdots[linecolor=colour0, dotsize=0.2](10.0,0.6)
\psdots[linecolor=colour0, dotsize=0.2](10.0,0.20000003)
\psdots[linecolor=colour0, dotsize=0.2](9.6,0.20000003)
\psdots[linecolor=colour0, dotsize=0.2](9.2,0.20000003)
\psdots[linecolor=colour0, dotsize=0.2](8.8,0.20000003)
\psdots[linecolor=colour0, dotsize=0.2](8.4,0.20000003)
\psdots[linecolor=colour0, dotsize=0.2](8.8,-0.19999996)
\psdots[linecolor=colour0, dotsize=0.2](9.2,-0.19999996)
\psdots[linecolor=colour0, dotsize=0.2](9.2,-0.19999996)
\psdots[linecolor=colour0, dotsize=0.2](9.6,-0.19999996)
\psdots[linecolor=colour0, dotsize=0.2](10.0,-0.19999996)
\psdots[linecolor=colour0, dotsize=0.2](10.0,-0.59999996)
\psdots[linecolor=colour0, dotsize=0.2](9.6,-0.59999996)
\psdots[linecolor=colour0, dotsize=0.2](9.2,-0.59999996)
\psdots[linecolor=colour0, dotsize=0.2](9.6,-0.99999994)
\psdots[linecolor=colour0, dotsize=0.2](10.0,-0.99999994)
\psdots[linecolor=colour0, dotsize=0.2](10.0,-1.4)
\psdots[linecolor=black, dotsize=0.2](3.6,-0.99999994)
\psdots[linecolor=black, dotsize=0.2](3.6,-0.59999996)
\psdots[linecolor=black, dotsize=0.2](3.2,-0.59999996)
\psdots[linecolor=black, dotsize=0.2](3.2,-0.19999996)
\psdots[linecolor=black, dotsize=0.2](3.2,0.20000003)
\psdots[linecolor=black, dotsize=0.2](2.8,0.20000003)
\psdots[linecolor=black, dotsize=0.2](2.8,0.6)
\psdots[linecolor=black, dotsize=0.2](2.8,1.0)
\psdots[linecolor=black, dotsize=0.2](2.4,1.0)
\psdots[linecolor=black, dotsize=0.2](2.4,1.4000001)
\psdots[linecolor=black, dotsize=0.2](10.0,-1.4)
\psdots[linecolor=black, dotsize=0.2](10.0,-0.99999994)
\psdots[linecolor=black, dotsize=0.2](9.6,-0.99999994)
\psdots[linecolor=black, dotsize=0.2](9.6,-0.59999996)
\psdots[linecolor=black, dotsize=0.2](9.6,-0.19999996)
\psdots[linecolor=black, dotsize=0.2](9.2,-0.19999996)
\psdots[linecolor=black, dotsize=0.2](9.2,0.20000003)
\psdots[linecolor=black, dotsize=0.2](9.2,0.6)
\psdots[linecolor=black, dotsize=0.2](8.8,0.6)
\psdots[linecolor=black, dotsize=0.2](8.8,1.0)
\psdots[linecolor=black, dotsize=0.2](8.8,1.4000001)
\psdots[linecolor=black, dotsize=0.2](8.4,1.4000001)
\psline[linecolor=black, linewidth=0.04, arrowsize=0.05291667cm 2.0,arrowlength=1.4,arrowinset=0.0]{->}(0.0,1.4000001)(0.8,1.4000001)
\psline[linecolor=black, linewidth=0.04, arrowsize=0.05291667cm 2.0,arrowlength=1.4,arrowinset=0.0]{->}(6.0,1.4000001)(6.8,1.4000001)
\psline[linecolor=black, linewidth=0.04, arrowsize=0.05291667cm 2.0,arrowlength=1.4,arrowinset=0.0]{->}(6.0,-1.4)(8.4,-1.4)
\psline[linecolor=black, linewidth=0.04, arrowsize=0.05291667cm 2.0,arrowlength=1.4,arrowinset=0.0]{->}(0.0,-0.99999994)(2.4,-0.99999994)
\end{pspicture}
}
}

\def\fourcases{% \usepackage[usenames,dvipsnames]{pstricks}
\psscalebox{.7 .7} % Change this value to rescale the drawing.
{
\begin{pspicture}[shift=-0.3](0,-0.29962605)(14.098557,1)
\rput(-0.4,-0.2){$n$}
\rput(-0.6,.3){$n+1$}
\rput(2,.7){Type I}
\rput(5.5,.7){Type II}
\rput(9.2,.7){Type III}
\rput(12.8,.7){Type IV}
\psdots[linecolor=black, dotsize=0.2](1.2,0.20106839)
\psdots[linecolor=black, dotsize=0.2](1.6,0.20106839)
\psdots[linecolor=black, dotsize=0.2](2.0,0.20106839)
\psdots[linecolor=black, dotsize=0.2](2.4,0.20106839)
\psdots[linecolor=black, dotsize=0.2](2.8,0.20106839)
\psdots[linecolor=black, dotsize=0.2](4.8,0.20106839)
\psdots[linecolor=black, dotsize=0.2](5.2,0.20106839)
\psdots[linecolor=black, dotsize=0.2](5.6,0.20106839)
\psdots[linecolor=black, dotsize=0.2](6.0,0.20106839)
\psdots[linecolor=black, dotsize=0.2](6.4,0.20106839)
\psdots[linecolor=black, dotsize=0.2](8.4,0.20106839)
\psdots[linecolor=black, dotsize=0.2](8.8,0.20106839)
\psdots[linecolor=black, dotsize=0.2](9.2,0.20106839)
\psdots[linecolor=black, dotsize=0.2](9.6,0.20106839)
\psdots[linecolor=black, dotsize=0.2](10.0,0.20106839)
\psdots[linecolor=black, dotsize=0.2](12.0,0.20106839)
\psdots[linecolor=black, dotsize=0.2](12.4,0.20106839)
\psdots[linecolor=black, dotsize=0.2](12.8,0.20106839)
\psdots[linecolor=black, dotsize=0.2](13.2,0.20106839)
\psdots[linecolor=black, dotsize=0.2](13.6,0.20106839)
\psdots[linecolor=black, dotsize=0.2](1.6,-0.19893162)
\psdots[linecolor=black, dotsize=0.2](2.0,-0.19893162)
\psdots[linecolor=black, dotsize=0.2](2.4,-0.19893162)
\psdots[linecolor=black, dotsize=0.2](2.8,-0.19893162)
\psdots[linecolor=black, dotsize=0.2](5.2,-0.19893162)
\psdots[linecolor=black, dotsize=0.2](5.6,-0.19893162)
\psdots[linecolor=black, dotsize=0.2](6.0,-0.19893162)
\psdots[linecolor=black, dotsize=0.2](6.4,-0.19893162)
\psdots[linecolor=black, dotsize=0.2](8.8,-0.19893162)
\psdots[linecolor=black, dotsize=0.2](9.2,-0.19893162)
\psdots[linecolor=black, dotsize=0.2](9.6,-0.19893162)
\psdots[linecolor=black, dotsize=0.2](10.0,-0.19893162)
\psdots[linecolor=black, dotsize=0.2](12.4,-0.19893162)
\psdots[linecolor=black, dotsize=0.2](12.8,-0.19893162)
\psdots[linecolor=black, dotsize=0.2](13.2,-0.19893162)
\psdots[linecolor=black, dotsize=0.2](13.6,-0.19893162)
\psdots[linecolor=black, dotstyle=o, dotsize=0.2, fillcolor=white](1.2,-0.19893162)
\psdots[linecolor=black, dotstyle=o, dotsize=0.2, fillcolor=white](3.2,-0.19893162)
\psdots[linecolor=black, dotstyle=o, dotsize=0.2, fillcolor=white](4.8,-0.19893162)
\psdots[linecolor=black, dotstyle=o, dotsize=0.2, fillcolor=white](10.4,-0.19893162)
\psdots[linecolor=black, dotsize=0.2](6.8,-0.19893162)
\psdots[linecolor=black, dotsize=0.2](8.4,-0.19893162)
\psdots[linecolor=black, dotsize=0.2](12.0,-0.19893162)
\psdots[linecolor=black, dotsize=0.2](14.0,-0.19893162)
\psline[linecolor=black, linewidth=0.04, arrowsize=0.05291667cm 2.0,arrowlength=1.4,arrowinset=0.0]{->}(0.0,0.20106839)(0.8,0.20106839)
\psline[linecolor=black, linewidth=0.04, arrowsize=0.05291667cm 2.0,arrowlength=1.4,arrowinset=0.0]{->}(0.0,-0.19893162)(0.8,-0.19893162)
\end{pspicture}
}
}

\def\Catalan{% \usepackage[usenames,dvipsnames]{pstricks}
\psscalebox{.5 .5} % Change this value to rescale the drawing.
{
\begin{pspicture}[shift=-4.1](0,-4.1006947)(8.201389,4.1006947)
\psdots[linecolor=black, dotstyle=o, dotsize=0.2, fillcolor=white](0.10069458,4.0)
\psdots[linecolor=black, dotstyle=o, dotsize=0.2, fillcolor=white](0.9006946,3.2)
\psdots[linecolor=black, dotstyle=o, dotsize=0.2, fillcolor=white](1.7006946,2.4)
\psdots[linecolor=black, dotstyle=o, dotsize=0.2, fillcolor=white](2.5006945,1.6)
\psdots[linecolor=black, dotstyle=o, dotsize=0.2, fillcolor=white](3.3006945,0.8000001)
\psdots[linecolor=black, dotstyle=o, dotsize=0.2, fillcolor=white](4.1006947,0.0)
\psdots[linecolor=black, dotstyle=o, dotsize=0.2, fillcolor=white](4.9006944,-0.79999995)
\psdots[linecolor=black, dotstyle=o, dotsize=0.2, fillcolor=white](3.3006945,-0.79999995)
\psdots[linecolor=black, dotstyle=o, dotsize=0.2, fillcolor=white](1.7006946,-0.79999995)
\psdots[linecolor=black, dotstyle=o, dotsize=0.2, fillcolor=white](0.10069458,-0.79999995)
\psdots[linecolor=black, dotstyle=o, dotsize=0.2, fillcolor=white](5.7006946,-1.5999999)
\psdots[linecolor=black, dotstyle=o, dotsize=0.2, fillcolor=white](6.5006948,-2.3999999)
\psdots[linecolor=black, dotstyle=o, dotsize=0.2, fillcolor=white](4.9006944,-2.3999999)
\psdots[linecolor=black, dotstyle=o, dotsize=0.2, fillcolor=white](3.3006945,-2.3999999)
\psdots[linecolor=black, dotstyle=o, dotsize=0.2, fillcolor=white](1.7006946,-2.3999999)
\psdots[linecolor=black, dotstyle=o, dotsize=0.2, fillcolor=white](0.10069458,-2.3999999)
\psdots[linecolor=black, dotstyle=o, dotsize=0.2, fillcolor=white](0.9006946,-1.5999999)
\psdots[linecolor=black, dotstyle=o, dotsize=0.2, fillcolor=white](2.5006945,-1.5999999)
\psdots[linecolor=black, dotstyle=o, dotsize=0.2, fillcolor=white](4.1006947,-1.5999999)
\psdots[linecolor=black, dotstyle=o, dotsize=0.2, fillcolor=white](7.3006945,-3.1999998)
\psdots[linecolor=black, dotstyle=o, dotsize=0.2, fillcolor=white](5.7006946,-3.1999998)
\psdots[linecolor=black, dotstyle=o, dotsize=0.2, fillcolor=white](4.1006947,-3.1999998)
\psdots[linecolor=black, dotstyle=o, dotsize=0.2, fillcolor=white](2.5006945,-3.1999998)
\psdots[linecolor=black, dotstyle=o, dotsize=0.2, fillcolor=white](0.9006946,-3.1999998)
\psdots[linecolor=black, dotstyle=o, dotsize=0.2, fillcolor=white](8.100695,-4.0)
\psdots[linecolor=black, dotstyle=o, dotsize=0.2, fillcolor=white](6.5006948,-4.0)
\psdots[linecolor=black, dotstyle=o, dotsize=0.2, fillcolor=white](4.9006944,-4.0)
\psdots[linecolor=black, dotstyle=o, dotsize=0.2, fillcolor=white](3.3006945,-4.0)
\psdots[linecolor=black, dotstyle=o, dotsize=0.2, fillcolor=white](1.7006946,-4.0)
\psdots[linecolor=black, dotstyle=o, dotsize=0.2, fillcolor=white](0.10069458,-4.0)
\psdots[linecolor=black, dotstyle=o, dotsize=0.2, fillcolor=white](2.5006945,0.0)
\psdots[linecolor=black, dotstyle=o, dotsize=0.2, fillcolor=white](0.9006946,0.0)
\psdots[linecolor=black, dotstyle=o, dotsize=0.2, fillcolor=white](1.7006946,0.8000001)
\psdots[linecolor=black, dotstyle=o, dotsize=0.2, fillcolor=white](0.10069458,0.8000001)
\psdots[linecolor=black, dotstyle=o, dotsize=0.2, fillcolor=white](0.9006946,1.6)
\psdots[linecolor=black, dotsize=0.2](0.10069458,4.0)
\psdots[linecolor=black, dotstyle=o, dotsize=0.2, fillcolor=white](0.10069458,2.4)
\psline[linecolor=black, linewidth=0.04, arrowsize=0.08cm 1.75,arrowlength=1.4,arrowinset=0.42]{->}(0.10069458,4.0)(0.9006946,3.2)
\psline[linecolor=black, linewidth=0.04, arrowsize=0.08cm 1.75,arrowlength=1.4,arrowinset=0.42]{->}(0.9006946,3.2)(0.10069458,2.4)
\psline[linecolor=black, linewidth=0.04, arrowsize=0.08cm 1.75,arrowlength=1.4,arrowinset=0.42]{->}(0.9006946,3.2)(1.7006946,2.4)
\psline[linecolor=black, linewidth=0.04, arrowsize=0.08cm 1.75,arrowlength=1.4,arrowinset=0.42]{->}(1.7006946,2.4)(2.5006945,1.6)
\psline[linecolor=black, linewidth=0.04, arrowsize=0.08cm 1.75,arrowlength=1.4,arrowinset=0.42]{->}(2.5006945,1.6)(3.3006945,0.8000001)
\psline[linecolor=black, linewidth=0.04, arrowsize=0.08cm 1.75,arrowlength=1.4,arrowinset=0.42]{->}(3.3006945,0.8000001)(4.1006947,0.0)
\psline[linecolor=black, linewidth=0.04, arrowsize=0.08cm 1.75,arrowlength=1.4,arrowinset=0.42]{->}(4.1006947,0.0)(4.9006944,-0.79999995)
\psline[linecolor=black, linewidth=0.04, arrowsize=0.08cm 1.75,arrowlength=1.4,arrowinset=0.42]{->}(4.9006944,-0.79999995)(5.7006946,-1.5999999)
\psline[linecolor=black, linewidth=0.04, arrowsize=0.08cm 1.75,arrowlength=1.4,arrowinset=0.42]{->}(5.7006946,-1.5999999)(6.5006948,-2.3999999)
\psline[linecolor=black, linewidth=0.04, arrowsize=0.08cm 1.75,arrowlength=1.4,arrowinset=0.42]{->}(6.5006948,-2.3999999)(7.3006945,-3.1999998)
\psline[linecolor=black, linewidth=0.04, arrowsize=0.08cm 1.75,arrowlength=1.4,arrowinset=0.42]{->}(7.3006945,-3.1999998)(8.100695,-4.0)
\psline[linecolor=black, linewidth=0.04, arrowsize=0.08cm 1.75,arrowlength=1.4,arrowinset=0.42]{->}(1.7006946,2.4)(0.9006946,1.6)
\psline[linecolor=black, linewidth=0.04, arrowsize=0.08cm 1.75,arrowlength=1.4,arrowinset=0.42]{->}(0.9006946,1.6)(0.10069458,0.8000001)
\psline[linecolor=black, linewidth=0.04, arrowsize=0.08cm 1.75,arrowlength=1.4,arrowinset=0.42]{->}(2.5006945,1.6)(1.7006946,0.8000001)
\psline[linecolor=black, linewidth=0.04, arrowsize=0.08cm 1.75,arrowlength=1.4,arrowinset=0.42]{->}(1.7006946,0.8000001)(0.9006946,0.0)
\psline[linecolor=black, linewidth=0.04, arrowsize=0.08cm 1.75,arrowlength=1.4,arrowinset=0.42]{->}(0.9006946,0.0)(0.10069458,-0.79999995)
\psline[linecolor=black, linewidth=0.04, arrowsize=0.08cm 1.75,arrowlength=1.4,arrowinset=0.42]{->}(3.3006945,0.8000001)(2.5006945,0.0)
\psline[linecolor=black, linewidth=0.04, arrowsize=0.08cm 1.75,arrowlength=1.4,arrowinset=0.42]{->}(2.5006945,0.0)(1.7006946,-0.79999995)
\psline[linecolor=black, linewidth=0.04, arrowsize=0.08cm 1.75,arrowlength=1.4,arrowinset=0.42]{->}(1.7006946,-0.79999995)(0.9006946,-1.5999999)
\psline[linecolor=black, linewidth=0.04, arrowsize=0.08cm 1.75,arrowlength=1.4,arrowinset=0.42]{->}(0.9006946,-1.5999999)(0.10069458,-2.3999999)
\psline[linecolor=black, linewidth=0.04, arrowsize=0.08cm 1.75,arrowlength=1.4,arrowinset=0.42]{->}(4.1006947,0.0)(3.3006945,-0.79999995)
\psline[linecolor=black, linewidth=0.04, arrowsize=0.08cm 1.75,arrowlength=1.4,arrowinset=0.42]{->}(3.3006945,-0.79999995)(2.5006945,-1.5999999)
\psline[linecolor=black, linewidth=0.04, arrowsize=0.08cm 1.75,arrowlength=1.4,arrowinset=0.42]{->}(2.5006945,-1.5999999)(1.7006946,-2.3999999)
\psline[linecolor=black, linewidth=0.04, arrowsize=0.08cm 1.75,arrowlength=1.4,arrowinset=0.42]{->}(1.7006946,-2.3999999)(0.9006946,-3.1999998)
\psline[linecolor=black, linewidth=0.04, arrowsize=0.08cm 1.75,arrowlength=1.4,arrowinset=0.42]{->}(0.9006946,-3.1999998)(0.10069458,-4.0)
\psline[linecolor=black, linewidth=0.04, arrowsize=0.08cm 1.75,arrowlength=1.4,arrowinset=0.42]{->}(4.9006944,-0.79999995)(4.1006947,-1.5999999)
\psline[linecolor=black, linewidth=0.04, arrowsize=0.08cm 1.75,arrowlength=1.4,arrowinset=0.42]{->}(4.1006947,-1.5999999)(3.3006945,-2.3999999)
\psline[linecolor=black, linewidth=0.04, arrowsize=0.08cm 1.75,arrowlength=1.4,arrowinset=0.42]{->}(3.3006945,-2.3999999)(2.5006945,-3.1999998)
\psline[linecolor=black, linewidth=0.04, arrowsize=0.08cm 1.75,arrowlength=1.4,arrowinset=0.42]{->}(2.5006945,-3.1999998)(1.7006946,-4.0)
\psline[linecolor=black, linewidth=0.04, arrowsize=0.08cm 1.75,arrowlength=1.4,arrowinset=0.42]{->}(5.7006946,-1.5999999)(4.9006944,-2.3999999)
\psline[linecolor=black, linewidth=0.04, arrowsize=0.08cm 1.75,arrowlength=1.4,arrowinset=0.42]{->}(4.9006944,-2.3999999)(4.1006947,-3.1999998)
\psline[linecolor=black, linewidth=0.04, arrowsize=0.08cm 1.75,arrowlength=1.4,arrowinset=0.42]{->}(4.1006947,-3.1999998)(3.3006945,-4.0)
\psline[linecolor=black, linewidth=0.04, arrowsize=0.08cm 1.75,arrowlength=1.4,arrowinset=0.42]{->}(6.5006948,-2.3999999)(5.7006946,-3.1999998)
\psline[linecolor=black, linewidth=0.04, arrowsize=0.08cm 1.75,arrowlength=1.4,arrowinset=0.42]{->}(7.3006945,-3.1999998)(6.5006948,-4.0)
\psline[linecolor=black, linewidth=0.04, arrowsize=0.08cm 1.75,arrowlength=1.4,arrowinset=0.42]{->}(5.7006946,-3.1999998)(4.9006944,-4.0)
\psline[linecolor=black, linewidth=0.04, arrowsize=0.08cm 1.75,arrowlength=1.4,arrowinset=0.42]{->}(0.10069458,2.4)(0.9006946,1.6)
\psline[linecolor=black, linewidth=0.04, arrowsize=0.08cm 1.75,arrowlength=1.4,arrowinset=0.42]{->}(0.9006946,1.6)(1.7006946,0.8000001)
\psline[linecolor=black, linewidth=0.04, arrowsize=0.08cm 1.75,arrowlength=1.4,arrowinset=0.42]{->}(1.7006946,0.8000001)(2.5006945,0.0)
\psline[linecolor=black, linewidth=0.04, arrowsize=0.08cm 1.75,arrowlength=1.4,arrowinset=0.42]{->}(2.5006945,0.0)(3.3006945,-0.79999995)
\psline[linecolor=black, linewidth=0.04, arrowsize=0.08cm 1.75,arrowlength=1.4,arrowinset=0.42]{->}(3.3006945,-0.79999995)(4.1006947,-1.5999999)
\psline[linecolor=black, linewidth=0.04, arrowsize=0.08cm 1.75,arrowlength=1.4,arrowinset=0.42]{->}(4.1006947,-1.5999999)(4.9006944,-2.3999999)
\psline[linecolor=black, linewidth=0.04, arrowsize=0.08cm 1.75,arrowlength=1.4,arrowinset=0.42]{->}(4.9006944,-2.3999999)(5.7006946,-3.1999998)
\psline[linecolor=black, linewidth=0.04, arrowsize=0.08cm 1.75,arrowlength=1.4,arrowinset=0.42]{->}(5.7006946,-3.1999998)(6.5006948,-4.0)
\psline[linecolor=black, linewidth=0.04, arrowsize=0.08cm 1.75,arrowlength=1.4,arrowinset=0.42]{->}(0.10069458,0.8000001)(0.9006946,0.0)
\psline[linecolor=black, linewidth=0.04, arrowsize=0.08cm 1.75,arrowlength=1.4,arrowinset=0.42]{->}(0.9006946,0.0)(1.7006946,-0.79999995)
\psline[linecolor=black, linewidth=0.04, arrowsize=0.08cm 1.75,arrowlength=1.4,arrowinset=0.42]{->}(1.7006946,-0.79999995)(2.5006945,-1.5999999)
\psline[linecolor=black, linewidth=0.04, arrowsize=0.08cm 1.75,arrowlength=1.4,arrowinset=0.42]{->}(2.5006945,-1.5999999)(3.3006945,-2.3999999)
\psline[linecolor=black, linewidth=0.04, arrowsize=0.08cm 1.75,arrowlength=1.4,arrowinset=0.42]{->}(3.3006945,-2.3999999)(4.1006947,-3.1999998)
\psline[linecolor=black, linewidth=0.04, arrowsize=0.08cm 1.75,arrowlength=1.4,arrowinset=0.42]{->}(4.1006947,-3.1999998)(4.9006944,-4.0)
\psline[linecolor=black, linewidth=0.04, arrowsize=0.08cm 1.75,arrowlength=1.4,arrowinset=0.42]{->}(0.10069458,-0.79999995)(0.9006946,-1.5999999)
\psline[linecolor=black, linewidth=0.04, arrowsize=0.08cm 1.75,arrowlength=1.4,arrowinset=0.42]{->}(0.9006946,-1.5999999)(1.7006946,-2.3999999)
\psline[linecolor=black, linewidth=0.04, arrowsize=0.08cm 1.75,arrowlength=1.4,arrowinset=0.42]{->}(1.7006946,-2.3999999)(2.5006945,-3.1999998)
\psline[linecolor=black, linewidth=0.04, arrowsize=0.08cm 1.75,arrowlength=1.4,arrowinset=0.42]{->}(2.5006945,-3.1999998)(3.3006945,-4.0)
\psline[linecolor=black, linewidth=0.04, arrowsize=0.08cm 1.75,arrowlength=1.4,arrowinset=0.42]{->}(0.10069458,-2.3999999)(0.9006946,-3.1999998)
\psline[linecolor=black, linewidth=0.04, arrowsize=0.08cm 1.75,arrowlength=1.4,arrowinset=0.42]{->}(0.9006946,-3.1999998)(1.7006946,-4.0)
\end{pspicture}
}
} 

\def\CatalanTwo{% \usepackage[usenames,dvipsnames]{pstricks}
\psscalebox{.5 .5} % Change this value to rescale the drawing.
{
\begin{pspicture}[shift=-3.3](0,-3.3006945)(6.601389,3.3006945)
\psdots[linecolor=black, dotstyle=o, dotsize=0.2, fillcolor=white](0.10069458,3.2)
\psdots[linecolor=black, dotstyle=o, dotsize=0.2, fillcolor=white](0.9006946,2.4)
\psdots[linecolor=black, dotstyle=o, dotsize=0.2, fillcolor=white](1.7006946,1.6)
\psdots[linecolor=black, dotstyle=o, dotsize=0.2, fillcolor=white](2.5006945,0.8000001)
\psdots[linecolor=black, dotstyle=o, dotsize=0.2, fillcolor=white](3.3006945,0.0)
\psdots[linecolor=black, dotstyle=o, dotsize=0.2, fillcolor=white](4.1006947,-0.79999995)
\psdots[linecolor=black, dotstyle=o, dotsize=0.2, fillcolor=white](4.9006944,-1.5999999)
\psdots[linecolor=black, dotstyle=o, dotsize=0.2, fillcolor=white](3.3006945,-1.5999999)
\psdots[linecolor=black, dotstyle=o, dotsize=0.2, fillcolor=white](1.7006946,-1.5999999)
\psdots[linecolor=black, dotstyle=o, dotsize=0.2, fillcolor=white](0.10069458,-1.5999999)
\psdots[linecolor=black, dotstyle=o, dotsize=0.2, fillcolor=white](5.7006946,-2.3999999)
\psdots[linecolor=black, dotstyle=o, dotsize=0.2, fillcolor=white](6.5006948,-3.1999998)
\psdots[linecolor=black, dotstyle=o, dotsize=0.2, fillcolor=white](4.9006944,-3.1999998)
\psdots[linecolor=black, dotstyle=o, dotsize=0.2, fillcolor=white](3.3006945,-3.1999998)
\psdots[linecolor=black, dotstyle=o, dotsize=0.2, fillcolor=white](1.7006946,-3.1999998)
\psdots[linecolor=black, dotstyle=o, dotsize=0.2, fillcolor=white](0.10069458,-3.1999998)
\psdots[linecolor=black, dotstyle=o, dotsize=0.2, fillcolor=white](0.9006946,-2.3999999)
\psdots[linecolor=black, dotstyle=o, dotsize=0.2, fillcolor=white](2.5006945,-2.3999999)
\psdots[linecolor=black, dotstyle=o, dotsize=0.2, fillcolor=white](4.1006947,-2.3999999)
\psdots[linecolor=black, dotstyle=o, dotsize=0.2, fillcolor=white](2.5006945,-0.79999995)
\psdots[linecolor=black, dotstyle=o, dotsize=0.2, fillcolor=white](0.9006946,-0.79999995)
\psdots[linecolor=black, dotstyle=o, dotsize=0.2, fillcolor=white](1.7006946,0.0)
\psdots[linecolor=black, dotstyle=o, dotsize=0.2, fillcolor=white](0.10069458,0.0)
\psdots[linecolor=black, dotstyle=o, dotsize=0.2, fillcolor=white](0.9006946,0.8000001)
\psdots[linecolor=black, dotsize=0.2](0.10069458,3.2)
\psdots[linecolor=black, dotstyle=o, dotsize=0.2, fillcolor=white](0.10069458,1.6)
\psline[linecolor=black, linewidth=0.04, arrowsize=0.08cm 1.75,arrowlength=1.4,arrowinset=0.42]{->}(0.10069458,3.2)(0.9006946,2.4)
\psline[linecolor=black, linewidth=0.04, arrowsize=0.08cm 1.75,arrowlength=1.4,arrowinset=0.42]{->}(0.9006946,2.4)(0.10069458,1.6)
\psline[linecolor=black, linewidth=0.04, arrowsize=0.08cm 1.75,arrowlength=1.4,arrowinset=0.42]{->}(0.9006946,2.4)(1.7006946,1.6)
\psline[linecolor=black, linewidth=0.04, arrowsize=0.08cm 1.75,arrowlength=1.4,arrowinset=0.42]{->}(1.7006946,1.6)(2.5006945,0.8000001)
\psline[linecolor=black, linewidth=0.04, arrowsize=0.08cm 1.75,arrowlength=1.4,arrowinset=0.42]{->}(2.5006945,0.8000001)(3.3006945,0.0)
\psline[linecolor=black, linewidth=0.04, arrowsize=0.08cm 1.75,arrowlength=1.4,arrowinset=0.42]{->}(3.3006945,0.0)(4.1006947,-0.79999995)
\psline[linecolor=black, linewidth=0.04, arrowsize=0.08cm 1.75,arrowlength=1.4,arrowinset=0.42]{->}(4.1006947,-0.79999995)(4.9006944,-1.5999999)
\psline[linecolor=black, linewidth=0.04, arrowsize=0.08cm 1.75,arrowlength=1.4,arrowinset=0.42]{->}(4.9006944,-1.5999999)(5.7006946,-2.3999999)
\psline[linecolor=black, linewidth=0.04, arrowsize=0.08cm 1.75,arrowlength=1.4,arrowinset=0.42]{->}(5.7006946,-2.3999999)(6.5006948,-3.1999998)
\psline[linecolor=black, linewidth=0.04, arrowsize=0.08cm 1.75,arrowlength=1.4,arrowinset=0.42]{->}(1.7006946,1.6)(0.9006946,0.8000001)
\psline[linecolor=black, linewidth=0.04, arrowsize=0.08cm 1.75,arrowlength=1.4,arrowinset=0.42]{->}(0.9006946,0.8000001)(0.10069458,0.0)
\psline[linecolor=black, linewidth=0.04, arrowsize=0.08cm 1.75,arrowlength=1.4,arrowinset=0.42]{->}(2.5006945,0.8000001)(1.7006946,0.0)
\psline[linecolor=black, linewidth=0.04, arrowsize=0.08cm 1.75,arrowlength=1.4,arrowinset=0.42]{->}(1.7006946,0.0)(0.9006946,-0.79999995)
\psline[linecolor=black, linewidth=0.04, arrowsize=0.08cm 1.75,arrowlength=1.4,arrowinset=0.42]{->}(0.9006946,-0.79999995)(0.10069458,-1.5999999)
\psline[linecolor=black, linewidth=0.04, arrowsize=0.08cm 1.75,arrowlength=1.4,arrowinset=0.42]{->}(3.3006945,0.0)(2.5006945,-0.79999995)
\psline[linecolor=black, linewidth=0.04, arrowsize=0.08cm 1.75,arrowlength=1.4,arrowinset=0.42]{->}(2.5006945,-0.79999995)(1.7006946,-1.5999999)
\psline[linecolor=black, linewidth=0.04, arrowsize=0.08cm 1.75,arrowlength=1.4,arrowinset=0.42]{->}(1.7006946,-1.5999999)(0.9006946,-2.3999999)
\psline[linecolor=black, linewidth=0.04, arrowsize=0.08cm 1.75,arrowlength=1.4,arrowinset=0.42]{->}(0.9006946,-2.3999999)(0.10069458,-3.1999998)
\psline[linecolor=black, linewidth=0.04, arrowsize=0.08cm 1.75,arrowlength=1.4,arrowinset=0.42]{->}(4.1006947,-0.79999995)(3.3006945,-1.5999999)
\psline[linecolor=black, linewidth=0.04, arrowsize=0.08cm 1.75,arrowlength=1.4,arrowinset=0.42]{->}(3.3006945,-1.5999999)(2.5006945,-2.3999999)
\psline[linecolor=black, linewidth=0.04, arrowsize=0.08cm 1.75,arrowlength=1.4,arrowinset=0.42]{->}(2.5006945,-2.3999999)(1.7006946,-3.1999998)
\psline[linecolor=black, linewidth=0.04, arrowsize=0.08cm 1.75,arrowlength=1.4,arrowinset=0.42]{->}(4.9006944,-1.5999999)(4.1006947,-2.3999999)
\psline[linecolor=black, linewidth=0.04, arrowsize=0.08cm 1.75,arrowlength=1.4,arrowinset=0.42]{->}(4.1006947,-2.3999999)(3.3006945,-3.1999998)
\psline[linecolor=black, linewidth=0.04, arrowsize=0.08cm 1.75,arrowlength=1.4,arrowinset=0.42]{->}(5.7006946,-2.3999999)(4.9006944,-3.1999998)
\psline[linecolor=black, linewidth=0.04, arrowsize=0.08cm 1.75,arrowlength=1.4,arrowinset=0.42]{->}(0.10069458,1.6)(0.9006946,0.8000001)
\psline[linecolor=black, linewidth=0.04, arrowsize=0.08cm 1.75,arrowlength=1.4,arrowinset=0.42]{->}(0.9006946,0.8000001)(1.7006946,0.0)
\psline[linecolor=black, linewidth=0.04, arrowsize=0.08cm 1.75,arrowlength=1.4,arrowinset=0.42]{->}(1.7006946,0.0)(2.5006945,-0.79999995)
\psline[linecolor=black, linewidth=0.04, arrowsize=0.08cm 1.75,arrowlength=1.4,arrowinset=0.42]{->}(2.5006945,-0.79999995)(3.3006945,-1.5999999)
\psline[linecolor=black, linewidth=0.04, arrowsize=0.08cm 1.75,arrowlength=1.4,arrowinset=0.42]{->}(3.3006945,-1.5999999)(4.1006947,-2.3999999)
\psline[linecolor=black, linewidth=0.04, arrowsize=0.08cm 1.75,arrowlength=1.4,arrowinset=0.42]{->}(4.1006947,-2.3999999)(4.9006944,-3.1999998)
\psline[linecolor=black, linewidth=0.04, arrowsize=0.08cm 1.75,arrowlength=1.4,arrowinset=0.42]{->}(0.10069458,0.0)(0.9006946,-0.79999995)
\psline[linecolor=black, linewidth=0.04, arrowsize=0.08cm 1.75,arrowlength=1.4,arrowinset=0.42]{->}(0.9006946,-0.79999995)(1.7006946,-1.5999999)
\psline[linecolor=black, linewidth=0.04, arrowsize=0.08cm 1.75,arrowlength=1.4,arrowinset=0.42]{->}(1.7006946,-1.5999999)(2.5006945,-2.3999999)
\psline[linecolor=black, linewidth=0.04, arrowsize=0.08cm 1.75,arrowlength=1.4,arrowinset=0.42]{->}(2.5006945,-2.3999999)(3.3006945,-3.1999998)
\psline[linecolor=black, linewidth=0.04, arrowsize=0.08cm 1.75,arrowlength=1.4,arrowinset=0.42]{->}(0.10069458,-1.5999999)(0.9006946,-2.3999999)
\psline[linecolor=black, linewidth=0.04, arrowsize=0.08cm 1.75,arrowlength=1.4,arrowinset=0.42]{->}(0.9006946,-2.3999999)(1.7006946,-3.1999998)
\psbezier[linecolor=black, linewidth=0.04, arrowsize=0.08cm 1.75,arrowlength=1.4,arrowinset=0.42]{->}(0.9006946,2.4)(0.9006946,1.6)(1.3006946,1.6)(1.7006946,1.6000000762939453)
\psbezier[linecolor=black, linewidth=0.04, arrowsize=0.08cm 1.75,arrowlength=1.4,arrowinset=0.42]{->}(1.7006946,1.6)(1.7006946,0.8000001)(2.1006947,0.8000001)(2.5006945,0.8000000762939453)
\psbezier[linecolor=black, linewidth=0.04, arrowsize=0.08cm 1.75,arrowlength=1.4,arrowinset=0.42]{->}(0.10069458,3.2)(0.10069458,2.4)(0.5006946,2.4)(0.9006946,2.4000000762939453)
\psbezier[linecolor=black, linewidth=0.04, arrowsize=0.08cm 1.75,arrowlength=1.4,arrowinset=0.42]{->}(2.5006945,0.8000001)(2.5006945,0.0)(2.9006946,0.0)(3.3006945,0.0)
\psbezier[linecolor=black, linewidth=0.04, arrowsize=0.08cm 1.75,arrowlength=1.4,arrowinset=0.42]{->}(4.9006944,-1.5999999)(4.9006944,-2.3999999)(5.3006945,-2.3999999)(5.7006946,-2.3999999237060545)
\psbezier[linecolor=black, linewidth=0.04, arrowsize=0.08cm 1.75,arrowlength=1.4,arrowinset=0.42]{->}(5.7006946,-2.3999999)(5.7006946,-3.1999998)(6.1006947,-3.1999998)(6.5006948,-3.199999923706055)
\psbezier[linecolor=black, linewidth=0.04, arrowsize=0.08cm 1.75,arrowlength=1.4,arrowinset=0.42]{->}(3.3006945,0.0)(3.3006945,-0.79999995)(3.7006946,-0.79999995)(4.1006947,-0.7999999237060547)
\psbezier[linecolor=black, linewidth=0.04, arrowsize=0.08cm 1.75,arrowlength=1.4,arrowinset=0.42]{->}(4.1006947,-0.79999995)(4.1006947,-1.5999999)(4.5006948,-1.5999999)(4.9006944,-1.5999999237060547)
\end{pspicture}
}
}

\def\Heap{% \usepackage[usenames,dvipsnames]{pstricks}
\psscalebox{.5 .5} % Change this value to rescale the drawing.
{
\begin{pspicture}[shift=-1.69](0,-1.6985577)(2.1971154,1.6985577)
\psdots[linecolor=black, dotsize=0.2](0.09855774,1.6)
\psdots[linecolor=black, dotsize=0.2](0.09855774,1.2)
\psdots[linecolor=black, dotsize=0.2](0.49855775,1.6)
\psdots[linecolor=black, dotsize=0.2](0.49855775,1.2)
\psdots[linecolor=black, dotsize=0.2](0.49855775,0.8)
\psdots[linecolor=black, dotsize=0.2](0.49855775,0.4)
\psdots[linecolor=black, dotsize=0.2](0.49855775,0.0)
\psdots[linecolor=black, dotsize=0.2](0.8985577,1.2)
\psdots[linecolor=black, dotsize=0.2](0.8985577,0.8)
\psdots[linecolor=black, dotsize=0.2](0.8985577,0.4)
\psdots[linecolor=black, dotsize=0.2](0.8985577,0.0)
\psdots[linecolor=black, dotsize=0.2](0.8985577,-0.4)
\psdots[linecolor=black, dotsize=0.2](1.2985578,-1.2)
\psdots[linecolor=black, dotsize=0.2](1.2985578,-0.8)
\psdots[linecolor=black, dotsize=0.2](1.2985578,-0.4)
\psdots[linecolor=black, dotsize=0.2](1.2985578,0.0)
\psdots[linecolor=black, dotsize=0.2](1.6985577,-1.6)
\psdots[linecolor=black, dotsize=0.2](1.6985577,-1.2)
\psdots[linecolor=black, dotsize=0.2](2.0985577,-1.6)
\end{pspicture}
}
}

\begin{document}

\begin{frontmatter}
\author{Sadek Al Harbat, Camilo Gonz\'alez and David Plaza}
\title{Type  $\tilde{C}$ Temperley-Lieb algebra quotients and Catalan combinatorics}
\address{Instituto de Matem\'atica y F\'isica\\ Universidad de Talca \\Talca, Chile}

\begin{abstract}
We study some algebraic and combinatorial features of two algebras that arise as quotients of Temperley-Lieb algebras of type $\tilde{C}$, namely, the two-boundary Temperley-Lieb algebra  and the symplectic blob algebra. We provide a monomial basis for both algebras. The elements of these bases are parameterized by certain subsets of fully commutative elements. We enumerate these elements according to their affine length.\\     
\end{abstract}

\end{frontmatter}

\section{Introduction}
\label{intro}
The Temperley-Lieb algebra was introduced around fifty years ago from considerations in statistical mechanics \cite{temperley1971relations}.  Since then it has turned out to be related to many topics of mathematics, including knot theory, algebraic combinatorics, algebraic Lie theory, etc. In 1987, Jones \cite{jones1987hecke} observed that the Temperley-Lieb algebra can be realized as a quotient of the Hecke algebra associated to a Coxeter system of type $A$. In his thesis,  Graham \cite{graham1995modular} took that observation far beyond its original scope. Concretely, given an arbitrary Coxeter system $(W,S)$, he defined a quotient of the Hecke algebra associated to $(W,S) $ that he called \emph{generalized Temperley-Lieb algebra}. Furthermore, he showed that for any $(W,S)$ the associated generalized Temperley-Lieb algebra admits a basis indexed by  the set of fully commutative (FC for short) elements  of $W$. 
 \begin{defi}
 	An element $w\in W$ is fully commutative if any   reduced expression for $w$ can be obtained from any other by using exclusively braid relations of the form $st=ts$,  $s,t\in S$.
 \end{defi}

Full commutativity has been given a proper place by Stembridge in the series of papers \cite{stembridge1996fully,stembridge1997some,stembridge1998enumeration}. In particular, and after classifying the Coxeter systems with finitely many FC elements \cite{stembridge1996fully}, he gave a normal form for FC elements in each of the infinite families of finite Coxeter systems \cite{stembridge1997some}. In this work we focus on type $B$. Later on the first author gave a normal form for FC elements in the four infinite families of affine Coxeter systems.  In this paper we are interested in type $\tilde{C}$ \cite{al2017typec}.

Let $K$ be an algebraically closed field, let $ n $ be a positive integer and let $\delta , \delta_L, \delta_R, \kappa_L,  \kappa_R,\kappa \in K^\times$. In this work we study the following $K$-algebras. 
\begin{itemize}
	\item The two-boundary Temperley-Lieb algebra with $n+1$ generators: $2BTL_n(\delta , \delta_L, \delta_R, \kappa_L,  \kappa_R)$.
	\item The symplectic blob algebra with $n+1$ generators: $SB_n(\delta , \delta_L, \delta_R, \kappa_L,  \kappa_R,\kappa)$.
\end{itemize}

\medskip
 
The first algebra was introduced in \cite{de2009two} as a quotient of the $\tilde{C}$-type affine Hecke algebra and was given a diagrammatical presentation. The second algebra was defined in \cite{martin2007towers} as an algebra of diagrams, then was given by generators and relations in \cite{green2012presentation} and was studied more deeply later on,  see \cite{reeves2011blob} for example. 

Given a positive integer $n$ we denote by $W(\tilde{C}_{n}) $  the affine Coxeter group of type $\atc$ ($W(\tilde{C}_{1}) $ to be understood  as the Coxeter group of type $B_2$). This is, $W(\tilde{C}_{n}) $ is the group  given by the following Coxeter diagram:

\begin{equation}
	\DynkintypeC
\end{equation}

Let  $\tl$ be the Temperley-Lieb algebra associated to $W(\tilde{C}_{n}) $. Let $W^c(\atc)$ be the set of FC elements in $W(\atc)$ and $\{b_w \, | \, w\in W^c(\atc) \}$ the monomial basis of $\tl$ (see Definition \ref{defin monomial basis}).  The starting point of this work is the following observation (already present in \cite{martin2007towers}): By specializing the parameters of the two algebras under focus we obtain a sequence of surjective morphisms of $K$-algebras 
\begin{equation} \label{intro surjective morphisms}
	\tl \twoheadrightarrow \twob \twoheadrightarrow SB_n(q,Q,\kappa),
\end{equation}
where $\twob$ and $SB_n(q,Q,\kappa)$ denote the specialized algebras (see Section \ref{section algebras} for  details). The tower of $\tilde{C}$-type TL algebras defined in \cite{al2017typec} gives rise to a faithful tower of two boundary TL algebras but not of symplectic blob algebras (see Remark \ref{abouttowers}).

\medskip
 We define two subsets of $W^c(\tilde{C}_n)$, namely: 
\begin{itemize}
	\item The set of positive fully commutative elements, $W^{+c+}(\tilde{C}_n)$ (see Definition \ref{defi positive fully commutative elements}).
	\item The set of blobbed fully commutative elements, $W^c_b(\tilde{C}_n)$ (see Definition \ref{defi blobbed fully commutative elements}).
\end{itemize}
The  set $\pos$ is infinite while $\blobbed $ is finite. Furthermore, we have
\begin{equation}
	W^c(\tilde{C}_n)\hookleftarrow  W^{+c+}(\tilde{C}_n)  \hookleftarrow  W^c_b(\tilde{C}_n).
\end{equation}
\noindent
Actually, those three sets index bases of the three algebras above. Indeed, our first main results are the following (where we use the same notation for  $b_w$ and its image under the morphisms in \eqref{intro surjective morphisms}).

\begin{thmx}\label{teo intro basis two boundary}
	The set $\{b_w \, | \, w\in \pos \}$ is a $K$-basis for $\twob$.
\end{thmx} 

\begin{thmx}  \label{teo intro basis symplectic}
	The set $\{b_w \, | \, w\in \blobbed \}$ is a $K$-basis for $SB_n(q,Q,\kappa)$.
\end{thmx}

\begin{defi}
 Let $w \in W^c (\tilde C_{n})$. We define the {\rm affine length} of $w$  to be the number of times  $t_n$ occurs  in a (any)  reduced expression of $w$. We denote  it by $L(w)$. 
\end{defi}

Viewing $W(\tilde{C}_{n}) $ as an ``affinization'' of $W(B_n)$ allows us to see how the affine length  can be a powerful tool in studying the behaviour of FC elements, since we have a finite number of elements of a given affine length and since we know exactly how to get from affine length $k$ to affine length $k+1$ via the normal form given in \cite{al2017typec}.\\

 Motivated by Theorem \ref{teo intro basis two boundary} and Theorem \ref{teo intro basis symplectic} we undertake the task of enumerating the elements of $\pos$ and $\blobbed $ according to their affine length. Let us make precise the above sentence. Let $n$ and $s$ be integers such that $n\geq 1$ and $s\geq 0$. We define $A_n^s =\{ w \in W^{+c+}(\tilde{C}_n) \, |\, L(w)=s \} $ and $B_n^s =\{ w \in W^c_b(\tilde{C}_n) \, |\, L(w) =s  \}$. Furthermore, we set $a_{n}^s=|A_n^s|$ and $b_n^s = |B_n^s |$. Then, our goal is to find closed formulas for the numbers $a_n^s $ and $ b_n^s $. 

\medskip
There are two main ingredients in our counting. On the one hand, the set $W^{+c+}(\tilde{C}_n)$ inherits naturally the normal form from $W^c( \atc)$. Fortunately, the normal form for the elements of  $W^{+c+}(\tilde{C}_n)$ simplifies drastically. This simplification allows us to introduce a graphical interpretation for any element $w\in W^{+c+}(\tilde{C}_n)$, which we call the grid of $w$ and denote by $G(w)$ (see Definition \ref{defi grid}). Then, we focus on enumerating grids rather than elements. In other words, we transform an algebraic problem into a combinatorial one. Since $\blobbed \subset \pos $ the same strategy applies for the enumeration of the elements of $\blobbed$ as well. On the other hand,  we introduce an array of numbers, which we call the \emph{Blobbed Catalan triangle}. This name is justified since it corresponds to a certain generalization of Forder's Catalan triangle \cite{forder1961some}. The blobbed Catalan triangle is an infinite matrix formed by non negative integers $\{C_{i,j}\}_{i,j\geq -1}$. The definition and properties of the blobbed Catalan triangle are given in Section \ref{section catalan combinatorics}, where in particular we provide a closed formula for the numbers $C_{i,j}$ in terms of binomial coefficients (see Theorem \ref{teo closed formula for Cs}). By combining these two ingredients we eventually obtain the following.

\begin{thmx}  \label{teo intro enumeration positive}
	Let $n$ and $s$ be integers such that $n\geq 1$ and $s\geq 0$. Then, we have $a_{n}^s=C_{2n,2s}$. 
\end{thmx}

\begin{thmx}\label{teo intro enumeration blobbed}
	 Let $D_n^s= A_n^s - B_n^s$ and set $d_{n}^s =|D_n^s|= a_n^s -b_n^s$. Then, there is a closed formula
	 for the numbers $d_n^s$ (see Section \ref{section enumeration blobbed} for the explicit formulas) which only  involves coefficients occurring in the blobbed Catalan triangle. 
	  Therefore, we have a closed formula for $b_n^s$. In particular, we obtain $ b_n^s=0 $ for $s > n$. Finally, the number 
\begin{equation}
	p_n:= \sum_{s=0}^n b_n^s = |\blobbed |
\end{equation}
gives the dimension of $ SB_n(q,Q,\kappa)$. 
\end{thmx}

This article is structured as follows. In Section \ref{section algebras} we introduce the algebras  under study. In Section \ref{section positive and twoboundary} we define positive FC elements and prove Theorem \ref{teo intro basis two boundary}. In Section \ref{section blobbed and symplectic} we define blobbed FC elements and prove Theorem \ref{teo intro basis symplectic}. In Section \ref{section catalan combinatorics} we introduce the blobbed Catalan triangle and study the properties of the numbers occurring in it. In Section \ref{section enumeration positive} we prove Theorem \ref{teo intro enumeration positive}. Finally, in Section \ref{section enumeration blobbed} we prove Theorem \ref{teo intro enumeration blobbed}.

\section*{Acknowledgements}
Sadek Al Harbat was supported by Fondecyt Postdoctoral grant 3170544. David Plaza was partially supported by FONDECYT project 11160154 and the Inserci\'on en la Academia project PAI-Conicyt 79150016. 
%%%%%%%%%%%%%%%%%%%%%%%%%%%%%%%%%%%%%%%%%%%%%%%%%%%%%%%%%%%%%%%%%%%%%%%%%%

\section{Definition of algebras} \label{section algebras}
In this section we define the  algebras which we study . Hereinafter, $K$  will denote an algebraically closed field. Let $n$ be a positive integer.  Let $W(\tilde{C}_{n}) $ be the affine Coxeter group of type $\tilde{C}_n$.  
\begin{defi} \label{defi Affine C Hecke algebra} 
	Let $q,Q\in K^\times$. The Hecke algebra of type $ \atc $, $H_n(q,Q)$, is the associative  unital $K$-algebra with generators 
	$\{ g_0,g_1,\ldots ,g_{n-1},g_n \}$ subject to the relations
	\begin{equation}  \label{relations in Hecke affine Type C}
	\begin{array}{rlccl}
	g_i^2 & =(Q-1) g_i+Q, &   &   &  \mbox{if } i=0 \mbox{ or } i=n;\\
			g_i^2 & =(q-1) g_i+q, &   &   &  \mbox{if } 0< i <n;\\
			g_ig_j & =g_jg_i, & & &\mbox{if } |i-j|>1; \\
			g_ig_jg_i& = g_jg_ig_j, & & & \mbox{if } |i-j|=1 \mbox{ and } 0<i,j<n;\\
			g_ig_jg_ig_j& = g_jg_ig_jg_i, & & & \mbox{if } \{i,j\}=\{0,1\} \mbox{ or }  \{i,j\}=\{n-1,n\} .
	\end{array}
\end{equation}
\end{defi}

For $x,y$ in a given ring with  identity,  we define:
$$\begin{aligned}
V(x,y) &= xyx+xy+yx+x+y+1, \\ 
  Z(x,y)   &=  xyxy+xyx+yxy+xy+yx+x+y+1.
\end{aligned}
$$

\begin{defi} \label{defi affine C TL algebra} 
	Let $q,Q\in K^\times$. The Temperley-Lieb algebra of type $ \atc $, $ \tl $, is the associative unital $K$-algebra with generators 
	$\{ g_0,g_1,\ldots ,g_{n-1},g_n \}$ subject to the relations as in \eqref{relations in Hecke affine Type C} together with
	\begin{equation}
		Z(g_0,g_1) = Z(g_{n-1},g_n) =V(g_i,g_{i+1})  =0,   \mbox{ for } 0< i <n.
\end{equation}
\end{defi}

By setting $U_0=\frac{1}{\sqrt{Q}} (g_0+1)$, $U_n=\frac{1}{\sqrt{Q}} (g_n+1)$ and $U_i=\frac{1}{\sqrt{q}} (g_i+1)$ for $0<i<n$, we obtain another presentation for $\tl $ with generators $\{U_0,U_1, \ldots , U_n\}$ subject to the relations
\begin{equation} \label{relations temperley lieb five parameters}
	\begin{array}{rlccl}
		U_0^2 &= \delta_L U_0; & & & \\
		U_n^2 & =\delta_R U_n ;& & & \\
			U_i^2 & =\delta U_i, &   &   &  \mbox{if } 0< i <n;\\
			U_iU_j & =U_jU_i, & & &\mbox{if } |i-j|>1; \\
			U_iU_jU_i& = U_i, & & & \mbox{if } |i-j|=1 \mbox{ and } 0<i,j<n;\\
			U_iU_jU_iU_j& = k_LU_iU_j, & & & \mbox{if } \{i,j\}=\{0,1\} ;\\
	U_iU_jU_iU_j& = k_RU_iU_j, & & & \mbox{if } \{i,j\}=\{n-1,n\} ;
	\end{array}
\end{equation}
where 
\begin{equation}  \label{parameters specialization}
	\displaystyle     \delta_L = \delta_R =\frac{1+Q}{\sqrt{Q}}, \qquad  \delta = \frac{1+q}{\sqrt{q}}, \qquad  k_L =k_R= \frac{q+Q}{\sqrt{qQ}}.
\end{equation}

In order to describe a basis for $\tl$ we need to recall the following.
   
   \begin{defi}\rm
			In a Coxeter system $(W,S)$ of graph $\Gamma$, elements for which one can pass from any reduced expression to any other one only by applying commutation relations are called {\rm fully commutative elements}. We denote  by $W^{c}(\Gamma)$ the set of fully commutative elements in $ W(\Gamma)$. 
   \end{defi}

 \begin{defi}\rm \label{defin monomial basis}
 	Given $w$ in $W^c(\tilde C_{n}) $ and $w= \sigma_{i_1} \dots \sigma_{i_r} $  any reduced expression of $w$, we set $b_w=  U_{i_1} \dots U_{i_r}$.
It is well-known that  $b_w$ is well-defined and that the set $\{ b_w \ | \ w \in W^c(\tilde C_{n})\}$ is  a basis for $\tl $, which is called the  monomial  basis.
 \end{defi}  

 \begin{rem}\rm  
In the literature (for example \cite{ernst2012diagram})  we see that  $\tl  $  was defined in the equal parameters case, i.e., when $q=Q$. Sometimes it was defined with a weight function (for example \cite{graham1995modular}). Now since the $K$-vector space structure of the algebras $TL\atc (q,q)$ and  $\tl $ is the same, we use the  monomial basis for $\tl $ exactly in the same way in which it was used in the equal parameters case.
 \end{rem}

\begin{defi} \rm
Let $\delta, \delta_L, \delta_R, \kappa_L , \kappa_R \in K^\times$. The two-boundary Temperley-Lieb algebra, which we denote by $2BTL_n(\delta , \delta_L, \delta_R, \kappa_L , \kappa_R)$, is the associative unital $K$-algebra with generators $\{U_0,U_1, \ldots , U_n\}$ subject to the  relations  \eqref{relations temperley lieb five parameters} together with 
\begin{equation} \label{relation to pass from TL to two-boundary}
U_1U_0U_1 = \kappa_LU_1 \qquad \mbox{and} \qquad    	U_{n-1}U_nU_{n-1} = \kappa_RU_{n-1}.
\end{equation}
\end{defi}

By specializing  the parameters $\delta, \delta_L, \delta_R, \kappa_L$ and $  \kappa_R $ as in \eqref{parameters specialization} we obtain a two-parameter version of the two-boundary Temperley-Lieb algebra which we denote by $\twob $. In this setting, we see that  $\twob $ is the quotient of $ \tl $ by the ideal generated by

\begin{equation}
	U_1U_0U_1 - \kappa_L U_1 \qquad \mbox{and} \qquad    	U_{n-1}U_nU_{n-1} - \kappa_R U_{n-1}.
\end{equation}

\begin{defi} \rm
Let $\delta, \delta_L, \delta_R, \kappa_L , \kappa_R, \kappa \in K^\times$. The symplectic blob algebra, which we denote by  $SB_n(\delta , \delta_L, \delta_R, \kappa_L , \kappa_R,\kappa)$, is the associative unital  $K$-algebra with generators $\{U_0,U_1, \ldots , U_n\}$ subject to the relations \eqref{relations temperley lieb five parameters} and \eqref{relation to pass from TL to two-boundary}, together with 
\begin{equation} 
IJI = \kappa I \qquad \mbox{and} \qquad    	JIJ = \kappa J,
\end{equation}
where 
\begin{equation}
	I = \left\{  \begin{array}{ll}
	U_1U_3\cdots U_{n-1},  	& \mbox{if } n \mbox{ is even;}     \\
	U_1U_3\cdots U_{n},  	& \mbox{if } n \mbox{ is odd;}     \\
	\end{array}  \right.
\qquad \mbox{and}\qquad 
J = \left\{  \begin{array}{ll}
	U_0U_2\cdots U_{n},  	& \mbox{if } n \mbox{ is even;}     \\
	U_0U_2\cdots U_{n-1},  	& \mbox{if } n \mbox{ is odd.}     \\
	\end{array}  \right.
\end{equation}
\end{defi}

By specializing  the parameters $\delta, \delta_L, \delta_R, \kappa_L$ and $  \kappa_R $ as in \eqref{parameters specialization} we obtain a three-parameter version of the symplectic blob algebra which we denote by $\simp  $. In this setting, we see that  $\simp $ is  the quotient of $ \twob $ by the ideal generated by

\begin{equation} 
IJI - \kappa I \qquad \mbox{and} \qquad    	JIJ - \kappa J.\\
\end{equation}

With the exception of the Hecke algebra of type $\atc$, all the algebras defined in this section have a diagram calculus given by certain generalizations of classical Temperley-Lieb diagrams (see \cite{martin2007towers,de2009two,ernst2012diagram}). In this paper we do not touch the diagrammatic setting and for this reason we do not recall it here.

\begin{rem} \label{abouttowers}\rm
In \cite[\S 6]{al2017typec} the first author has  defined the morphism $R_n: TL\tilde{C}_{n}(q,Q) \longrightarrow TL\tilde{C}_{n+1}(q,Q)$ and has  shown  the faithfulness of this morphism. We can easily see that the composition with the quotient morphism 
onto $2BTL_{n+1}(q,Q)$  factors through  $2BTL_{n}(q,Q)$, so   $R_n$ gives rise to  a morphism of algebras $\bar{R}_n: 2BTL_n(q,Q) \longrightarrow 2BTL_{n+1}(q,Q)$. Since the maps $I,J$ defined in  \cite[Definition 5.1]{al2017typec}    send positive elements -see below- to positive elements (with a slight change concerning the definition of $I$ and $J$ on elements with affine length $1$), the faithfulness of $\bar{R}_n$ follows. The importance of this morphism comes from the fact that it comes from a morphism of $\tilde{C}$-type braid groups, thus the resulting  faithful tower of two-boundary T-L algebras encodes most-likely topological data for example. On the other hand the morphism $\bar{R}_n$ fails to 
factor through the symplectic blob algebras, so we do not get in this way a tower of symplectic blob algebras, unfortunately. 
\end{rem}

%%%%%%%%%%%%%%%%%%%%%%%%%%%%%%%%%%%%%%%%%%%%%%%%%%%%%%%%%%%%%%%%%%%%%%%%%%%

\section{Positive fully commutative elements and the two-boundary Temperley-Lieb algebra }
\label{section positive and twoboundary}  

The purpose of this section is to determine a (monomial) basis for $\twob$. Concretely, we find a subset of the monomial basis of $\tl $ whose elements are mapped under the canonical projection to a basis for $\twob$. To this end, we begin by recalling the classification of fully commutative elements in type $\atc$.

 \begin{defi}\label{defin fully commutative elements} \rm
 Let $w \in W (\tilde C_{n})$. We define the \emph{ affine length}  of $w$  to be the number of times  $t_{n}$ occurs   in a (any) 
 reduced expression of $w$. We denote it by $L(w)$. In particular, we identify  $  \{ w \in W(\tilde C_{n}) \, | \, L(w) =0 \}$   with the Coxeter group $W(B_{n})$.
\end{defi}

Before recalling the classification we write some suitable notations. We define
\begin{equation}  \label{first parentheses}
	\begin{array}{rlll}
	     [i,j]  &= \sigma_i \sigma_{i+1} \dots \sigma_j , & & \text{ for } 0\leq i\leq j < n  \quad   \mbox{ and }   \quad   [ n, n-1 ]   = 1;  \\
	    \lbrack   -i, j \rbrack  & = \sigma_i \sigma_{i-1} \dots  \sigma_1 \sigma_0\sigma_{1}  \dots \sigma_{j-1}  \sigma_j , &  &  \mbox{ for } 1\leq i\leq j < n    \quad   \mbox{ and }   \quad      [0, -1]   = 1 .
\end{array}  
\end{equation}

The classification of fully commutative elements of affine length zero in type $\atc$  (or equivalently  fully commutative elements in type $B_n$) was given by Stembridge  \cite[Theorem 5.1]{stembridge1997some}. We recall it in the following result.

\begin{teo}\label{1_2}
  Let  $w$ be a fully commutative element in $W^c(B_n)$ different from the identity. Then, $w$ can be written in a unique way as a reduced word of the form
   \begin{equation}\label{Stembridge}
[ l_1, g_1 ] 	[  l_2, g_2 ]  \dots  [l_r , g_r ]	
 \end{equation}
with $n> g_1 > \dots > g_r \ge 0$ and 
$|l_t| \le g_t$ for $1\le t \le r$, such that  either

\begin{itemize}
\item[(a)] $ l_1 > \dots > l_{r-1}  > l_r > 0 $;
\item[(b)] $ l_1 > \dots > l_s  = l_{s+1} = \cdots = l_r = 0  $ for some  $s \le r$; or 
\item[(c)] $ l_1 > \dots > l_{r-1}  > -l_r > 0 $. 
\end{itemize}
\end{teo}
 
Hereinafter, we call elements of $W^c(B_{n})$ satisfying (c) in Theorem \ref{1_2} negative elements. We call  positive element any $w \in W^c(B_{n})$ which is not negative. We also call elements satisfying (a) in Theorem \ref{1_2} strictly positive elements. Finally, we consider the identity as a strictly positive element. 

\begin{lem}  \label{lemma counting psotitive fully commu in type B}
The number of positive elements in $W^c(B_{n})$ is $\binom{2n}{n} $.
\end{lem}
\begin{demo}
	By \cite[Proposition 5.9]{stembridge1997some} we know that 
	\begin{equation} \label{lema positive type B A}
		|W^c(B_n)|= (n+2)C_n-1, 
	\end{equation}
	where $C_n = \frac{1}{n+1}\binom{2n}{n}$ denotes the $n$-th Catalan number.
	Let us denote by $P$, $SP$ and $N$ the set of positive, strictly positive and negative fully commutative elements in $W^c(B_{n})$, respectively. We notice that the elements of $SP$ are precisely the elements of  $W^c(B_{n})$ in which $\sigma_0$ does not appear. We conclude that 
	\begin{equation}\label{lema positive type B B}
		|SP|=|W^c(A_{n-1})|= C_n. 
	\end{equation}
	 On the other hand, there is a bijection  $N \rightarrow SP-\{ 1 \} $, which is obtained by replacing $l_r$ by $-l_r$. Therefore, 
	\begin{equation}\label{lema positive type B C}
		|N|= C_n-1. 
	\end{equation}
	Since $W^c(B_n)= P\, \dot{\cup} \, N $, a combination of \eqref{lema positive type B A} and \eqref{lema positive type B C} yields the result. 
\end{demo}

 \medskip
  We now explain the classification of elements of  $W^c(\atc )$ with positive affine length.

%A direct consequence of Theorem \ref{1_2} is the following classification of positive elements in $W^c(B_{n})$.
%  
%  \begin{coro}\label{1_3}
%  The set of positive elements in $W^c(B_{n})$ is the set of elements of the following form: 
% \begin{equation}\label{Stembridgepositive}
%[ l_1, g_1 ] 	[  l_2, g_2 ]  \dots  [l_r , g_r ]	
% \end{equation}
%with $n-1\ge g_1 > \dots > g_r \ge 0$ and 
%$l_t \le g_t$ for $1\le t \le r$, such that  either
%
%\begin{itemize}
%\item[\rm (a)] $ l_1 > \dots > l_s  > l_{s+1} = \dots = l_r = 0  $ for some  $s \le r$, or 
%\item[\rm (b)] $ l_1 > \dots > l_{r-1}  > l_r > 0 $. \\
%\end{itemize}
%  \end{coro}
% 
%  We notice that the set of positive elements  in which $\sigma_0$ does not appear is exactly $W^c(A_{n-1})$. Therefore, the set $W^c(A_{n-1})$ is formed by the elements satisfying (b) in Corollary \ref{1_3}.  We now explain the classification of elements of  $W^c(\tilde C_{n})$ with positive affine length.

 \begin{teo}\label{FC}{\rm \cite[Theorem 4.7]{al2017typec}}
Let $w \in W^c(\tilde C_{n})$ with $L(w)   \ge 2$.  
Then $w$ can be written in a unique way as a reduced word of   one and only one of the following two forms, for non negative integers $p$ and $k$:  
\begin{description}
\item[First type]
\begin{equation}\label{formefinalefirsttype}
  w =   [  i,n-1 ]   t_{n}( [  -(n-1),n-1 ] \;  t_{n} )^k  ([ f,n-1 ] )^{-1} \end{equation}
 with $k\geq 1$,  $-n < i \le n $ and $-n < f \le n $.\\ 

\item[Second type]
\begin{equation}\label{formefinalesecondtype}
\begin{aligned}
w &=   [  i_1,n-1 ]  \; t_{n}  [  i_2,n-1 ]  \; t_{n} \dots  [  i_p,n-1 ]  \;  t_{n}    ( [  0,n-1 ]  \; t_{n} )^k  
\;    w_r   \quad  \text{if } p>0 ,  \\  
w &=       ( [  0,n-1 ]  \; t_{n} )^k  
\;    w_r  \quad   \text{if } p=0, 
 \end{aligned}
 \end{equation}
  with $w_r \in W^c(B_{n})$   and  \\

\begin{itemize}
\item if $k > 0$:    $w_r=1$ or $w_r= [ 0,r_1 ][ 0,r_2 ] \dots [ 0,r_u ]$ with $0 \le r_u < \dots < r_1  < n$; 
\item   if $p > 0$: 
 $ \  
n \ge i_1>... > i_{p-1} > |i_p| >0    \  $; 
\item   if $p > 0$ and   $  i_p  < 0$:   $k=0$,  $w_r=1$ and $i_p \neq -(n-1)$;
\item   if $k =0$ and $i_p > 0$:  $w_r= \lbrack l_1,g_1 \rbrack \lbrack l_2,g_2 \rbrack \dots \lbrack l_r,g_r  \rbrack $ with $ |l_1 |< i_p$. 
\end{itemize}
\end{description}

The affine length of $w$ of the first (resp. second) type is  $k+1$ (resp. $p + k$) and we have  $ 0 \le p \le n  $. \\

Now suppose that $L(w) = 1$, then it has a reduced expression of the form: 
 \begin{equation}\label{formefinalelongueur1}
  [  i,n-1 ]  \  t_{n}   \    v  
  \end{equation}
 where

\begin{itemize}
\item  if $0 < i \le n$ then $v=[ l_1,g_1 ][ l_2,g_2 ] \ldots [ l_r,g_r ]  $  such that for $1\leq j \leq r$ either $l_j= n-j$ or $l_j<i$;
\item   if $i <0 $ then $v = ([ h,n-1 ] )^{-1}$ with $-n < h \le n$;
\item  if $  i=0 $ then 
\begin{equation}
	v=\left\{  \begin{array}{ll}
		([ h,n-1 ] )^{-1},  &  \mbox{for } -n < h \le n, \mbox{ or } \\
		([ z,n-1 ])^{-1}[ 0,r_1 ][ 0,r_2 ] \dots [ 0,r_m ], & \mbox{for } 0 \le r_m < \dots r_2 < r_1 < z \le n .
	\end{array}   \right. 
\end{equation}

%$v$ is equal to $ ([ h,n-1 ] )^{-1}$ for $-(n-1) \le h \le n$, or to 
%
%$([ z,n-1 ])^{-1}[ 0,r_1 ][ 0,r_2 ] \dots [ 0,r_m ]$ for $-1 \le r_m < \dots r_2 < r_1 < z \le n$.\\
\end{itemize}

Conversely, every $w$ of the above form is in $W^c(\tilde C_{n})$.
\end{teo} 
 
We want to extend the notion of positivity to the elements of $W^c(\tilde C_{n})$. To this end we need to define the notions of subword and avoidance. 

\begin{defi}\rm
	Let $(W,S)$ be an arbitrary Coxeter system. By a subword of a word $w =s_1s_2\ldots s_r $ ($s_i\in S$) we mean a word of the form $s_is_{i+1} \ldots s_j  $, for some $1\leq i \leq j \leq r $. 
	\end{defi}

\begin{defi}\rm
	 Given two elements $w$ and $u $ in $W$ we say that $w$ contains $u$ if there exist reduced expressions $\underline{w} $ of $w$ and $\underline{u}$ of $u$ such that $\underline{u}$ is a subword of $\underline{w}$. Otherwise, we say that $w$ avoids $u$ or that $w$ is $u$-avoiding.    
\end{defi}

We are now ready to define the promised extension of the notion of positivity. 

\begin{defi}\rm \label{defi positive fully commutative elements}
 	 An element $w\in W^c(\atc )  $ is called left-positive (resp. right-positive) if it avoids $\sigma_1 \sigma_0 \sigma_1$ (resp. $\sigma_{n-1}\sigma_{n}\sigma_{n-1}$). Elements which are both left-positive and right-positive will be called positive.    We denote by  $ W^{+c}(\tilde C_{n})$  (resp. $W^{c+}(\tilde C_{n}) $ ) the set of  all left-positive (resp. right-positive) elements  in $W^c(\atc)$. Finally, we denote by  $ W^{+c+}(\tilde C_{n})$  the set of all positive elements in  $W^c(\tilde C_{n})$. 
 \end{defi}

 \begin{rem} \rm \label{remark positivity is detected by normal forms}
 	We notice that the normal form of a given fully commutative element  is a reduced expression  that allows us to determine if such an element is left-positive (resp. right-positive) or not. More precisely, let  $\underline{w}$ be the normal form of an element $w\in W^c(\atc )$. Then, $w$ is left-positive (resp. right-positive) if and only if $\sigma_1\sigma_0 \sigma_1$  (resp. $\sigma_{n-1}\sigma_{n}\sigma_{n-1}$) does not occurs as a subword of $\underline{w}$. %It is clear for the first type elements since they are rigid, while for second type and 1-affine length elements the maximal number for appearance is 1, and we see it explicitly whenever the element is left-positive or right-positive.
 \end{rem}

We define $Lp$ to be  the two-sided ideal of  $ \tl $ generated by  $ U_{1}U_0  U_{1} -\kappa_L U_{1} $. We shall, in what follows, investigate the behavior of the elements of the monomial basis of $\tl $ under the canonical surjection $ \tl \rightarrow \tl / Lp $. We keep the same notations for elements of $\tl$ and their images. Our aim is to show that left-positive elements index a basis for $\tl /Lp $. To this end, we study the structure of the ideal $Lp$.\\
  
A precise inspection of the  normal forms in Theorem \ref{1_2} and Theorem \ref{FC}  leads to the following classification of non-left-positive elements, where we use the same notation as in the referred theorems.  

\begin{propos}\label{proposition classification non-left positive}
An element  $w\in W^c(\atc )$  is not left-positive if and only if it satisfies one of the following conditions:

\begin{itemize}
\item[(1)] $L(w)\geq 2$ and $w$ is a first type element. Here  $\sigma_1 \sigma_0 \sigma_1$ appears at least $L(w)-1$ times, 
at most $L(w)+1$ times.
\item[(2)]  $L(w)\geq 2$ and $w$ is a second type element with $k=0$, hence $p\ge 2$, and either 
\begin{itemize}
	\item[(a)]  $i_p < 0$ and $w_r=1$, or
	\item[(b)]  $i_p >0$ and $w_r= [ l_1,g_1 ][ l_2,g_2 ] \ldots [ l_r,g_r ] $ with  $l_r<0$.
\end{itemize}
In both cases  $\sigma_1 \sigma_0 \sigma_1$ appears exactly once.
\item[(3)] $L(w)=1$ and  either
\begin{itemize}
\item[(a)]  $i < 0$ and $h\ge 0$, then $\sigma_1 \sigma_0 \sigma_1$ appears once in $w$,
\item[(b)]  $i < 0$ and $h< 0$, then $\sigma_1 \sigma_0 \sigma_1$ appears twice in $w$,
\item[(c)]   $i=0$, $v=[h,n-1]^{-1} $ and $h<0$, then   $\sigma_1 \sigma_0 \sigma_1$ appears once in $w$, or
\item[(d)]   $i > 0$ and $l_r < 0$, then    $\sigma_1 \sigma_0 \sigma_1$ appears once in $w$.
\end{itemize}

\item[(4)]  $L(w)=0$ and w is a negative element in $W^c(B_n)$. Here  $\sigma_1 \sigma_0 \sigma_1$ appears exactly once. 
\end{itemize}
\end{propos}

\begin{defi} \rm
Given an element $w \in W^c(\atc )- W^{+c}(\atc ) $ we define $\overline w \in W^c(\atc)$ as follows (we use the same notation as in Theorem \ref{1_2} and Theorem \ref{FC} and consider the same cases as in Proposition \ref{propo classi left non right}).

\begin{itemize}
 \item[(1)] Suppose that $L(w)\geq 2$ and $w$ is a first type element. We recall that this means that
  \begin{equation*}
 	w= [  i,n-1 ]   t_{n}( [  -(n-1),n-1 ] \;  t_{n} )^k  ([ f,n-1 ] )^{-1} ,
 \end{equation*}
for some  integers $k$, $i $ and $f$ such that $k\geq 1$, $-n<i\leq n$ and $-n < f \leq n$. In this case, we define
\begin{equation*}
 	\overline w =\left\{  \begin{array}{ll}
 		 \lbrack  i,n-1 \rbrack   t_{n}( \lbrack -(n-1),n-1 \rbrack  \;  t_{n} )^k  ( \lbrack -f,n-1 \rbrack )^{-1},	    &         \mbox{if } f<0; \\
 		  t_n\sigma_{n-1}t_n, &   \mbox{if }  k=1, f=n \mbox{ and }  i = n; \\
 		 \lbrack i,n-1  \rbrack   t_n ( \lbrack -(n-1),n-1 \rbrack t_n  )^{k-1} ( \lbrack f,n-1 \rbrack )^{-1},     &   \mbox{otherwise. } 
 	\end{array}   \right. 
 \end{equation*}

 \item[(2)]  Suppose that $L(w)\geq 2$ and $w$ is a second type element with $k=0$ and $p\geq 2$.
 \begin{itemize}
\item[(a)] If $i_p<0$ and $w_r=1$ then  $\overline{w}$ is obtained from the normal form of $w$ by replacing  the block $[i_p,n-1]$  by the block $[-i_p,n-1]$. 
\item[(b)] If $i_p>0$ and $w_r= [ l_1,g_1 ][ l_2,g_2 ] \ldots [ l_r,g_r ] $ with  $l_r<0$ then $\overline{w}$ is obtained from the normal form of $w$ by replacing the block $[l_r,g_r]$ by the block $[-l_r,g_r ]$.
  \end{itemize}
 
\item[(3)] Suppose that $L(w)=1$. 
 \begin{itemize}
 \item[(a)] If $i<0$ and $h\geq 0$ then $\overline{w}$ is obtained from the normal form of $w$ by replacing  the block $[i,n-1]$  by the block $[-i,n-1]$. 
 \item[(b)] If $i<0$ and $h<0$ then $\overline{w}$ is obtained from the normal form of $w$ by replacing  the block $[h,n-1]^{-1}$  by the block $[-h,n-1]^{-1}$. 
 \item[(c)] If $i=0$, $v=[h,n-1]^{-1}$ and $h<0$  then $\overline{w}$ is obtained from the normal form of $w$ by replacing  the block $[h,n-1]^{-1}$  by the block $[-h,n-1]^{-1}$. 
 \item[(d)] If $i>0$ and $l_r<0 $ then $\overline{w}$ is obtained from the normal form of $w$ by replacing  the block $[l_r,g_r]$  by the block $[-l_r,g_r]$. 
\end{itemize}
\item[(4)] Suppose that $L(w)=0$ and $w=[l_1,g_1][l_2,g_2]\cdots [l_r,g_r]$ with $l_r<0$. Then, we define $\overline{w}=[l_1,g_1][l_2,g_2]\cdots [-l_r,g_r] $.
 \end{itemize}
	
\end{defi}

%We observe that in cases (2), (3c) and (4)  above   we can write 
%$w=u  [-a, b]$ as a reduced expression,    with $1\le a \le b \le n-1$  and $u$    left positive.   
%In case (3b) we can write $w=u  [-a, n-1]^{-1}$ as a reduced expression, with $1\le a \le n-1$ and $u$ left-positive. Finally, in case (3a) we can write $w= [-a, n-1]u$ as a reduced expression with  $1\le a \le n-1$, but this time $u$ may not be left-positive.  
%
%
%\begin{defi} \rm
%Given an element $x \in W^c(\atc )- W^{+c}(\atc ) $ we define $\bar x \in W^c(\atc)$ as follows.
%
%\begin{itemize}
% \item[a)] Suppose that $x$ is not a first type element. If $x=u[ - i,j ] $ (resp.  $x= [ - i,j ]u$), then 
% \begin{equation*}
% 	\bar x = u[  i,j ] \qquad  (\mbox{resp. } \bar x= [i,j]u ).
% \end{equation*}  
% 
% \item[b)] Suppose that $x$ is a first type element. We recall that this means that
%  \begin{equation*}
% 	x= [  i,n-1 ]   t_{n}( [  -(n-1),n-1 ] \;  t_{n} )^k  ([ f,n-1 ] )^{-1} ,
% \end{equation*}
%for some  integers $k,i $ and $f$ such that $k\geq 1$, $-n<i\leq n$ and $-n < f \leq n$. In this case, we define
% 
% 
% \begin{equation*}
% 	\bar x =\left\{  \begin{array}{ll}
% 		 \lbrack  i,n-1 \rbrack   t_{n}( \lbrack -(n-1),n-1 \rbrack  \;  t_{n} )^k  ( \lbrack -f,n-1 \rbrack )^{-1},	    &         \mbox{if } f<0; \\
% 		  t_n\sigma_{n-1}t_n, &   \mbox{if }  k=1, f=n \mbox{ and }  i = n; \\
% 		 \lbrack i,n-1  \rbrack   t_n ( \lbrack -(n-1),n-1 \rbrack t_n  )^{k-1} ( \lbrack f,n-1 \rbrack )^{-1},     &   \mbox{otherwise. } 
% 	\end{array}   \right. 
% \end{equation*}
%
% 
% \end{itemize}
%	
%\end{defi}

It is obvious from the definition that  $l(w) > l(\overline w) +1$,  for any $w\in W^c(\tilde C_{n}) - W^{+c}(\tilde C_{n})$. We now extend the bar operator to a linear transformation, 
\begin{equation*}
	\bar  \space  :   \mbox{Span}_K  \{ b_x \, | \,  x\in W^c(\tilde C_{n}) - W^{+c}(\tilde C_{n}) \} \longrightarrow \tl ,
\end{equation*}
 which is determined in the set $\{ b_x \, | \,  x\in W^c(\tilde C_{n}) - W^{+c}(\tilde C_{n}) \}$ by the rule
 
 \begin{equation}
 	\overline{b_x} =  \left\{   \begin{array}{rl}
 b_{\bar{x}},	&  \mbox{if } x \mbox{ is not a first type element;} \\
 b_{\bar{x}}, & \mbox{if } x \mbox{ is  a first type element with } f<0 \mbox{ or }  k=1, f=n \mbox{ and }  i = n; \\
 		 \kappa_R  b_{\bar{x}},  & \mbox{otherwise}.  
  	\end{array}      \right.
 \end{equation}

% \begin{defi}\rm \label{definebar}
% 
% We define $\bar\space$ from the vector space  $ \{b_w; w\in W^c(\tilde C_{n}) - W^{+c}(\tilde C_{n}) \} $ to $TL\tilde{C}_{n}(q)$ ($x$ to be understood $b_x$ in what follows):\\
% 
% \begin{itemize}
% 
% \item if $x=u[ - i,j ] $ (resp.  $[ - i,j ]u$)   is not of the first type, then $\bar x= u[  i,j ]$ (resp. $ [  i,j ]u$), reduced obviously. \\
% 
% \item if $x= [  i,n-1 ]   t_{n}( [  -(n-1),n-1 ] \;  t_{n} )^k  ([ f,n-1 ] )^{-1} $ with $ 1 \le k $ and $1 \le- f \le n-1 $, then $\bar x=  [  i,n-1 ]   t_{n}( [  -(n-1),n-1 ] \;  t_{n} )^k  ([- f,n-1 ] )^{-1}$ .\\ 
%
% 
% \item if $x= [  i,n-1 ]   t_{n}( [  -(n-1),n-1 ] \;  t_{n} )^k  ([ f,n-1 ] )^{-1} $ with $ 1 \le k $  and $0 \le f \le n $, then $\bar x=\kappa_R [  i,n-1 ]   t_{n}( [  -(n-1),n-1 ] \;  t_{n} )^{k-1}  ([ f,n-1 ] )^{-1}  $.\\ 
%
% \item if $x= [  -n,-n]$ then $ \bar x= t_n \sigma_{n-1} t_n$.\\
% 
% \end{itemize}
%
% \end{defi}
% 
% It is obvious from the definition that for any $x\in W^c(\tilde C_{n}) - W^{+c}(\tilde C_{n})$ we  have : $l(x) > l(\bar x) +1$. \\

 \begin{teo}\label{basisLp} 
 
 The set $X= \{ b_x-\kappa_L  \overline{b_{ x}} \, |\,    x\in W^c(\tilde C_{n}) - W^{+c}(\tilde C_{n})  \}$ is a $K$-basis of $Lp$. 
 
 \end{teo} 
 
 \begin{proof}   It is enough to prove  that:
 \begin{enumerate}
 	\item  The set $X$ is linearly independent,
 	\item  $ U_1U_0U_1-\kappa_LU_1 \in  X$,
 	\item  $X\subset Lp$,
 	\item  The $K$-linear space spanned by $X$ is an ideal of $\tl$. 
 \end{enumerate}

\noindent
1.  We recall that $\{ b_x \,| \, x \in W^c(\tilde C_n)\}$ is the monomial basis of $\tl $. Consider the $K$-linear map $\Phi: \tl \longrightarrow \tl $ determined in the monomial basis by the rule
\begin{equation}
	  \Phi (b_x) = \left\{  \begin{array}{ll}
	  b_x,	& \mbox{if } x \in   W^{+c}(\tilde C_{n})  \\
	  b_x-\kappa_L \overline{b_{x}}, & \mbox{if } x\in  W^c(\tilde C_{n}) - W^{+c}(\tilde C_{n}) .
	  \end{array} \right. 
\end{equation}
Since $l(x) > l(\overline{x} )$, for any non-left-positive element $x$, we see that $\Phi$  is a linear map with
unipotent triangular matrix for any ordering of the monomial basis compatible with the length.  Hence,  $\Phi$ is an isomorphism and 
\begin{equation} \label{another way to say X}
	X =\Phi (\{ b_x \, |\, x \in   W^c(\tilde C_{n}) - W^{+c}(\tilde C_{n})    \} )
\end{equation}
is linearly independent. \\
%{\red{Actually, it is worth mentioning the formula  
%$$ x= \sum^{i=d-1}_{i=0} \kappa^i(x ^{(i)}-\kappa x ^{-(i+1)})+ \kappa^d \bar x ^{(d)} $$
%where $x$ is non left positive and $\bar x ^{(d)} $ is positive.\\ 
%Is it worth? If so, why we didn't refer such a formula in the following?}}
 
\noindent 
2.  If $x= \sigma_1\sigma_0\sigma_1$ then $\overline{x} =\sigma_1$. Therefore, $U_1U_0U_1-\kappa_L U_1 = b_x-\kappa_L \overline{b_x} \in X$.\\
 
\noindent 
3.  Let $\mathscr{L}_p : \tl \longrightarrow \tl/ Lp  $ be the quotient map. Given $1\leq i \leq j <n$ we have
\begin{equation}
\begin{aligned}\label{equation new ideal}
	\mathscr{L}_p(b_{[-i,j]}) & = \mathscr{L}_p(U_iU_{i-1}\cdots U_1U_0U_1 \cdots U_j ) \\
	 & = \kappa_L \mathscr{L}_p(U_iU_{i-1}\cdots U_2U_1U_2 \cdots U_j )\\
	 & =\kappa_L \mathscr{L}_p(U_iU_{i-1}\cdots U_3U_2U_3 \cdots U_j )\\
	 & \vdots \\
	 & = \kappa_L \mathscr{L}_p(U_iU_{i+1}\cdots U_j ) \\
	 & = \kappa_L \mathscr{L}_p(b_{[i,j]} ). 
\end{aligned}
\end{equation}
Similarly, we  obtain  $\mathscr{L}_p(b_{[-i,j]^{-1}})= \kappa_L \mathscr{L}_p(b_{[i,j]^{-1}})$. On the other hand, we have
\begin{equation}
\begin{aligned}\label{equation new ideal A}
	\mathscr{L}_p(b_{\sigma_{n-1}t_n[-(n-1),n-1]t_n}) & = \mathscr{L}_p(U_{n-1}U_{n} U_{n-1}U_{n-2}\cdots U_1U_0U_1 \cdots U_{n-1}U_{n}) \\
	 & = \kappa_L \mathscr{L}_p(U_{n-1}U_{n} U_{n-1}U_{n-2}\cdots U_2U_1U_2 \cdots U_{n-1}U_{n})\\
	 & =\kappa_L \kappa_L \mathscr{L}_p(U_{n-1}U_{n} U_{n-1}U_{n-2}\cdots U_3U_2U_3 \cdots U_{n-1}U_{n})\\
	 & \vdots \\
	 & = \kappa_L \mathscr{L}_p(U_{n-1}U_{n} U_{n-1}U_n ) \\
	 & = \kappa_L\kappa_R \mathscr{L}_p(U_{n-1}U_n ). 
\end{aligned}
\end{equation}
Similarly, we  obtain  $\mathscr{L}_p(b_{t_n[-(n-1),n-1]t_n})= \kappa_L \mathscr{L}_p(U_nU_{n-1}U_n)$. By the way the bar operator was defined and the four formulas above we conclude that $\mathscr{L}_p(b_{x})=\kappa_L \mathscr{L}_p(\overline{b_{x}})  $, for all $x\in W^c(\tilde C_{n}) - W^{+c}(\tilde C_{n})  $. Therefore, $X\subset Lp$.\\

\noindent
4.  By \eqref{another way to say X},  it is enough to show that both 
$U_s \Phi (b_x)$ and $  \Phi (b_x)U_s $ belong to the $K$-linear space spanned by $X=\Phi (\{ b_y \, |\, y \in   W^c(\tilde C_{n}) - W^{+c}(\tilde C_{n})    \} )$, for any  $ x \in W^c(\tilde C_n)-  W^{+c}(\tilde C_n)$ and any  $0\le s \le n$. This can be achieved by performing a case-by-case analysis. For the sake of brevity, we only  detail the  case 
\begin{equation}\label{the only detailed case}
	x= [  i,n-1 ]   t_{n}( [  -(n-1),n-1 ] \;  t_{n} )^k  ([ f,n-1 ] )^{-1} 
\end{equation}
 with $ 1 \le k $  and $0 \le f < n $, when the multiplication by $U_s$ is done on the right. All the other cases are dealt with similarity.\\
 We remark that a first type element is determined by a triple $(i,k,f)$. In this setting, we set 
 $b(i,k,f):=b_x$ if $x$ is as in \eqref{the only detailed case}. We also define $b(i,0,f)= b_{[i,n-1]t_n(\sigma_{n-1}\sigma_{n-2} \cdots \sigma_f )}$. For instance, with the new notation, the bar operator over monomials indexed by first type elements is given by 
 \begin{equation}
 	\displaystyle \overline{b(i,k,f)} = \left\{  \begin{array}{ll}
 	 b(i,k,-f),	& \mbox{if } f<0;\\ 
 	  	U_nU_{n-1}U_n,	& \mbox{if }  k=1, f=n \mbox{ and }  i = n;  \\
 			\kappa_R b(i,k-1,f), & \mbox{otherwise. } \\  
 	\end{array}  \right. 
 \end{equation}

 We split the proof in several cases in accordance with the relation between $f$ and $s$. \\
 
\noindent
{\textbf{Case A }}$(s=f)$.  We have $\Phi ( b(i,k,f) )U_f = \lambda \Phi ( b(i,k,f))$, where $\lambda = \delta_L$ and $\lambda = \delta$ for $f=0$ and $ 0<f<n$, respectively. In both cases, we are done.\\
 
\noindent
{\textbf{Case B }} $(s=f+1)$. We have 
 	\begin{equation}
		\Phi ( b(i,k,f) )U_{f+1} =   \left\{  \begin{array}{l} 
		\Phi (b(i,k,-1))+\kappa_L \Phi (b(i,k,1))- \kappa_L\kappa_R \Phi (b(i,k-1,-1))	,	 \\
			\Phi ( b(i,k,f+1) ),	   \\
		 	 \kappa_R \Phi (b(i,k,f+1) ),	  \\
 \end{array}  \right.
 \end{equation}
for $f=0$, $0<f<n-1$ and $f=n-1$, respectively. Then,  $ \Phi ( b(i,k,f) )U_{f+1}$ belongs to the  $K$-linear space spanned by $X$.\\

\noindent
{\textbf{Case B1}} $(s=f-1)$. Similar to Case B.\\

\noindent
{\textbf{Case C }} $(f+2\leq s <n)$. A repeated application of the relations in  \eqref{relations temperley lieb five parameters} together with an inductive argument yield 
 \begin{equation}
	b(i,k,f)U_s = \kappa_L^{k-1}\kappa_R^{k} b_{[i,n-1]}(U_n U_{n-1}\cdots U_s)(U_{s-2}\cdots U_1U_0U_1 \cdots U_f ) ,
	\end{equation}
	for all $k\geq 1$. Then, we obtain 
\begin{equation}\label{equation long proof A}
\Phi (b(i,k,f))U_s	= b(i,k,f)U_s- \kappa_L \kappa_R b(i,k-1,f) U_s =0,   
\end{equation}
which gives us the result if $k>1$. If $k=1$ we have $\Phi (b(i,1,f))U_s = \kappa_R \Phi ( b_y)$,
where 
\begin{equation}
y= [i,n-1] t_n( \sigma_{n-1}\sigma_{n-2} \cdots \sigma_s)(\sigma_{s-2}\sigma_{s-3} \cdots \sigma_{f+1})[-f,f].   
\end{equation}
We remark that $y$ is a non-left-positive element. Therefore, $\Phi ( b_y)$ does belong to $X$ and the result follows in this case as well.\\

\noindent
{\textbf{Case C1 }}  $( 0<s\leq f-2)$. Similar to Case C.\\

\noindent
{\textbf{Case D }} $(f+2\leq s =n)$. As in Case C, a repeated application of the relations in  \eqref{relations temperley lieb five parameters} together with an inductive argument yield 
\begin{equation} \label{equation long proof B}
	b(i,k,f)U_n = \left\{ \begin{array}{rl}
	\kappa_L^k\kappa_R^k	b_{[i,n-1]}(U_n U_{n-1}\cdots U_1U_0 \cdots U_f)U_n,
&  \mbox{if } f=0; \\
	\kappa_L^{k-1}\kappa_R^kb_{[i,n-1]}(U_n U_{n-1}\cdots U_1U_0 \cdots U_f)U_n,
	 & \mbox{if } 0<f\leq n-2.
	\end{array}  \right.  \end{equation}
for all $k\geq 1$. Then, $\Phi (b(i,k,f) )U_n=0$ which  implies the result for $k>1$.   We notice that \eqref{equation long proof B}  still holds if $k=0$ and $f=0$, which gives us the result in the case $k=1$ and $f=0$. The remaining case, i.e. $k=1$ and $0<f\leq n-2$, must be split in different cases according to the value of $i$. For $i=n$ we have
$\Phi (b(n,1,f)) U_n = \kappa_R \Phi (b_z)$, where
\begin{equation}
	z= t_n\sigma_{n-1} t_n \sigma_{n-2}\sigma_{n-3} \cdots \sigma_1\sigma_0\sigma_1 \cdots \sigma_f.
\end{equation}
We remark that $z$ is a second type element,  that $z$ is non-left positive and that 
\begin{equation}
\overline{z}= t_n\sigma_{n-1} t_n \sigma_{n-2}\sigma_{n-3} \cdots \sigma_f .	
\end{equation}
All the above gives us the result when $i=n$. We now assume that $i\neq n$. Under this assumption, we notice that there is a generator $U_{n-1}$ on the right of the element $b_{[i,n-1]}$ in  \eqref{equation long proof B}.  By moving the rightmost $U_n$ to the left in \eqref{equation long proof B} and applying the relation $U_{n-1}U_nU_{n-1}U_n= \kappa_R U_{n-1}U_n $ we obtain

\begin{equation}
	b(i,1,f)U_n= \left\{  \begin{array}{ll}
	\kappa_R^2 (U_iU_{i-1} \cdots U_1 U_0U_1 \cdots U_f) U_n	, & \mbox{if } 0<i<n; \\
	\kappa_L\kappa_R^2 (U_0U_1\cdots U_f )U_n,  	& \mbox{if } i=0; \\
	\kappa_L\kappa_R^2 (U_{|i|}U_{|i|-1} \cdots U_1 U_0U_1 \cdots U_f) U_n	,	& \mbox{if } i<0.  \\
	\end{array} 
	\right.
\end{equation}
Finally, we have

\begin{equation}
	\Phi (b(i,1,f))U_n= \left\{  \begin{array}{ll}
0,  	& \mbox{if } i=0; \\
	\kappa_R \Phi (b_u)	,	& \mbox{otherwise } ,  \\
	\end{array} 
	\right.
\end{equation}
where   $u= \sigma_{|i|} \sigma_{|i|-1} \cdots  \sigma_1\sigma_0 \sigma_1 \cdots \sigma_f $ is a non-left-positive element and the result follows. \\

\noindent
{\textbf{Case D1 }} $ ( 0=s\leq f-2)$. Similar to Case D. 
 \end{proof} 

%%%%%%%%%%%%%%%%%%%%%%%%%%

 \begin{coro} The set   
 $\{ b_w  \, |\,  w\in W^{+c}(\tilde C_{n})  \}$ is a $K$-basis  of  $ \tl / Lp $.
 \end{coro}
 
  \begin{proof} We use the same notation for $b_w$ and its image in the quotient.  On the one hand, the proof of the first claim in Theorem \ref{basisLp} reveals that $ X \cup 	\{ b_w  \, |\,  w\in W^{+c}(\tilde C_{n})  \} $ is a  $K$-basis for $ \tl $. On the other hand, Theorem \ref{basisLp} shows that  $X$ is a  $K$-basis of $Lp$. The result follows. 
 \end{proof} 
 
We define $Rp$ to be the ideal of $\tl / Lp$ generated by $U_{n-1}U_nU_{n-1}-\kappa_R U_{n-1}$. We now want to  mimic the previous strategy in order to determine a basis for the ideal $Rp$. The first step in this case is to give a classification of the left-positive elements that are not right-positive. A moment's thought reveals that first type elements are not left-positive. Another moment's thought (perhaps longer) reveals that second type elements are right-positive. Furthermore, it is clear that affine length zero elements are right-positive. So that left-positive elements that are not right-positive must have affine length one. Finally, an inspection of this case in Theorem \ref{FC} gives us the following classification.

 \begin{propos} \label{propo classi left non right}
An element $w\in W^c(\atc)$ belongs to $ W^{+c}(\atc )-W^{c+}(\atc )$  if and only if it satisfies the following conditions. First, $L(w)=1$. Therefore, we can write $w= [i,n-1]t_n v$. In this setting, we have $0\leq i <n$, and one of the following. 
\begin{description}
	\item[(a)]  $0<i<n $ and $ v$ is  of the form \eqref{Stembridge} with $l_1=n-1$ and $l_j=n-j$ or $0\leq l_j<i$ for $2\leq j \leq r$.
	\item[(b)] $i=0$ and $v=  ([h,n-1])^{-1} $  with  $0 \leq h <n$.
	\item[(c)] $i=0$ and $v=   ([z,n-1])^{-1} [0,r_1][0,r_2]\cdots [0,r_m]$  with $0 \leq r_m <...<r_2<r_1<z<n $.
\end{description}
 	Furthermore, in the three cases $\sigma_{n-1}\sigma_n\sigma_{n-1}$ occurs exactly once.
 \end{propos}
 
 \begin{defi} \rm 
For $w \in W^{+c}(\atc )-W^{c+}(\atc )$   we define $\widetilde{w} \in W^{+c+}(\atc )$ by  (we use the same notation and  cases considered in Proposition \ref{propo classi left non right})
 \begin{description}
	\item[(a)] 
	\begin{equation}
		\widetilde{w} =\left\{ \begin{array}{rl}
			[i,l_\alpha ][l_{\alpha +1},g_{\alpha +1}]\ldots [l_r,g_r],  & \mbox{if } i\leq l_\alpha;  \\
			 ([l_\alpha ,i])^{-1}[l_{\alpha +1},g_{\alpha +1}]\ldots [l_r,g_r], & \mbox{if } i> l_\alpha,
		\end{array} \right. 
	\end{equation} 
	where  $\alpha= \max \{ 1\leq j \leq r \, |\, l_j=n-j  \}$. 
	\item[(b)] \begin{equation}
		\widetilde{w} =\left\{ \begin{array}{rl}
			 [0,h],  & \mbox{if }  h>0;  \\
			  \sigma_0\sigma_1\sigma_0  , & \mbox{if } h=0.
		\end{array} \right.
	\end{equation} 
	\item[(c)] \begin{equation}
	\widetilde{w} = [0,z][0,r_1]\ldots [0,r_m].	
	\end{equation}
\end{description}
Finally, we define a $K$-linear map $\widetilde \space : \mbox{Span}_K \{b_w \, |\,  w\in W^{+c}(\atc )-W^{c+}(\atc )  \} \rightarrow \tl / Lp$ determined by $\widetilde{b_w} =  b_{\widetilde{w}}$.
 \end{defi}

  \begin{teo} The set   
 $\{ b_w  \, |\, w\in W^{+c+}(\tilde C_{n})  \}$    is a $K$-basis  of  $ (\tl /Lp )/Rp$.
 Therefore, the set of positive fully commutative elements indexes a monomial basis for $\twob$.
 \end{teo}
 
  \begin{proof} 
 Consider the $K$-linear map $\Phi': \tl /Lp  \rightarrow \tl /Lp $ determined by 
 \begin{equation}
 	\Phi' (b_w) = \left\{  \begin{array}{ll}
	  b_w,	& \mbox{if } w \in   W^{+c+}(\tilde C_{n});  \\
	  b_w-\kappa_R \widetilde{b_{w}}, & \mbox{if } w\in  W^{+c}(\tilde C_{n}) - W^{c+}(\tilde C_{n}) .
	  \end{array} \right. 
 \end{equation}
 Since $l(w)>l(\widetilde{w})$ we see that $\Phi'$ is an isomorphism of $K$-vector spaces. By the same arguments as the ones used in Theorem \ref{basisLp} we obtain that the set 
$\{	\Phi' (b_w ) \, |\,  w \in W^{+c}(\tilde C_{n}) - W^{c+}(\tilde C_{n})  \}$
 is a $K$-basis for $Rp$. Therefore, $\{ b_w  \, |\, w\in W^{+c+}(\tilde C_{n})  \}$    is a $K$-basis  of  $ (\tl /Lp )/Rp$. Finally, we notice that 
 \begin{equation}
 	(\tl /Lp )/Rp \simeq \twob.  
 \end{equation}
 
%  Again we use the same notation for $b_w$ and its image in the quotient.  
%  Using the Dynkin automorphism and exchanging the roles of $\kappa_R$ and $\kappa_L$, we can define  a map 
%  \begin{equation}
%  	\widetilde\space :    \mbox{Span}_K  \{ b_x \, | \,  x\in W^c(\tilde C_{n}) - W^{c+}(\tilde C_{n}) \} \longrightarrow \tl 
%  \end{equation}
%  and obtain as a basis of the ideal $Rp$ the set  
%  \begin{equation}
%  Y :=	\{ b_x-\kappa_R  \widetilde{ b_{ x}}   \, |\, x\in W^c(\tilde C_{n}) - W^{c+}(\tilde C_{n})  \}.
%  \end{equation}
%  Then,  the  set $X\cup Y$ spans the ideal $LpR= Lp+Rp$  and the supplement has  $\{ b_w  \, |\,  w\in W^{+c+}(\tilde C_{n})  \}$ as a basis. 
 \end{proof} 
 
% \begin{rem} \rm
%    
%    In the $K$-algebra $TL\tilde{C}_{n}(q)/Lp$ the set 
%    $$\begin{aligned} 
%     \{U_{i} \dots U_{n-2} U_{n-1}& U_n U_{n-1} U_{n-2} \dots U_{j} -\kappa_R  U_{i} \dots U_{j};  0 \le i,j \le n-1 , \ (i,j)\ne (0,0)  \}   \\ & \cup    
%     \{U_0 \dots U_{n-2}U_{n-1}U_n U_{n-1} U_{n-2} \dots U_0 -\kappa_R  \kappa_L U_0 U_{1}  \}
%    \end{aligned} $$
%     is a $K$-basis for the ideal generated by   $U_{n-1} U_n U_{n-1} - \kappa_R U_{n-1} $.
%\end{rem}  
   
%{ \red{Some comments on Theorem 3.11 and Remark 3.12:  I am not convinced by your argument. The set $X\cup Y$ spans $LpR$, but it is not a basis. Then, why $\{ b_w  \, |\,  w\in W^{+c+}(\tilde C_{n})  \}$ is a basis for the quotient? Surely, this deserves some lines of explanation. Especially considering that this is perhaps the most important theorem of the article. I would also avoid the use of the word "supplement", What do you mean by the supplement of a vector subspace? I also think that the Remark is completely unnecessary. If the remark is needed in the proof of the Theorem, then add it inside the proof. Otherwise, why should we care about the algebra $TL\tilde{C}_{n}(q)/Lp$. Anyway, I believe this Remark must be eliminated. }}  \\ 

   We conclude this section by providing a classification for positive fully commutative elements. If we look back at Theorem \ref{FC}, we notice that the first type elements cannot be positive, while we can give some conditions on the normal form  of  second type elements and elements of affine length one, so that they would be positive. We sum up in the notations of  Theorem \ref{FC}. Let $w \in W^c(\atc )$ be a positive element. 
If $L(w)   \ge 2$,  it has the following form, for non negative integers $p$ and $k$:  
\begin{equation}\label{secondtypepositive}
\begin{aligned}
w &=   [  i_1,n-1 ]  \; t_{n}  [  i_2,n-1 ]  \; t_{n} \dots  [  i_p,n-1 ]  \;  t_{n}    ( [  0,n-1 ]  \; t_{n} )^k  
\;    w_r   \quad  \text{if } p>0 ,  \\  
w &=       ( [  0,n-1 ]  \; t_{n} )^k  
\;    w_r  \quad   \text{if } p=0, 
 \end{aligned}
 \end{equation}
  with $w_r \in W^c(B_{n})$   and  
\begin{itemize}
\item if $k > 0$: $w_r=1$ or  $w_r=[ 0,r_1 ][ 0,r_2 ] \dots [ 0,r_u ]$ with $0 \le r_u < \dots < r_1<n$ ; 
\item if $p > 0$: 
 $ \  
n \ge i_1>... > i_{p-1} > i_p >0    \  $; 
\item if $k =0$ :  $w_r$ is a positive element in $W^c(B_n)$ with $ l_1 < i_p$. 
\end{itemize}

Now suppose that $L(w) = 1$, then $w$  has a reduced expression of the form $[  i,n-1 ]    t_{n}       v $ where
\begin{itemize}
\item  if $0 < i \le n$ then $v$ is a positive element in $W^c(B_n)$ with $l_1<i$;
\item  if $  i=0 $ then  $v=1$ or $v$ is equal to $[ 0,r_1 ][ 0,r_2 ] \dots [ 0,r_m ]$ for $0 \le r_m < \dots r_2 < r_1 <  n$.\\
\end{itemize}

\noindent
In order to provide a uniform description of the positive fully commutative elements we need a small modification of the notation introduced in \eqref{first parentheses}. To this end, we define 
\begin{equation}
	 \langle i,j \rangle = \sigma_i \sigma_{i+1} \dots \sigma_j  
\end{equation}
for $0\leq i \leq j\leq n$. With this notation at hand, we can reformulate the above description of positive fully commutative elements as follows.

\begin{propos}  \label{proposition normal form posit}
	Let $n>1$ be an integer. Let $w$ be a positive element in $  W^c(\tilde C_{n})$ other than the unit. Then  there exists a positive integer  $k$ such that
	\begin{equation} \label{eq normal form blob}
		w=\langle l_1,r_1 \rangle \langle   l_2, r_2 \rangle \ldots \langle l_k, r_k \rangle, 
	\end{equation}
	for some integers $l_i$ and $r_i$ such that
	\begin{enumerate}
		\item $n \geq l_1\geq l_2 \geq \ldots \geq l_k \geq 0$;
		\item $n\geq r_1\geq r_2 \geq \ldots \geq r_k \geq 0$;
		\item $l_i\leq r_i$;
		\item If $l_{i+1} = l_i$ then $l_i=0$;
		\item If $r_{i+1}=r_i$ then $r_i=n$.       
	\end{enumerate}   
Conversely, every $w$ of the form \eqref{eq normal form blob} is in $\pos$.
\end{propos}

We call the elements $\langle l_s,r_s \rangle$, $ 1 \leq s \leq k$,  the {\em rigid blocks} of $w$.

%%%%%%%%%%%%%%%%%%%%%%%%%%%%%%%%%%%%%%%%%%%%%%%%%%%%%%%%%%%%%%%%%%%%%%%%%%%%%%%%%%%%%%%%%%
\section{Blobbed fully commutative elements and the symplectic blob algebra.}   
\label{section blobbed and symplectic}

In the previous section we found a monomial basis for the two-boundary Temperley-Lieb algebra which was indexed by positive fully commutative elements. In this section we find  a monomial basis for the symplectic blob algebra which is indexed by blobbed fully commutative elements. Let us begin by defining these elements.

\begin{defi} \rm  \label{defi blobbed fully commutative elements}
	An element  $w\in \pos$ is called $I$-blobbed (resp. $J$-blobbed) if it avoids $IJI$ (resp. $JIJ$), where
	\begin{equation}
		I = \left\{  \begin{array}{ll}
		\sigma_1\sigma_3 \cdots \sigma_n 	& \mbox{if } n \mbox{ is odd;}  \\
		  \sigma_1\sigma_3 \cdots \sigma_{n-1} 	& \mbox{if } n \mbox{ is even,}  \\  
		\end{array}  \right. 
		\qquad
		\mbox{and} 
		\qquad
		J = \left\{  \begin{array}{ll}
		\sigma_0\sigma_2 \cdots \sigma_{n-1} 	& \mbox{if } n \mbox{ is odd;}  \\
		  \sigma_0\sigma_2 \cdots \sigma_{n} 	& \mbox{if } n \mbox{ is even.}  \\  
		\end{array}  \right.
	\end{equation}
	Elements which are both $I$-blobbed and $J$-blobbed will be called blobbed. We denote by $~^I\!\blobbed$ (resp. $~^J\!\blobbed$) the set of all $I$-blobbed  (resp. $J$-blobbed) elements in $W^{+c+}(\atc)$. Finally, we denote by $\blobbed$ the set of all blobbed elements in $W^{+c+}(\atc)$. 
	\end{defi}

 Our goal in this section is to demonstrate the following result.
	
\begin{teo}  \label{theorem monomial basis for symplectic}
	The set $\{ b_w \, | \,  w\in \blobbed  \}$ is a $K$-basis for the symplectic blob algebra $\simp $.
\end{teo}

We  utilize the same strategy as the one used in the previous section. For this reason and for the sake of brevity we only provide a sketch of the proof of Theorem \ref{theorem monomial basis for symplectic}. The rest of this section is devoted to this aim. The first step in order to reach our goal is to provide a normal form for positive elements in which we can see the occurrences of $IJI$ and $JIJ$. We stress that the normal form for positive elements given in Proposition \ref{proposition normal form posit} is not well suited for this purpose. It is convenient to introduce a graphical interpretation for the elements of $\pos $.

 \begin{defi} \rm \label{defi grid}
	 Let $w=\langle l_1,r_1 \rangle \langle l_2, r_2\rangle \ldots \langle l_k, r_k\rangle \in \pos$.  The \emph{grid} of the  element $w$ is the set
	 \begin{equation}
	 	G(w):= \{ (i,j)\in \mathbb{Z}^2 \, |\,  1\leq i \leq k; \,     l_i \leq j \leq r_i  \},
	 \end{equation}
	 up to horizontal translation. 
\end{defi}

\begin{exa} \rm
Let $w= \langle 7,8 \rangle \langle  4,8 \rangle \langle  3,7 \rangle \langle 1,4\rangle \langle 0,1 \rangle \langle 0,0\rangle \in W^{+c+}(\tilde{C}_8)$.  The grid of $w$ can be described graphically as 
\begin{equation}
	\scalebox{2}{\Heap}
\end{equation}
\end{exa}
From now on we do not distinguish between a grid and its graphical interpretation. This geometric presentation gives rise to another normal form of positive fully commutative elements, in which our point of viewing a fully commutative element is to determine the reduced expression of maximal blocks of commuting generators. Indeed, moving a line with slope  $-2$  from left to right through the grid, we see that each intersection of the grid with such a line is made of commuting generators. 
We call for short {\it oblique of }$G(w)$ (or of $w$) such an intersection.  We  identify a point $(i,j)$ with the generator $\sigma_{j}$. This allows us to identify an oblique with an element of $W(\atc) $ by considering the product of the generators involved in the oblique. We recall that the generators involved in a given oblique commute so it is not necessary to specify an order for the  aforementioned product.  Let us illustrate the above with an example. 

\begin{exa}\rm \label{exa grid}
Let $w= \langle 7,8 \rangle \langle  4,8 \rangle \langle  3,7 \rangle \langle 1,4\rangle \langle 0,1 \rangle \langle 0,0\rangle \in W^{+c+}(\tilde{C}_8)$. 
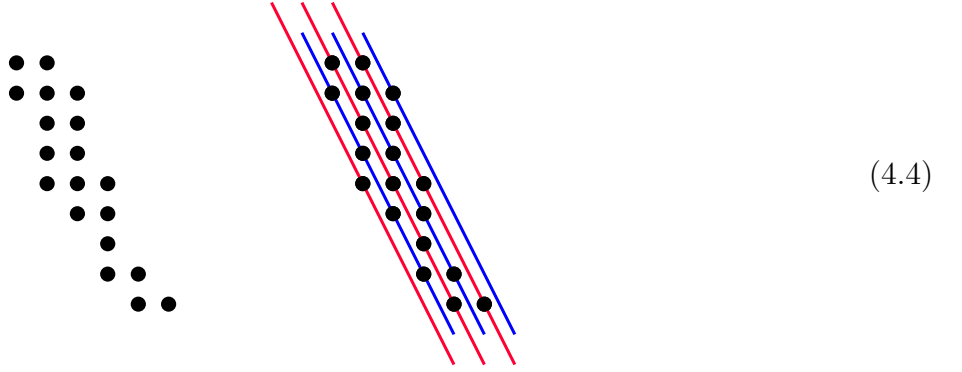
\begin{equation}
	\scalebox{2}{\Heap}  \qquad 	\oblique
\end{equation}	
In this case we have six obliques. If we label them from left to right as $O_1,\ldots , O_6$ then the corresponding elements in $W(\atc)$ are
\begin{equation}
O_1= \sigma_4; \quad O_2= O_4	=\sigma_1\sigma_3\sigma_5\sigma_7=I;  \quad O_3= \sigma_0\sigma_2\sigma_4\sigma_6\sigma_8=J; \quad O_5 = \sigma_0\sigma_4\sigma_6\sigma_8; \quad O_6= \sigma_7 .
\end{equation}
\end{exa}

 \begin{propos}\label{obliques}  Let $w\in W^{+c+}(\atc)$. Then,  $w$ is the product of the obliques of $G(w)$ taken from left to right.  Furthermore, if we write the generators involved in each oblique in increasing order then we obtain a well-defined reduced expression for $w$, which we call the oblique form of $w$.
 \end{propos}    
 
 \begin{proof}
 	The result follows by induction on the number of obliques of $G(w)$ once we notice that any generator involved in the leftmost oblique  can be pushed to the left of the expression (\ref{eq normal form blob}) for $w$ and that the element obtained from $w$ by eliminating its leftmost oblique is  positive. 
 \end{proof}
 
 \begin{exa}\rm
 	If $w$ is as in Example \ref{exa grid} then its oblique form is given by 
\begin{equation}
 \sigma_4\sigma_1\sigma_3\sigma_5\sigma_7\sigma_0\sigma_2\sigma_4\sigma_6\sigma_8\sigma_1\sigma_3\sigma_5
 \sigma_7\sigma_0\sigma_4\sigma_6\sigma_8\sigma_7 	.
\end{equation}
 \end{exa}
 
  Now let $(W,S)$ be a Coxeter system and  let $w$ be a FC element in $W$.  By Matsumoto's theorem and the definition of full  commutativity, we know that the number of occurrences of a generator $s\in S$ in  $w$ is well-defined. Moreover, suppose that $s$ occurs $k$ times in $w$. Given a reduced expression $\underline{w}$ of $w$, let us order the  occurrences of $s$ in $\underline{w}$ from left to right, for example $s^i$ for $1\leq i \leq k$. Again by Matsumoto's theorem and full commutativity, the relative positions of $s^1,  s^2, \dots , s^k$ are independent of the reduced expression. In other terms we can mark the occurrences of any simple reflexion in some reduced expression of $w$. We choose to mark by coloring in what follows.

 \begin{lem} \label{lemma Sadek to prove the only if}
 Let $(W,S)$ be a Coxeter system of graph $\Gamma$, let $s,t,u\in S$ be such that the subgraph of  $\Gamma$ determined by $s,t,u$ is connected and $s$ commutes with $u$, suppose  $s,t,u \in Supp(w)$ for $w\in W^c$ then the following holds: ${\blue{t}}$ occurs between ${\red{s}}$ and ${\green{u}}$ in a reduced expression of $w$ if and only if ${\blue{t}}$ occurs between ${\red{s}}$ and ${\green{u}}$ in any reduced expression of $w$.
 \end{lem}

\begin{demo}
	The result follows directly from Matsumoto's theorem and the definition of full commutativity.
\end{demo}

 %\begin{lem}  \label{lemma Sadek to prove the only if}
 %Let $(W,S)$ be an arbitrary Coxeter system. We assumme that $s,s',s''\in S$ and that $s$ commutes with $s''$(or $s=s''$), $s$ does not commute with $s'$ and $s'$ does not commute with $s''$. Let $w$ be a FC element in $W$. Suppose that $s$ and $s''$ occur $k$ and $k''$ times in $w$, respectively. Let $\underline{w}$  and $\underline{\underline{w}}$ be reduced expressions for $w$. Then,  $s'$ occurs in between $s_{\underline{w}}(i)$ and $s''_{\underline{\underline{w}}}(i'')$ in $\underline{w}$ if and only if $s'$  occurs in between $s_{\underline{w}}(i)$ and $s''_{\underline{\underline{w}}}(i'')$ in $\underline{\underline{w}}$, for all $1\leq i \leq k$  and $1\leq i'' \leq k''$. 	
 %\end{lem}

 %\begin{coro}  \label{coro blobbed only if part}
 %Let $w\in W^{+c+}(\atc)$. If $G(w)$  contains  $G(I)$  (resp. $G(J)$, resp. $G(IJI)$, resp. $G(JIJ)$) then $w$ contains $I$ (resp. $J$, resp. $IJI$, resp. $JIJ$).
 %\end{coro}    
 %\begin{proof}
 %	  The result follows immediately from  Proposition \ref{obliques} once we remember that $G(w)$ was defined up to horizontal translation.
 %	  \end{proof}

  \begin{propos}  \label{lemma blobbed if part}
 Let $w\in W^{+c+}(\atc)$, then $w$ contains $I$ (resp. $J$, resp. $IJI$, resp. $JIJ$) if and only if  $G(w)$  contains  $G(I)$  (resp. $G(J)$, resp. $G(IJI)$, resp. $G(JIJ)$).
  \end{propos}    
 \begin{proof}
  The ``if'' part follows immediately from  Proposition \ref{obliques} once we remember that $G(w)$ was defined up to horizontal translation.\\ 
  For proving the  ``only if'' part we assume that $n$ is odd (the case $n$  even is treated similarly). Suppose that $w$ contains $I$, this gives a reduced expression  $\underline{w}$ of $w$ in which we can see $I=\sigma_n \sigma_{n-2}\cdots \sigma_3\sigma_1$. By marking the letters involved in $I$ we obtain  
  \begin{equation}
  	\underline{w} = u {\red{\sigma_n}} {\red{\sigma_{n-2}}} \dots  {\red{\sigma_3}}  {\red{\sigma_1}}   v, 
  \end{equation}
  where $l(w)= l(u)+l(v) + (n+1)/2$. We consider the normal form of $w$ given by Proposition \ref{proposition normal form posit}. In this form, 
  the rigid block  of some marked ${\red{s}} \in S$ is uniquely determined by ${\red{s}}$, it is to be denoted by $[{\red{s}}] $. 
  In the notation of Proposition \ref{proposition normal form posit}, we say that the rigid block   $\langle l_s,r_s \rangle$
   is directly on the left (resp. right)  of the rigid block    $\langle l_t,r_t \rangle$ if $t=s+1$ (resp. $t=s-1$).  
  
  Now $[{\red{\sigma_n}}]$ and $[{\red{\sigma_{n-2}}}]$ are consecutive. Otherwise, any rigid block between those two blocks must contain $\sigma_{n-1}$, which contradicts  - by Lemma \ref{lemma Sadek to prove the only if} -  the fact that in $\underline{w}$ the generator  $\sigma_{n-1}$ does not occur between $ {\red{\sigma_n}}$ and $ {\red{\sigma_{n-2}}}$. Moreover, $[{\red{\sigma_n}}] $ is on the left of $[{\red{\sigma_{n-2}}}]$, otherwise $\sigma_{n-1}$ occurs between them by the normal form of Proposition \ref{proposition normal form posit}. We sum up: $[{\red{\sigma_n}}] $ is directly on the left of $[{\red{\sigma_{n-2}}}]$. We apply the same argument to $ {\red{\sigma_{n-2}}}$ and $ {\red{\sigma_{n-4}}}$ and so on,  arriving to ${\red{\sigma_{3}}}$ and ${\red{\sigma_{1}}}$. That means: the $\frac{ n+1}{2}$ rigid blocks $[{\red{\sigma_n}}] [{\red{\sigma_{n-2}}}] \dots [{\red{\sigma_3}}] [{\red{\sigma_1}}]$ are consecutive in this order.  Consequently,   $G(w)$  contains  $G(I)$. The same argument applies replacing $I$ by $J$.\\ 
  
  Now suppose that $w$ contains $IJI$. Then,  there exists  a reduced expression $\underline{w}$ of $w$ in which we can see $IJI$. By marking the letters involved in $IJI$ we obtain
  \begin{equation}
  	\underline{w} = u  {\red{\sigma_n}} \dots  {\red{\sigma_3}} {\red{\sigma_1}}  {\blue{\sigma_{n-1}}}  \dots  {\blue{\sigma_{2}}}  {\blue{\sigma_{0}}} {\green{\sigma_{n}}} \dots  {\green{\sigma_{3}}}  {\green{\sigma_{1}}} v. 
  \end{equation} 
  where $l(w)= l(u)+l(v) + 3(n+1)/2$. Now consider the rigid block $[ {\red{\sigma_n}}]$, it is directly on the left of $[{\green{\sigma_{n}}}]$ due to the conditions in Proposition \ref{proposition normal form posit}. Moreover, it is directly on the left of   $[ {\red{\sigma_{n-2}}}]$ by the argument above and finally it is directly on the left of $[ {\blue{\sigma_{n-1}}}]$, because the position of $ {\blue{\sigma_{n-1}}}$ between  $ {\red{\sigma_n}}$ and ${\green{\sigma_{n}}}$ is independent of the reduced expression. We sum up: $[{\green{\sigma_{n}}}]= [{ \blue{\sigma_{n-1}}}] =[ {\red{\sigma_{n-2}}}]$ and this rigid block  is directly on the right of  $[ {\red{\sigma_n}}]$. By applying the very same argument $(n-3)/2$ times we obtain that 
  $[{\green{\sigma_{i}}}]= [ {\blue{\sigma_{i-1}}}] =[{\red{\sigma_{i-2}}}]$  and  that this block  is directly on the right of  $[{\red{\sigma_{i}}}]$ for $i$ odd, $ 3 \leq i \leq n$. Finally, by considering the relative positions of $\red{\sigma_1}$, $\green{\sigma_1}$ and $\blue{\sigma_0} $ we conclude that  $[{\green{\sigma_{1}}}]= [ {\blue{\sigma_{0}}}]$ and that this block is directly on the right of $[{\red{\sigma_1}}]$.  All the above implies that 
  $G(w)$ contains $G(IJI)$.  The same argument applies replacing $IJI$ by $JIJ$.
  \end{proof}

\begin{rem}\rm
	Proposition \ref{lemma blobbed if part}  states  a strong property that does not hold in general.  For instance let $w= \sigma_1 \sigma_2 \sigma_3 \sigma_0 \sigma_1 \sigma_2$ and   $w'= \sigma_2 \sigma_1$. Then, $G(w') \subset G(w)$ as is illustrated in \eqref{eq remark contain not contain}. However,  $w$ does not contain $w'$. 
\end{rem}

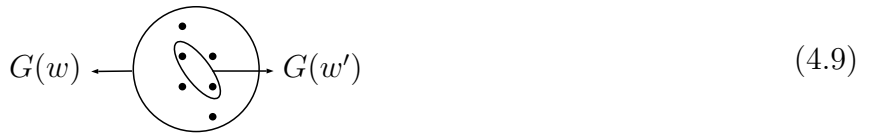
\begin{equation}\label{eq remark contain not contain}
	\remarknotcontained
\end{equation}

 \begin{lem} \label{lemma two I forces all inthe middle}
 	Let $w\in \pos$. Let $O_1, \ldots , O_m$ be the obliques of $G(w)$ labelled from left to right. If $O_i=O_{i+2j}= I$  for some positive integers $i $ and $j$ then 
 	\begin{equation}  \label{obliques lema IJIJIJI A}
 		O_{i+1}=O_{i+3}=\ldots = O_{i+2j-1}=J \quad \mbox{ and }\quad  O_{i+2}=O_{i+4}=\ldots = O_{i+2j-2}=I.
 	\end{equation}
 	Similarly, if  $O_i=O_{i+2j}= J$ for some positive integers $i $ and $j$ then 
 	\begin{equation}  \label{obliques lema IJIJIJI B}
 		O_{i+1}=O_{i+3}=\ldots = O_{i+2j-1}=I \quad \mbox{ and }\quad  O_{i}=O_{i+3}=\ldots = O_{i+2j-1}=J.
 	\end{equation}
 \end{lem}
 
 \begin{proof}
 We only prove \eqref{obliques lema IJIJIJI A}. Equation \eqref{obliques lema IJIJIJI B} is treated similarly. If we have two $G(I)$'s in a grid of a positive element then, in order to satisfy the conditions in Proposition \ref{proposition normal form posit}, we are forced to locate all the points ``in between''of  the two occurrences of $G(I)$, as illustrated in \eqref{obliques IJIJIJIJI} when $n=14$ and $j=7$. The result follows.  
 \end{proof}

 \begin{equation}\label{obliques IJIJIJIJI}
 	\scalebox{.75}{\obliquesIJIJIJIJIJI }
 \end{equation}	
  
 \begin{coro} \label{coro oblique form}
 	Let $w\in \pos - ~^I\!\blobbed$. Then, the oblique form of $w$ is given by 
 	\begin{equation} \label{oblique form with the IJ}
 		O_1\ldots O_r (IJ)^kI O_1'\ldots O_s',
 	\end{equation}
 	for some integer $k\geq 1$ and some obliques $O_i$ and $O_i'$ such that $O_i\neq I$, $O_i'\neq J$, $O_i'\neq I$ for all $i$ and $O_i\neq J$ for $i\neq r$.
 \end{coro}

\begin{demo}
	The result follows by a direct application of Proposition \ref{obliques} and Lemma \ref{lemma two I forces all inthe middle}.
\end{demo} 
 
 \medskip
We are now ready to reapply the strategy used in the previous section in order to prove Theorem \ref{theorem monomial basis for symplectic}. We recall from Section \ref{section algebras} that $SB_n(q,Q,k) $ is the quotient of $\twob $ by the ideal generated by the elements $IJI-\kappa I $ and $  JIJ -\kappa J$. We define $Ib$ to be the ideal of $ \twob$ generated by $IJI-\kappa I$. We are going to find a $K$-basis for this
 ideal. To this end we define a map $ \bar\space : \pos - ~^I\!\blobbed \longrightarrow \pos  $ given by
 \begin{equation}
 	\bar{w } = O_1\ldots O_r (IJ)^{k-1}I O_1'\ldots O_s',
 \end{equation}
 where $w$ is as in \eqref{oblique form with the IJ}. We define a $K$-linear map 
 \begin{equation}
 	\bar \space : \mbox{Span}_K \{  b_w \, | \, w\in \pos - ~^I\!\blobbed    \} \longrightarrow \twob
 \end{equation}
 determined by $\overline{b_w} = b_{\bar{w}}$.  By using the same methods as the ones used  in the proof of Theorem \ref{basisLp}, we can see that  the set
 \begin{equation}
 Y:= \{ b_w - \kappa 	\overline{b_{ {w} }} \, | \, w\in \pos - ~^I\!\blobbed    \} 	
 \end{equation}
is a $K$-basis for the ideal $Ib$. Furthermore, one can see that the set 
\begin{equation}
	Y\cup \{ b_w \, | \, w\in  ~^I\!\blobbed   \}
\end{equation}
 is a $K$-basis for $\twob$ and, therefore, $\{ b_w \, | \, w\in  ~^I\!\blobbed   \}$ is a $K$-basis for the quotient $\twob / Ib $. 
 
 Let $Jb$ be the ideal of $\twob / Ib $ generated by the element $JIJ-\kappa J$.  We notice that the oblique form for an element $I$-blobbed which is not $J$-blobbed reduces to 
 \begin{equation} \label{eq to define widetilde blobbed}
 	O_1\ldots O_r JIJ O_1'\ldots O_s',
 \end{equation}
 for some obliques $O_i$ and $O_i'$ different from $I $ and $J$. We  define a map 
 \begin{equation}
 	\widetilde\space:  ~^I\!\blobbed - ~^J\!\blobbed  \longrightarrow \blobbed
 \end{equation}
 given by 
 \begin{equation}
 	\widetilde{w} = 	O_1\ldots O_r J O_1'\ldots O_s',
 \end{equation}
 if  $w$ is as in \eqref{eq to define widetilde blobbed}. Then, we extend this map to a $K$-linear map 
 \begin{equation}
 	\widetilde \space : \mbox{Span}_K \{  b_w \, | \, w\in  ~^I\!\blobbed - ~^J\!\blobbed    \} \longrightarrow \twob / Ib
 \end{equation}
 determined by $\widetilde{b_w}=b_{\widetilde{w}} $. As before, we can conclude that the set 
 \begin{equation}
 	Z:= \{  b_w-\kappa \widetilde{b_{{w}}}  \, | \,  w\in  ~^I\!\blobbed - ~^J\!\blobbed    \}
 \end{equation}
 is a $K$-basis for $ Jb$ and that the set $Z\cup \{ b_w \, |\, w\in \blobbed  \}$
 is a $K$-basis of $\twob / Ib$. Therefore, the set $ \{ b_w \, |\, w\in \blobbed  \}$ is a $K$-basis of 
 \begin{equation}
 	(\twob / Ib )/ Jb \simeq SB_n(q,Q,\kappa ).
 \end{equation} 
 This finishes the sketch of proof of Theorem \ref{theorem monomial basis for symplectic}.

 \begin{rem}\rm
  The grid presentation used in this work can be extended to express any element in $W^c(\tilde{C}_{n})$. More generally, all fully commutative elements of the three finite and four affine infinite families of Coxeter groups have a grid presentation coming from their normal forms. So we can use the oblique point of view to obtain a normal form of maximal blocks of commuting generators.    
 \end{rem}

\section{Catalan combinatorics}  
\label{section catalan combinatorics}
In this section we introduce and study a variation of Forder's Catalan triangle \cite{forder1961some}, which we call the \emph{blobbed Catalan triangle}. It might be unclear for the reader why we would care about such a triangle.  However, we point out that all the information needed to perform the enumeration of  positive and blobbed fully commutative elements in the forthcoming sections is encoded in this triangle. Let us start by recalling the construction of Forder's Catalan triangle\footnote{It is worth mentioning that in the literature there is a confusion between Forder's triangle and Shapiro's triangle \cite{shapiro1976catalan}. It is common to find authors who work with Forder's triangle but who cite Shapiro's work. But they are different triangles!}.

\begin{defi}\rm \label{defin classical catalan triangle}
The Catalan triangle is the infinite matrix $(c_{i,j})_{-1\leq i,j}$ defined recursively as follows:
\begin{enumerate}
    \item $c_{-1,-1}=1$; 
    \item $c_{-1,j}=0$, for all $j\geq 0$; 
	\item $c_{i,-1}=0$, for all $i\geq 0$;
	\item $c_{i,j}= c_{i-1,j-1}+c_{i-1,j+1}$, for all $i,j\geq 0 $. 
\end{enumerate} 	
\end{defi}

The left-hand side of (\ref{Example classical Catalan}) shows a  Catalan triangle truncated at row $10$ and at column $10$. We have omitted  the $(-1)$-th row and the $(-1)$-th column, as well as all zeroes. The numbers appearing in the Catalan triangle have the following combinatorial interpretation. Consider the  infinite directed graph obtained from the Catalan triangle by  replacing positive numbers by vertices and with arrows as illustrated in the right-hand side of (\ref{Example classical Catalan}). Then, for a given vertex $V$ the number of directed paths from the highest vertex to $V$ is equal to the number replaced by $V$.   In particular, the numbers occurring  in the $0$-th column coincide with the Catalan numbers, since in this case the relevant paths are Dick paths. In formulas, we have $C_n = c_{2n,0}$. 
  
\begin{equation} \label{Example classical Catalan}
	\scalemath{.8}{
\begin{array}{ccccccccccccc}
1 & & & & & & & & & & & & 	  \\
 & 1 & & & & & & & & & & & 	\\
  1&  &1 & & & & & & & & & &  \\
    &  2 & & 1  & & & & & & & & &  \\
       2 &   &3  &   & 1 & & & & & & & &  \\
           &  5 & & 4  & & 1 & & & & & & &  \\
                 5&   &9 &   & 5 & &1 & & & & & &  \\
           &  14 & & 14  & & 6 & & 1 & & & & &  \\
                            14&   &28 &   & 20 & &7 & & 1 & & & &  \\
           &  42 & & 48  & & 27 & & 8 & & 1 & & &  \\
           42&   &90 &   & 75 & &35 & & 9 & & 1& &  \\
\end{array}
}
\Catalan
\end{equation}

\begin{defi} \rm \label{defi Catalan triangle}
The blobbed Catalan triangle is the infinite matrix $(C_{i,j})_{-1\leq i, j}$ defined recursively as follows:
\begin{description}
	\item[\rm (1)]  For  $i=-1$ or $i=0$ and for all $j$ we have
$$  C_{i,j}=\left\{  \begin{array}{ll}
	1, & \mbox{ if } i+j\equiv 0\mod 2;\\
	0, & \mbox{ otherwise.}
\end{array} \right.   $$
	\item[\rm (2)]  $C_{i,-1}=0$, for all $i \geq 1$;
	\item[\rm (3)]  $C_{i,j} = C_{i-1,j-1}+C_{i-1,j+1}$, for all $i\geq 1$ and all $j\geq 0$. 
\end{description}  	
\end{defi}

The blobbed Catalan triangle truncated at column $8$ and at row $8$ is shown in (\ref{Example blobbed Catalan}). We have omitted the entire $(-1)$-th column, as well as all zeroes.  
\begin{equation} \label{Example blobbed Catalan}
\scalemath{1}{
\begin{array}{ccccccccccccccccc}
 & 1 & & 1& & 1 & & 1&   \\
 1 & &1  & &1 & &1 & & 1  \\
 &2  & & 2& & 2 & & 2&   \\
 2 & & 4 & & 4& & 4& & 4 	\\
  &6  & & 8& & 8 & & 8&   \\
   6 & & 14 & & 16& & 16& & 16 	\\
     &20  & & 30& & 32 & & 32&   \\
       20 & & 50 & & 62& & 64& & 64 	\\
     &70  & & 112& & 126 & & 128&   \\
        70 & & 182 & & 238& & 254& & 256 
\end{array}
}
\end{equation}

The blobbed Catalan triangle is not a triangle but a rectangle! The choice of the name triangle instead of rectangle is justified by the fact that our primary (but not exclusive) interest is in the numbers that appear under the main diagonal, since the numbers located above this diagonal are only powers of 2. Like for the numbers occurring in the classical Catalan triangle, the  numbers located under the main diagonal in the blobbed Catalan triangle have a combinatorial interpretation as follows. First, we erase all the zeroes and all the numbers appearing above the main diagonal, as illustrated in the left-hand side of (\ref{Example blobbed Catalan two}). Then, as before, we construct a directed graph by replacing numbers by vertices. But this time, we add two edges between each pair of consecutive vertices  located in the ``hypotenuse'' of the triangle, as is shown in the right-hand side of (\ref{Example blobbed Catalan two}). Then, for a given vertex $V$, the number of directed paths from the highest vertex to $V$ is  the number replaced by $V$. 

\begin{equation} \label{Example blobbed Catalan two}
\scalemath{.7}{
\begin{array}{ccccccccccccccccc}

 1 & &  & & & & & &   \\
 &2  & & & &  & & &   \\
 2 & & 4 & & & & & &  	\\
  &6  & & 8& &  & & &   \\
   6 & & 14 & & 16& & & & 	\\
     &20  & & 30& & 32 & & &   \\
       20 & & 50 & & 62& & 64& &  	\\
     &70  & & 112& & 126 & & 128&   \\
        70 & & 182 & & 238& & 254& & 256 
\end{array}
}
\qquad
\CatalanTwo
\end{equation}

\begin{lem} \label{lemma properties of Cs}
The numbers $C_{i,j}$ satisfy:
\begin{enumerate}
	\item $C_{j,0}=C_{j-1,1} $, for all $j\geq 0$.
	\item $ C_{i,j} = \sum_{k=0}^j C_{i-1-k, j+1-k}  $, for all $i,j\geq 1$.
	\item  $C_{i,j} =\sum_{k=0}^i C_{i-1-k, j-1+k}  $, for all $i,j\geq 1$.
\end{enumerate}
\end{lem}

\begin{demo}
	All the statements  follow immediately from Definition \ref{defi Catalan triangle}.
\end{demo}

The following result provides a closed formula for the numbers occurring below the main diagonal of the blobbed Catalan triangle. 

\begin{teo} \label{teo closed formula for Cs}
Let $i$ and $j$ be integers with the same parity such that $0\leq j \leq i$. Then, we have
	\begin{equation} \label{blobbed binomial}
	\displaystyle C_{i,j}= \sum_{k=\frac{1}{2}(i-j)}^{\frac{1}{2}(i+j)}\binom{i}{k}.           
	\end{equation}
\end{teo}

\begin{demo}
We proceed by induction on $i$. If $i=0$ or $i=1$ then the  result is clear.  We suppose that  $i>1 $ and that (\ref{blobbed binomial}) holds for $i-1$. The proof splits naturally into two cases in accordance with the parity of $i$. We only treat the case $i$ even since the case $i$ odd is handled similarly. For the rest of the proof we assume $i$ is even. Let $j$ be an even integer with $0\leq j\leq i $. For $j=0$ our induction hypothesis yields
\begin{equation}
	\displaystyle C_{i,0}= C_{i-1,1}=\sum_{k=\frac{i}{2}-1}^{\frac{i}{2}} \binom{i-1}{k} = \binom{i-1}{\frac{i}{2}-1} + \binom{i-1}{\frac{i}{2}}= \binom{i}{\frac{i}{2}}, 
\end{equation}
as we wanted to show. For $j=i$ we have $C_{i,i}=2^i$ and (\ref{blobbed binomial}) holds in this case as well.\\We now assume $0<j<i$. By Definition \ref{defi Catalan triangle} we have 
$C_{i,j} = C_{i-1,j-1}+C_{i-1,j+1}$, then  our  induction hypothesis shows that
\begin{equation} \label{sums}
C_{i,j}=\displaystyle\sum_{k=\frac{i-j}{2}}^{\frac{i+j}{2}-1}\binom{i-1}{k}+\sum_{k=\frac{i-j}{2}-1}^{\frac{i+j}{2}}\binom{i-1}{k}.
 \end{equation}
% Recall the Pascal's theorem: $\binom{n}{j}=\binom{n-1}{j}+\binom{n-1}{j-1}$.
By collecting in consecutive pairs the terms appearing in the sums in (\ref{sums}) and by applying the well-known Pascal's identities we can rewrite these sums as

\begin{equation}\label{sumsA}
	\displaystyle\sum_{k=\frac{i-j}{2}}^{\frac{i+j}{2}-1}\binom{i-1}{k} = \sum_{k=1}^{\frac{j}{2}} \binom{i}{\frac{i-j}{2}+2k-1} \quad   \mbox{ and } \quad  \sum_{k=\frac{i-j}{2}-1}^{\frac{i+j}{2}}\binom{i-1}{k}=  \sum_{k=0}^{\frac{j}{2}} \binom{i}{\frac{i-j}{2}+2k}.
\end{equation}
Finally, by combining (\ref{sums}) and (\ref{sumsA})  we obtain
\begin{equation}
	C_{i,j} = \sum_{k=0}^j \binom{i}{\frac{i-j}{2}+k} = \sum_{k=\frac{1}{2}(i-j)}^{\frac{1}{2}(i+j)}\binom{i}{k}. 
\end{equation}
\end{demo}

\begin{coro} \label{lemma bin to catalan}
For all $i\geq 1$ we have $C_{2i,0} = \binom{2i}{i}$.
\end{coro}

\begin{demo}
	This is just the  case $j=0$ of Theorem \ref{teo closed formula for Cs}.
\end{demo}

We conclude this section by showing how the results obtained in this section allow to recover the main result in \cite{lee2016catalan}. Concretely, we prove that any binomial coefficient can be written as a $2$-power weighted sum of numbers occurring in the Catalan triangle. This result is not needed in the sequel. We begin by considering central binomial coefficients.

\begin{lem} \label{lem binom as catalan times two powers center case}
	For all $i\geq 1$ we have 
	\begin{equation}
	  \binom{2i}{i}= \sum_{k=1}^i 2^k c_{2i-k-1,k-1}.
	  \end{equation}
\end{lem}

\begin{demo}
We recall that $C_{2i,0}$ counts the number of paths  from the vertex $(0,0)$ to the vertex $(2i,0)$ in the blobbed Catalan triangle. These paths can be classified into $i$ disjoint sets according to the number of steps on the hypothenuse. For $1\leq k \leq i$, let $P_k$ be the set of paths with $k$ steps on the hypothenuse.   Let $p $ be a path in $P_k$. The steps on the hypotenuse can be performed in $2^k$ different ways. Once $p$ completed the $k$ steps on the hypothenuse, it is forced to make one step in southwest direction till it reaches the vertex $(k+1,k-1)$. Finally, the number of ways in which $p$ can be completed to reach the vertex $(2i,0)$ is given by
    $ c_{2i-k-1,k-1} $. To see this just reverse the arrows in the classical Catalan triangle and count the number of paths from $(2i,0)$ to  $(k+1,k-1)$. Summing up, we have
    \begin{equation}
    	C_{2i,0}= \sum_{k=1}^i 2^k c_{2i-k-1,k-1} .
    \end{equation}
    The result is now a consequence of Corollary \ref{lemma bin to catalan}. 
    \end{demo}

\begin{teo}
	Let $i$ and $j$ be integers such that $1\leq j \leq i $. Then, 
	\begin{equation}
		\displaystyle  \binom{2i-j}{i} = \sum_{k=1}^{i-j+1} 2^{k-1}c_{2i-k-j,j+k-2}.
	\end{equation}
	\end{teo}

\begin{demo}
	We proceed by induction on $j$. The cases $j=1$ and $j=2$ are left to the reader. So we assume that 
	 $j>2$ and that the result holds for integers smaller than $j$. We have

\begin{equation}
	\begin{array}{rl}
	\displaystyle	\binom{2i-j}{i} & = \displaystyle \binom{2i-(j-1)}{i} -\binom{2(i-1)-(j-2)}{i-1}   \\
	 &          \\
	 &\displaystyle  = \sum_{k=1}^{i-j+2}  2^{k-1}  c_{2i-k-j+1,j+k-3} - \sum_{k=1}^{i-j+2}  2^{k-1}  c_{2i-k-j,j+k-4} \\
	 &          \\
	 &\displaystyle  = \sum_{k=1}^{i-j+2}  2^{k-1} \left(  c_{2i-k-j+1,j+k-3} -  c_{2i-k-j,j+k-4}\right) \\
 &          \\
	 &\displaystyle  = \sum_{k=1}^{i-j+1}  2^{k-1}   c_{2i-k-j,j+k-2}.  \\

	 \end{array}
\end{equation}
\end{demo}

%	
%	If $j=1$ then Lemma \ref{lem binom as catalan times two powers center case} yields
%	\begin{equation} \label{eq lemma chi}
%		2 \binom{2i-1}{i} =  \binom{2i}{i} =\sum_{k=1}^i 2^k c_{2i-k-1,k-1}. 
%	\end{equation}
%	Then, by cancelling out by $2$ on both sides of (\ref{eq lemma chi}), we obtain the result for $j=1$. We now treat the case $j=2$.  In this case we have
%	\begin{equation}
%	\begin{array}{rl}
%	\displaystyle	\binom{2i-2}{i} & = \displaystyle \binom{2i-1}{i} -\binom{2(i-1)}{i-1}   \\
%	 &          \\
%	 &\displaystyle  = \sum_{k=1}^i 2^{k-1} c_{2i-k-1,k-1}  -  \sum_{k=1}^{i-1} 2^k c_{2i-k-3,k-1}    \\
%	 & \displaystyle = c_{2(i-1),0}+ \sum_{k=2}^i 2^{k-1} c_{2i-k-1,k-1}  -  \sum_{k=1}^{i-1} 2^k c_{2i-k-3,k-1}         \\
%	 &     \\
%	 & \displaystyle = c_{2(i-1),0}+ \sum_{k=1}^{i-1} 2^{k} c_{2i-k-2,k}  -  \sum_{k=1}^{i-1} 2^k c_{2i-k-3,k-1}         \\
%	 &     \\
%	 & \displaystyle = c_{2(i-1),0}+ \sum_{k=1}^{i-1}  2^{k} \left( c_{2i-k-2,k}  - c_{2i-k-3,k-1} \right)        \\
%	 & \\
%	 & \displaystyle = c_{2i-3,1}+ \sum_{k=1}^{i-2}  2^{k}  c_{2i-k-3,k+1}       \\
%	 & \\
%	 & \displaystyle = \sum_{k=0}^{i-2}  2^{k}  c_{2i-k-3,k+1} \\
%	 & \\
%	 & \displaystyle = \sum_{k=1}^{i-1}  2^{k-1}  c_{2i-k-2,k},  
%	\end{array}
%	\end{equation}
%which is what we wanted to show in this case. Here we use Lemma \ref{lem binom as catalan times two powers center case}, the theorem for $j=1$ and property 4 in Definition \ref{defin classical catalan triangle}. 

%%%%%%%%%%%%%%%%%%%%%%%%%%%%%%%%%%%%%%%%%%%%%%%%%%%%%%%%%%%%%%%%%%%%%%%%%%%%%%%%%%

\section{Enumeration of positive fully commutative elements.}
\label{section enumeration positive}
Let $n$ and $s$ be integers such that $n\geq 1$ and $s\geq 0$. We define $A_n^s =\{ w\in \pos \, |\, L(w)=s \}$ 
and set $a_{n}^s=|A_n^s|$.  The purpose of this section is to  determine the value of the numbers $\{a_{n}^s \}$. In other words, we enumerate the elements of $\pos $ according to their affine length. The main (and unique) result in this section is the following.

\begin{teo}  \label{teo enumeration normal words}
	Let $n$ and $s$ be integers such that $n\geq 1$ and $s\geq 0$. Then,  $a_{n}^s=C_{2n,2s}$.
\end{teo}

\begin{demo}
	We proceed by induction on $s$. We begin by treating the case $s=0$. We recall that $a_n^0$ denotes the cardinality of the set of positive fully commutative elements of affine length zero in $W^c(\atc)$. The above set coincides with the set of positive fully commutative elements in $W^c(B_n)$. Thus, the result in this case follows by  combining   Lemma \ref{lemma counting psotitive fully commu in type B} and Corollary \ref{lemma bin to catalan}. \\
	We now treat the case $s=1$. We notice that the set $A_{n+1}^0$ splits into two disjoint subsets. Namely, the subset formed by the elements that contain $\sigma_n$ and the subset formed by the elements  that do not contain $\sigma_n$. The set of elements in $A_{n+1}^0$ that contain $\sigma_n$ can be identified with $ A_n^1 $. On the other hand, the set of elements in $A_{n+1}^0$ that do not contain $\sigma_n$ can be identified with $A_{n}^0$. Then, we have $
		a_{n+1}^0 = a_{n}^1 +a_{n}^0$. By combining Lemma \ref{lemma properties of Cs} and the result for $s=0$ we conclude
	\begin{equation}
		a_n^1 = a_{n+1}^0 - a_{n}^0 = C_{2n+2,0}-C_{2n,0}= C_{2n+1,1}-C_{2n,0}= C_{2n,2},
	\end{equation}
which gives us the result for $s=1$.\\
We now assume that $s\geq 2$ and that the result holds for integers less than $s$. The set $ A_{n+1}^{s-1}$ decomposes into four disjoint subsets according to the behavior of the grids of its elements in the  $n$-th row. We illustrate such a decomposition in (\ref{equation four cases}) for $s=6$.
\begin{equation} \label{equation four cases}
	\scalebox{1.4}{\fourcases } 
\end{equation}
By disregarding the $(n+1)$-th row, we can identify Type I elements with elements of  $A_{n}^{s-2}$, Type II and Type III elements with elements of $A_{n}^{s-1}$ and Type IV elements with elements of $A_{n}^{s}$. Therefore, we have
 \begin{equation} \label{equation decomposing four cases}
 	a_{n+1}^{s-1}=a_{n}^{s-2} +a_n^{s-1}+a_n^{s-1}+a_n^{s}.   
 \end{equation}
 Then,  property (3) in Definition \ref{defi Catalan triangle} and  our induction hypothesis yield
  \begin{equation}
 \begin{aligned}
 	a_n^s & =  C_{2(n+1),2(s-1)} -  C_{2n,2(s-2)}- 2C_{2n,2(s-1)} \\
 	      & =  C_{2n+1,2s-3}+C_{2n+1,2s-1} -  C_{2n,2(s-2)}- 2C_{2n,2(s-1)} \\
 	      & =  C_{2n,2(s-2)}+C_{2n,2(s-1)} + C_{2n,2(s-1)}+C_{2n,2s}  -  C_{2n,2(s-2)}- 2C_{2n,2(s-1)} \\
 	      & =  C_{2n,2s}. 
 \end{aligned}
 \end{equation} 
\end{demo}

%%%%%%%%%%%%%%%%%%%%%%%%%%%%%%%%%%%%%%%%%%%%

\section{Enumeration of blobbed fully commutative elements.}
\label{section enumeration blobbed}
The goal of this section is to count the elements of $\blobbed$. We set $w_n^I=IJI$ and $w_n^J=JIJ$. We recall that the elements of $\blobbed$ are the positive fully commutative elements that avoid $w_n^I $ and $ w_n^J$. We notice that for $n\geq 2$ we have 
\begin{equation}\label{definition wI}
	w_n^I = \left\{   \begin{array}{ll}
	  \langle n-1,n \rangle\langle  n-3,n-1  \rangle\langle n-5,n-3 \rangle \cdots \langle 1,3 \rangle\langle 0,1\rangle, 	& \mbox{if } n \mbox{ is even;}  \\
	
		\langle n,n \rangle\langle n-2,n \rangle\langle n-4,n-2\rangle \cdots \langle 1,3 \rangle\langle 0,1\rangle , 	& \mbox{if } n \mbox{ is odd, }
	\end{array}  \right.
\end{equation}
and 
\begin{equation}\label{definition wJ}
	w_n^J = \left\{   \begin{array}{ll}
	  \langle n,n \rangle\langle  n-2,n \rangle\langle n-4,n-2\rangle \cdots \langle 0,2 \rangle\langle 0,0\rangle , 	& \mbox{if } n \mbox{ is even;}  \\
		\langle n-1,n  \rangle\langle n-3,n-1 \rangle\langle n-5,n-3\rangle \cdots \langle 0,2\rangle\langle 0,0\rangle , 	& \mbox{if } n \mbox{ is odd.}
	\end{array}  \right.
\end{equation} 
If  $n=1$ we have $w_1^I=\langle 1,1 \rangle\langle 0,1 \rangle$ and  $ w_1^J= \langle 0,1 \rangle\langle 0,0\rangle$. As in the previous section, we enumerate the elements of $\blobbed $ according with their affine length. Given integers $n\geq 1$ and  $s\geq 0 $ we define
\begin{equation}
	B_n^s =\{ w\in \blobbed \, |\, L(w) =s  \} 
\end{equation}
and set $b_n^s = |B_n^s |$. We also define $D_n^s= A_n^s - B_n^s$
and set $d_{n}^s =|D_n^s|= a_n^s -b_n^s$. We compute the numbers $\{d_n^s \}$.  Since the numbers $\{a_n^s \}$ are known by Theorem \ref{teo enumeration normal words}, the knowledge of  $ d_n^s$ is equivalent to the knowledge of $b_n^s$. We stress that the elements of $D_n^s$ are  the positive FC elements in $W^c(\atc)$ of affine length $s$ that contain $w_n^I$ or $w_n^J$. By Proposition \ref{lemma blobbed if part} we can see that an element $w\in D_n^s$ if and only if $G(w)$ has $s$ black dots in the $n$-th row and  contains $G(w_n^I)$ or $G(w_n^J)$. Sometimes we abuse notation and think of $D_n^s$ as the set formed by the grids of its elements. We begin our counting by treating the cases  $s=0$ and $s>n$. 

\begin{teo} \label{lemma extremes cases blobbed fully commutative elements}
	Let $n$ be a positive integer. Then,  $d_{n}^0= 0$ and $ d_n^s = a_n^s $, if $s>n$. Consequently, $b_n^0=a_n^0$ and $b_n^s=0$ if $ s>n$.
\end{teo}

\begin{demo}
	Let $w\in A_n^s$. If $s=0$ then $w$ does not contain the generator $\sigma_n$. Thus $w$ cannot contain $w_n^I$ or $w_n^J$ since the generator $\sigma_n$ occurs in both $w_n^I$ and $w_n^J$. We conclude that  $d_{n}^0= 0$. 
	We now assume that $s>n$. In this case $G(\langle n,n  \rangle \langle  n-1,n  \rangle \ldots  \langle  0,n \rangle )$ is contained in $G(w)$ and this forces to $G(w)$ to contain $G(w_n^I)$, as illustrated in (\ref{s biger than n}) for $(n,s)=(6,7)$ and $(n,s)=(7,8)$, respectively. Therefore, $w\in D_n^s$ and  $ d_n^s = a_n^s $. 
\end{demo}

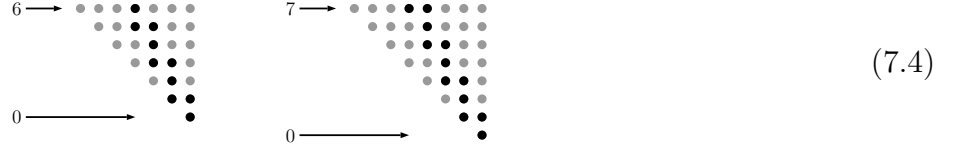
\begin{equation} \label{s biger than n}
		\scalebox{.6}{\sbigerthann }
	\end{equation}
Theorem  \ref{lemma extremes cases blobbed fully commutative elements} confirms the fact that $W^c_b(\atc)$ is finite. We continue our counting with the case $s=1$. We first introduce some suitable notations. 
 
\begin{defi}\rm
	We define ${I_{n,r}^\leftarrow }	$ to be  the subset of $D_n^1$ formed by all  the grids that are obtained from $G(w_n^I)$ by adding black dots exclusively on its left  and with $r$ dots added in its leftmost non-empty column. Similarly, we define ${I_{n,r}^\rightarrow }$ to be the subset of $D_n^1$ formed by all the grids that are obtained from $G(w_n^I)$  by adding black dots exclusively on its right  and with $r$ dots added in its highest non-empty diagonal\footnote{Here ``diagonal'' means a line with slope $-1$.}. We also define $J_{n,r}^{\leftarrow}$ and $J_{n,r}^{\rightarrow }$ by considering $w_n^J$ rather than $w_n^I$. Furthermore, we set $i_{n,r}^\leftarrow =|I_{n,r}^\leftarrow|, i_{n,r}^\rightarrow=|I_{n,r}^\rightarrow|, j_{n,r}^\rightarrow=|J_{n,r}^\rightarrow|$ and $j_{n,r}^\leftarrow=|J_{n,r}^\leftarrow|$. Finally, we define
	\begin{align}
	i_n^\leftarrow:=\sum_{r=0}^{n-1} i_{n,r}^\leftarrow \qquad \mbox{ and } \qquad i_n^\rightarrow:=\sum_{r=0}^{n-1}	 i_{n,r}^\rightarrow \label{definitionI} \\
	j_n^\leftarrow:=\sum_{r=0}^{n-1} j_{n,r}^\leftarrow \qquad \mbox{ and } \qquad j_n^\rightarrow:=\sum_{r=0}^{n-1}  j_{n,r}^\rightarrow.\label{definitionJ}
	\end{align}
\end{defi}

\begin{lem}\label{BiyectionI}
	Let $n$ and $r$ be integers such that $0\leq r <n$. We have $i_{n,r}^\leftarrow=i_{n,r}^\rightarrow$ and $j_{n,r}^\leftarrow=j_{n,r}^\rightarrow$.
\end{lem}

\begin{demo}
	We just prove $i_{n,r}^\leftarrow=i_{n,r}^{\rightarrow}$, the other equality is treated similarly. It is enough to exhibit a bijection between  $I_{n,r}^{\rightarrow} $ and $I_{n,r}^\leftarrow$. Such a bijection is given  by a rotation of the  region on the right of  $G(w_n^I)$ in $45^\circ$ clockwise followed by a reflection through a vertical edge. 	We depicted such transformations in (\ref{equation rotation of the diagonals}) for $n=8$.  We notice that under this bijection the points located on the highest diagonal are mapped to points in the leftmost column. 
\end{demo}

\begin{equation}\label{equation rotation of the diagonals}
	\scalebox{.5}{\rotatingthediagonals }
\end{equation}	
Lemma \ref{BiyectionI} allows us to define
\begin{equation}
	i_{n,r}:=i_{n,r}^\leftarrow=i_{n,r}^\rightarrow \qquad \mbox{ and } \qquad j_{n,r}:=j_{n,r}^\leftarrow=j_{n,r}^\rightarrow.
\end{equation}
Similarly, we define
\begin{equation}
	i_n:=i_n^\leftarrow=i_n^\rightarrow \qquad \mbox{ and } \qquad  j_n:=j_{n}^\leftarrow=j_n^\rightarrow.
\end{equation}

\begin{lem}   \label{theorem elements eliminated in affine length one}
	Let $n$ be a positive integer. We have
	\begin{equation*}
		d_{n}^1=\left\{ \begin{array}{l l}
		(i_n)^2, & \mbox{if n is even;}\\  (j_n)^2,  &	\mbox{if $n$ is odd.}
 \end{array}
\right.
	\end{equation*}
\end{lem}

\begin{demo}
	We assume that $n$ is even. We need to count the number of elements in $D_n^1$, that is, the number of positive fully commutative elements of affine length one that contain $w_n^I $ or $w_n^J$. Since $\sigma_n$ occurs twice in $w_n^J$ when $n$ is even, we conclude that $w_n^J$  cannot be contained in an element of affine length one. For this reason we only care about $w_n^I$. In (\ref{equation many heaps}) we have depicted  $G(w_n^I)$, for $n=2$, $n=4$, $n=6$ and $n=8$, where we have added white dots to indicate the positions where we can add black dots keeping the affine length one. Since any addition of black dots in the left  is independent of any addition of black dots in the right, and vice-versa,  we conclude that
\begin{equation}
		d_n^1= i_n^\leftarrow \cdot  i_n^\rightarrow=i_n^2.
	\end{equation}	
The result for $n$ odd is treated similarly. 
\end{demo}
	\begin{equation} \label{equation many heaps}
		\scalebox{.5}{\manyheaps }
	\end{equation}

%	{\red Now we want obtain a specifically value for $i_n$ using the blob Catalan numbers defined in the previous section.} For this we first find a relation between $i_{n,r}^\rightarrow$ and the blob Catalan numbers.
	
	\begin{lem} \label{lemma d appears in the trinagle}
		Let $n$ and $r$ be integers such that $0\leq r <n$. Then, we have  
		\begin{equation}\label{down region in the blobbed catalan triangle }
		i_{n,r}=\left\{ \begin{array}{l l}
		C_{n-2-r,r}, & \mbox{if n is even; }\\  0,  &	\mbox{if $n$ is odd;}
 \end{array}
\right.\qquad \mbox{ and } \qquad 
j_{n,r}=\left\{ \begin{array}{l l} 
 	0, & \mbox{ if n is even; }\\
 	\frac{1}{2}C_{n-1-r,r} &\mbox{ if n is odd.}
\end{array}\right.
	\end{equation}
Consequently, 
\begin{equation}\label{equation in   }
		i_{n}=\left\{ \begin{array}{l l}
		C_{n,0}, & \mbox{if n is even; }\\  0,  &	\mbox{if $n$ is odd;}
 \end{array}
\right.\qquad \mbox{ and } \qquad 
j_{n}=\left\{ \begin{array}{l l} 
 	0, & \mbox{ if n is even; }\\
 	\frac{1}{2}C_{n,1} &\mbox{ if n is odd.}
\end{array}\right.
	\end{equation}
\end{lem}	
	
	\begin{demo}
	We only prove the formulas for $i_{n,r}$ and $i_n$. We recall that $i_{n,r}=i_{n,r}^\leftarrow=i_{n,r}^\rightarrow$. Thus it is enough to show that $i_{n,r}^\leftarrow$ matches with the value given in \eqref{down region in the blobbed catalan triangle }. If $n$ is odd then $w_n^I $ has already two occurrences of the generator $t_n$. Therefore, $I_{n,r}^\leftarrow=\emptyset$ since  $I_{n,r}^\leftarrow$ is by definition a subset of $D_n^1$. We conclude that $i_{n,r}=i_{n,r}^\leftarrow =0 $. We now assume that $n$ is even and proceed by induction. It can be easily checked that
	\begin{equation}
		i_{2,0}^\leftarrow =1 = C_{0,0} \qquad \mbox{ and } \qquad i_{2,1}^{\leftarrow} =1=C_{-1,1}, 
 	\end{equation}
	which provides the base of our induction for $n=2$. We now assume that (\ref{down region in the blobbed catalan triangle }) holds for $n$ and we will prove it for $n+2$. The key point here is that the $G(w_{n+2}^I)$ and its left region ``contains'' $G(w_n^I)$ and its left region, as illustrated in  (\ref{equation heap in other heap}) for $n=6$. Therefore, any element of 
	$I_{n+2}^\leftarrow$ is given by a number of black dots added in the leftmost non-empty column of $G(w_{n+2}^I)$ and some element of $I_{n}^\leftarrow $. 
	\begin{equation} \label{equation heap in other heap}
		\scalebox{.5}{\heapinotherheap}
	\end{equation}
If we add $0$ or $1$ black dots in  the leftmost non-empty column of $G(w_{n+2}^I)$ then we can consider any element of $I_{n}^\leftarrow $ in order to construct an element of $I_{n+2}^\leftarrow $.  Thus, our induction hypothesis and Lemma \ref{lemma properties of Cs} imply 
 \begin{equation}
 i_{n+2,0}^{\leftarrow} =  i_{n+2,1}^{\leftarrow} =\sum_{k=0}^{n-1} i_{n,k}^{\leftarrow} = \sum_{k=0}^{n-1} C_{n-2-k,k} = C_{n-1,1} =C_{n,0} .
\end{equation}
We now assume that $ 2\leq r \leq n $. If we add $r $ black dots in  the leftmost non-empty column of $G(w_{n+2}^I)$ then we are forced to add at least $r-1$ black dots in its second column, otherwise the configuration obtained would not be a grid of a positive fully commutative element. Therefore, 
\begin{equation}
i_{n+2,r}^{\leftarrow} =  \sum_{k=r-1}^{n-1} i_{n,k}^{\leftarrow} = \sum_{k=r-1}^{n-1} C_{n-2-k,k} =\sum_{k=0}^{n-r} C_{ (n-r)-1-k,r-1+k} =  C_{(n+2)-2-r,r}.
\end{equation}
Furthermore, we note that $i_{n+2,n+1}^{\leftarrow} =1 = C_{-1,n+1}$. This complete the proof of \eqref{down region in the blobbed catalan triangle }. Finally, \eqref{equation in   } is now a consequence of Lemma \ref{lemma properties of Cs} and \eqref{down region in the blobbed catalan triangle }. 
	\end{demo}

\begin{teo}
Let $n$	be a positive integer. We have 
\begin{equation*}
		d_{n}^1=\left\{ \begin{array}{l l}
		(C_{n,0})^2 & \mbox{if n is even }\\  (\frac{1}{2}C_{n,1})^2  &	\mbox{if $n$ is odd}
 \end{array}
\right.
	\end{equation*}
\end{teo}

\begin{demo}
The result follows by a direct application of Lemma \ref{theorem elements eliminated in affine length one} and 	Lemma \ref{lemma d appears in the trinagle}.
\end{demo}

\medskip
\noindent
We recall that our goal in this section is to compute the numbers $\{ d_n^s  \}$. So far we have already computed these numbers for $s=0$, $s=1$ and $s>n$. To deal with the case $2\leq s \leq n$ we need a bit more of notation.

\begin{defi} \rm 
Let $n$ be a positive integer and $t$ be a non-negative integer. We define the set $\mathcal{I}_{n}^{\leftarrow , t }$ to be the set of grids obtained from  $G(w_n^I)$ by adding black dots exclusively on its left and exactly $t$ black dots on the $n$-th row. We define $\mathcal{I}_{n}^{\rightarrow , t }$ in a similar way but this time we consider the region located on the right of $G(w_n^I)$. We also define $\mathcal{J}_{n}^{\leftarrow , t }$ and $\mathcal{J}_{n}^{\rightarrow , t }$ by considering $w_n^J$ rather than $w_n^I$. Finally, we define $i_{n}^{\leftarrow, t} $, $i_{n}^{\rightarrow,t}$, $j_{n}^{\leftarrow , t} $ and $j_{n}^{\rightarrow , t}$ to be the cardinalities of the sets $\mathcal{I}_{n}^{\leftarrow , t }$,  $\mathcal{I}_{n}^{\rightarrow , t }$, $\mathcal{J}_{n}^{\leftarrow , t }$ and $\mathcal{J}_{n}^{\rightarrow , t }$, respectively.
\end{defi}

\begin{lem}  \label{lemma number of words i and j with a affine length given}
	Let $n$ be a positive integer and $t$ be a non-negative integer. We have
	\begin{equation} \label{equation ies increasing the affine lenght}
		i_{n}^{\leftarrow ,t} = i_n^{\rightarrow , t}  \qquad \mbox{ and } \qquad j_{n}^{\leftarrow ,t} = j_n^{\rightarrow , t}.
	\end{equation} 
\end{lem}

\begin{demo}
	We only prove  $i_{n}^{\leftarrow ,t} = i_n^{\rightarrow , t}$. The other equality is treated similarly.  In order to see that $ i_{n}^{\leftarrow ,t} = i_n^{\rightarrow , t}$ it is enough to exhibit a bijection between $ \mathcal{I}_{n}^{\rightarrow ,t} $ and  $ \mathcal{I}_n^{\leftarrow , t} $. Such a bijection is described graphically as follows. Let $G \in \mathcal{I}_{n}^{\rightarrow ,t}$. We draw lines starting from the $0$-th row with slope $-1$ connecting the points in $G$ that are not involved in $G(w_n^I)$. Then, we rotate these lines around the point located in the $0$-th row in $45^\circ$ clockwise and locate the rotated points at the same level as they were before the rotation. After doing this we apply a reflection through a vertical line. Finally, we locate $G(w_n^I)$ on the right of the resultant configuration.  We illustrate in (\ref{exa bijection})  the bijection for $n=10$ and $t=3$. Some remarks are in order. Blue dots have no special meaning. They must be thought simply as black dots. The difference between black and blue dots is that the black ones are forced to appear in any element of $ \mathcal{I}_{10}^{\leftarrow ,3} $ and  $ \mathcal{I}_{10}^{\rightarrow , 3} $, whereas the blue ones are just a choice we did to illustrate the bijection. Finally, we notice that under this bijection the points  on the  $n$-th row are mapped to points on the $n$-th row.
\end{demo}
\begin{equation} \label{exa bijection}
		\exbijection
	\end{equation}

\begin{rem}\rm \label{remark arrows do not matter}
A practical consequence of Lemma 	\ref{lemma number of words i and j with a affine length given} is that the direction of the arrows  in the symbols $i_{n}^{\leftarrow, t} $, $i_{n}^{\rightarrow,t}$, $j_{n}^{\leftarrow , t} $ and $j_{n}^{\rightarrow , t}$ is irrelevant. For this reason we can relax the notation by dropping the arrows. Concretely, we define
\begin{equation}
	i_n^t:= i_n^{\leftarrow,t}= i_n^{\rightarrow , t} \quad \mbox{ and } \quad j_n^t:= j_n^{\leftarrow,t}= j_n^{\rightarrow , t}.
\end{equation}
\end{rem}

\begin{lem} \label{teo the last one}
	Let $n$ and $s$ be integers with $2\leq s \leq n$. Then, 
	
	\begin{equation} \label{equation dns}
		d_n^s = \left\{   \begin{array}{cc}
	\displaystyle	 \sum_{k=0}^{s-1} i_n^{ k} i_n^{ s-1-k} + \sum_{k=0}^{s-2} j_n^{ k}j_n^{ s-2-k} -2 \sum_{k=0}^{s-2}i_n^k j_n^{ s-2-k}   , & \mbox{if }  $n$  \mbox{ is even};  \\
			&   \\
 \displaystyle 	 \sum_{k=0}^{s-1} j_n^{ k} j_n^{ s-1-k} + \sum_{k=0}^{s-2} i_n^{ k} i_n^{ s-2-k}  -2 \sum_{k=0}^{s-2} j_n^k i_n^{ s-2-k}  , & \mbox{if }  $n$  \mbox{ is odd}.
		\end{array}  \right.
	\end{equation}
\end{lem}

\begin{demo}
For the sake of brevity we only prove the result for $n$ even. The case $n$ odd is handled similarly and is left to the reader.  We recall that $d_n^s$ is the cardinality of $D_n^s$ and that $D_n^s$ is the set of positive fully commutative elements of affine length $s$ that contain $G(w_n^I)$ or $G(w_n^J) $.  The main obstacle to carry out the counting of such elements is that there are elements in $D_n^s$ that contain $G(w_n^I)$ or $G(w_n^J) $ more than once. In order to overcome this drawback we need a couple of definitions:
\begin{enumerate}
	\item First, we label the $s$ black dots appearing in the $n$-th row of any element of $D_n^s$ from left to right as $P_1,P_2,\ldots P_s$.
	\item Given an element $x\in D_n^s $ and an integer $1\leq u \leq s$  we say that $w_n^I$ \emph{appears in} $x$ \emph{at position }$u$ if there is an occurrence of $G(w_n^I)$ in $x$ in which $P_u$ is involved. 
	\item  Given an element $x\in D_n^s $ and an integer $1\leq u < s$  we say that $G(w_n^J)$ \emph{appears in }$x$ \emph{at position }$u$ if there is an occurrence of $G(w_n^J)$ in $x$ in which $P_u$ and $P_{u+1}$ are involved. 
	\item Finally, we define the sets
	 \begin{equation}
	\displaystyle 	D_n^s(I,u)= \{ x\in D_n^s \, |\, G(w_n^I) \mbox{ appears in } x \mbox{ at position } u \}
	 \end{equation}
	 \begin{equation}
	 \displaystyle	D_n^s(J,u)= \{ x\in D_n^s \, |\, G(w_n^J) \mbox{ appears in } x \mbox{ at position } u \}
	 \end{equation}	 
\end{enumerate}

We split the proof into two cases according the parity of $s$. \\

\noindent
	{\bf{Case A.}} We suppose that $s$ is odd and set $\s:= (s+1)/2$. We can decompose the set $D_n^s$ as a disjoint union of the sets
	
	\begin{equation} \label{super decomposition}
		\begin{array}{ccc}
			& \De{I}{\s} &  \\
	\De{J}{\s -1} \backslash \De{I}{\s}  	 &  & \De{J}{\s} \backslash \De{I}{\s}  \\
	 \De{I}{\s -1} \backslash \De{J}{\s -1}     &   &  \De{I}{\s +1} \backslash \De{J}{\s }  \\
	 \De{J}{\s -2} \backslash \De{I}{\s-1}  	 &  & \De{J}{\s+1} \backslash \De{I}{\s+1}  \\
	 \De{I}{\s -2} \backslash \De{J}{\s -2}     &   &  \De{I}{\s +2} \backslash \De{J}{\s +1 }  \\
	  & & \\
	  \vdots &  & \vdots \\ 
	  & & \\ 
	\De{J}{1} \backslash \De{I}{2}  	 &  & \De{J}{s-1} \backslash \De{I}{s-1}  \\
	 \De{I}{1} \backslash \De{J}{1}     &   &  \De{I}{s} \backslash \De{J}{s-1 } . \\  
		\end{array}
	\end{equation}
With this decomposition at hand, we have reduced the proof of the theorem to compute the cardinality of each one of the sets occurring in (\ref{super decomposition}).\newline
We begin by determining the cardinality of $ \De{I}{\s}$. By definition, each element of $ \De{I}{\s}$ has a $G(w_n^I)$ at position $\s$. Furthermore,  to the left and to the right of this occurrence of $G(w_{n}^I)$ there are  $\mathfrak{s}-1$ black dots in the $n$-th row,  as illustrated for $s=7$ in (\ref{eq real counting}).
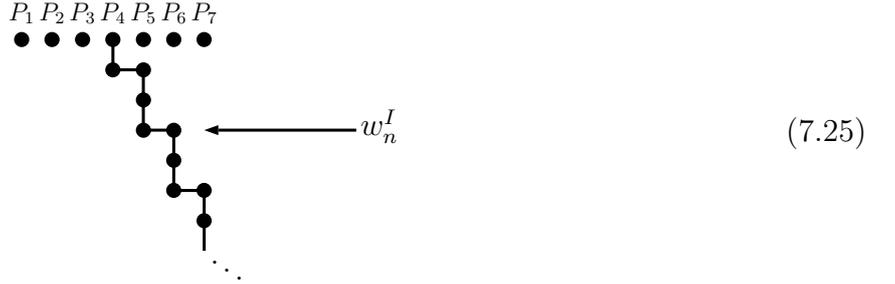
\begin{equation} \label{eq real counting}
\scalebox{1}{\realcounting }
\end{equation}
By using  the fact that any addition of black dots on the left is independent of any addition of black dots on the right and vice-versa, we conclude that
\begin{equation}\label{eq real counting A}
\displaystyle	|\De{I}{\s}| = i_n^{\mathfrak{s} -1} \cdot i_n^{\mathfrak{s} -1}.
\end{equation}	
We now compute $| \De{J}{\s -u} \backslash \De{I}{\s-u+1}|$, for $1\leq u <\s $. 	By definition, each element of
$\De{J}{\s -u}$ has a $G(w_n^J)$ at position $\s-u$. Furthermore,  to the left of this occurrence of $G(w_{n}^J)$ there are $\s-u-1$  black dots located in the $n$-th row. On the other hand, to the right of the aforementioned occurrence of $G(w_{n}^J)$ there are $\s+u-2$  black dots located in the $n$-th row. Therefore,  we have
\begin{equation}  \label{first counting}
	| \De{J}{\s -u} | =  j_n^{\s -u-1} \cdot j_n^{\s +u -2}. 
\end{equation}
To determine the cardinality of $\De{J}{\s -u} \backslash \De{I}{\s-u+1}$, it only remains to know the cardinality of
$\De{J}{\s -u} \cap \De{I}{\s-u+1}$. By definition, an element of $\De{J}{\s -u} \cap \De{I}{\s-u+1}$ has a $G(w_n^J)$ at position $\s-u$ and a $G(w_n^I)$ at position $\s-u+1$. These elements have $\s-u-1$  black dots  in the $n$-th row to the left of the aforementioned occurrence of $G(w_n^J)$ and $\s+u-2$ black dots  in the $n$-th row to the right of the aforementioned occurrence of $G(w_n^I)$. The above is illustrated for $s=13$ (and therefore $\s=7$) and $u=4$ in 
(\ref{eliminating the bad elements}), where we have depicted the relevant occurrences of $G(w_n^J)$ and $G(w_n^I)$ in black and red, respectively.
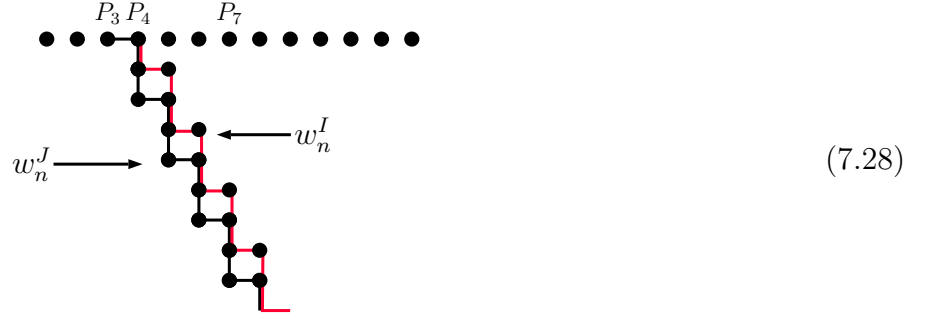
\begin{equation}  \label{eliminating the bad elements}
	\realcountingB
\end{equation}
We obtain
\begin{equation} \label{second counting}
	|\De{J}{\s -u} \cap \De{I}{\s-u+1}| = j_n^{\s -u-1} i_n^{\s +u-2}.
\end{equation}
Finally, a combination of (\ref{first counting}) and (\ref{second counting}) yields
\begin{equation}\label{eq real counting B}
	| \De{J}{\s -u} \backslash \De{I}{\s-u+1}| = j_n^{\s -u-1} ( j_n^{\s +u -2} -i_n^{\s +u-2}).
\end{equation}
Similarly, we obtain 
\begin{equation} \label{eq real counting C}
\begin{array}{r}
	|\De{I}{\s -u} \backslash \De{J}{\s -u}| =  i_n^{\s -u-1} ( i_n^{\s +u -1} -j_n^{\s +u-2});  \\
	| \De{J}{\s +u} \backslash \De{I}{\s+u}|  = j_n^{\s -u-2} ( j_n^{\s +u -1} -i_n^{\s +u-1}); \\
	|\De{I}{\s +u} \backslash \De{J}{\s +u-1}| =  i_n^{\s -u-1} ( i_n^{\s +u -1} -j_n^{\s +u-2});
\end{array}
\end{equation}		
for $ 1\leq u <\s$, $0\leq u <\s -1 $ and $1\leq u <\s $, respectively. We now combine (\ref{super decomposition}), (\ref{eq real counting A}), (\ref{eq real counting B}) and (\ref{eq real counting C}) to obtain
\begin{equation}\label{eq real counting D}
	\displaystyle d_n^s = \begin{array}{l}
	i_n^{\mathfrak{s} -1} \cdot i_n^{\mathfrak{s} -1} +\displaystyle	\sum_{u=1}^{\s-1}  j_n^{\s -u-1} ( j_n^{\s +u -2} -i_n^{\s +u-2}) \, +\\
	\displaystyle\sum_{u=1}^{\s-1}  i_n^{\s -u-1} ( i_n^{\s +u -1} -j_n^{\s +u-2}) + \displaystyle\sum_{u=1}^{\s-1}  i_n^{\s -u-1} ( i_n^{\s +u -1} -j_n^{\s +u-2})   \, + \\
	\displaystyle\sum_{u=0}^{\s-2}  j_n^{\s -u-2} ( j_n^{\s +u -1} -i_n^{\s +u-1}).
	\end{array} 
\end{equation}  	
In order to match the formula in (\ref{eq real counting D}) with the one in (\ref{equation dns}) we collect the terms formed by products of $j$'s, products of $i$'s and the products of $i$'s and $j$'s.  For instance, the products of $j$'s in (\ref{eq real counting D}) are
%\begin{equation}\label{eq real counting E}
%	\displaystyle	\sum_{u=1}^{\s-1}  j_n^{\s -u-1}   j_n^{\s +u -2} + \sum_{u=0}^{\s-2}  j_n^{\s -u-2} j_n^{\s +u -1}. 
%\end{equation} 
% By using the change of variables  $k=\s -u-1$ and $k=\s +u -1$ in the sums on the left-hand side and on the right-hand side of (\ref{eq real counting E}), respectively, equation (\ref{eq real counting E}) becomes
  \begin{equation}\label{eq real counting F}
	\displaystyle	\sum_{k=0}^{\s-2} j_n^kj_n^{2\s-3-k} + \sum_{k=\s-1}^{2\s-3} j_n^kj_n^{2\s-3-k} = \sum_{k=0}^{2\s-3} j_n^kj_n^{2\s-3-k}. 
	\end{equation}
	Finally, going back to the normal $s$, we conclude that the products of $j$'s in (\ref{eq real counting D}) is given by 
	\begin{equation}
		\sum_{k=0}^{s-2} j_n^kj_n^{s-2-k},  
	\end{equation}
	which is the same that appears in (\ref{equation dns}). The other terms are treated similarly. This finishes the proof of the theorem in this case.\\
	
	\noindent
	{\bf{Case B.}} We suppose that $s$ is even  and set $\s:= s/2$. In this case we can decompose $D_n^s$ as a disjoint union of the sets
		\begin{equation} \label{super decomposition s even}
	\begin{array}{ccc}
			& \De{J}{\s} &  \\
	\De{I}{\s } \backslash  \De{J}{\s}  	 &  &           \De{I}{\s+1} \backslash \De{J}{\s}  \\
	 \De{J}{\s -1} \backslash \De{I}{\s }     &   &    \De{J}{\s +1} \backslash \De{I}{\s +1 }  \\
	 \De{I}{\s -1} \backslash \De{J}{\s-1}  	 &  &         \De{I}{\s+2} \backslash \De{J}{\s+1}  \\
	 \De{J}{\s -2} \backslash \De{I}{\s -1}     &   &    \De{J}{\s +2} \backslash \De{I}{\s +2 }  \\
	  & & \\
	  \vdots &  & \vdots \\ 
	  & & \\ 
	\De{I}{2} \backslash \De{J}{2}   &  & \De{I}{s-1} \backslash \De{J}{s-2 } \\
	\De{J}{1} \backslash \De{I}{2}  	 &  & \De{J}{s-1} \backslash \De{I}{s-1}  \\
	 \De{I}{1} \backslash \De{J}{1}     &   &  \De{I}{s} \backslash \De{J}{s-1 } . \\  
		\end{array}
	\end{equation}
The rest of the argument carries over for this case. 
\end{demo}

\begin{lem}  \label{lemma number of words i and j with a affine length given A}
	Let $n$ be a positive integer and $t$ be a non-negative integer. We have
	
	\begin{equation} \label{equation ies increasing the affine lenght A}
		i_{n}^{t} =\left\{  \begin{array}{ll}
			C_{n,2t}, & \mbox{ if } n \mbox{ is even;}\\
			C_{n,2t+1},& \mbox{ if } n \mbox{ is odd.}  \quad \quad \quad
		\end{array} \right.
	\end{equation} 
	\begin{equation}\label{equation jes increasing the affine lenght A}
		j_{n}^{t}  =\left\{  \begin{array}{ll}
		(1/2)	C_{n+1,2t+1}, & \mbox{ if } n \mbox{ is even;}\\
		(1/2)	C_{n+1,2t},& \mbox{ if } n \mbox{ is odd.} 
		\end{array} \right.
	\end{equation}
	
\end{lem}

\begin{demo}
	We only prove (\ref{equation ies increasing the affine lenght A}). Equation (\ref{equation jes increasing the affine lenght A}) is treated similarly. We split the proof into two cases in accordance with the parity of $n$. We recall from Remark \ref{remark arrows do not matter}  that $i_n^t=i_n^{\leftarrow,t}=i_{n}^{\rightarrow,t}$. Thus it is enough to show that   $i_n^{\leftarrow,t}$ matches with the value given in \eqref{equation ies increasing the affine lenght A}.
	
\medskip
\noindent
{\bf{Case A:}} $n $ is even. If $t=0$ we have $\mathcal{I}_{n}^{\leftarrow ,0}= I_{n}^\leftarrow$, then Lemma \ref{lemma d appears in the trinagle} implies 
\begin{equation}
	i_{n}^{\leftarrow ,0} = i_n^{\leftarrow} =i_n= C_{n,0},
\end{equation}
	as we wanted to show. We now assume that $t>0$. The key point here is that any element of  $ \mathcal{I}_{n}^{\leftarrow ,t}$ can be seen as an element of $I_{n+2t}^{\leftarrow}$ with at least $2t-1$ black dots added in the leftmost column.  The above is achieved by adding black dots to $G(w_n^I)$ in order to obtain $G(w_{n+2t}^I)$, as illustrated in  (\ref{equation extending}) for $n=10$ and $t=3$, where we have drawn in blue the points needed to pass from $G(w_{10}^I)$ to $G(w_{16}^I)$. We have also depicted the black dots that are forced to appear. In particular, there are $5=2t-1$ black dots forced to appear in the leftmost non-empty column.  
	\begin{equation}\label{equation extending}
		\scalebox{.5}{\extending }
	\end{equation}
	 By combining Lemma \ref{lemma properties of Cs} and 	Lemma \ref{lemma d appears in the trinagle} we obtain
%	\begin{equation}
%		 \begin{array}{rl}
%		 i_{n}^{\leftarrow , t}	& = \displaystyle \sum_{k=2t-1}^{n+2t-1} i_{n+2t,k}^{\leftarrow}   \\
%		  & =  \displaystyle \sum_{k=0}^n  i_{n+2t,k+2t-1}^{\leftarrow }     \\
%		  & =  \displaystyle \sum_{k=0}^n C_{n+2t-2-(k+2t-1),k+2t-1} \\
%		  & =  \displaystyle \sum_{k=0}^n C_{n-1-k,2t-1+k} \\
%		  & = C_{n,2t}. 
%		 \end{array}
%	\end{equation}
		\begin{equation*}
i_{n}^{\leftarrow , t}	= \displaystyle \sum_{k=2t-1}^{n+2t-1} i_{n+2t,k}^{\leftarrow}   =  \displaystyle \sum_{k=0}^n  i_{n+2t,k+2t-1}^{\leftarrow }      =  \displaystyle \sum_{k=0}^n C_{n+2t-2-(k+2t-1),k+2t-1}  =  \displaystyle \sum_{k=0}^n C_{n-1-k,2t-1+k} = C_{n,2t}. 
		 \end{equation*}
		
\medskip		 
\noindent		 
{\bf{Case B:}} $n $ is odd. 	We claim that 
\begin{equation}\label{claim to prove odd part}
	i_n^{\leftarrow , t} =i_{n-1}^{\leftarrow , t} + i_{n-1}^{\leftarrow , t+1}.
\end{equation}
We notice that if (\ref{claim to prove odd part})  holds, then property (3) in Definition \ref{defi Catalan triangle} and an application of the theorem for the even case yield

\begin{equation}
	i_n^{\leftarrow , t} =i_{n-1}^{\leftarrow , t} + i_{n-1}^{\leftarrow , t+1}= C_{n-1,2t}+C_{n-1,2t+2}=C_{n,2t+1},
\end{equation}
which is what we want to show. So that to finish the proof of the theorem we only need to check (\ref{claim to prove odd part}). To see why the above formula is correct, we notice that there is a bijection between $\mathcal{I}_n^{\leftarrow , t} $ and the disjoint union of  $\mathcal{I}_{n-1}^{\leftarrow , t} $ and $ \mathcal{I}_{n-1}^{\leftarrow , t+1}$. Such a bijection is given  as follows. First, we draw $G(w_n^I)$. Then, we draw $t$ black dots in the $n$-th row to the left of $G(w_n^I)$. The occurrence of these $t$ dots forces the occurrence of $t$ black dots in the $(n-1)$-th row. Beside these black dots, we can still add another black dot in the intersection of the  $(n-1)$-th row with the leftmost non-empty column. We refer to  such a dot as the \emph{special} dot. The special dot is depicted in blue  in (\ref{picture odd case odd}) for  $n=11$ and $t=3$. We conclude that the elements of $\mathcal{I}_n^{\leftarrow , t} $ split into classes according whether we use the special dot or not. Let $G$ be an element of $\mathcal{I}_n^{\leftarrow , t} $.  If the special dot appears in $G$ then  by erasing the whole $n$-th row we obtain an element of $ \mathcal{I}_{n-1}^{\leftarrow , t+1}$. If the special dot does not  appear in $G$ then by erasing the whole $n$-th row we obtain an element of $ \mathcal{I}_{n-1}^{\leftarrow , t}$. This gives the promised bijection and completes the proof of (\ref{claim to prove odd part}). 
\end{demo}
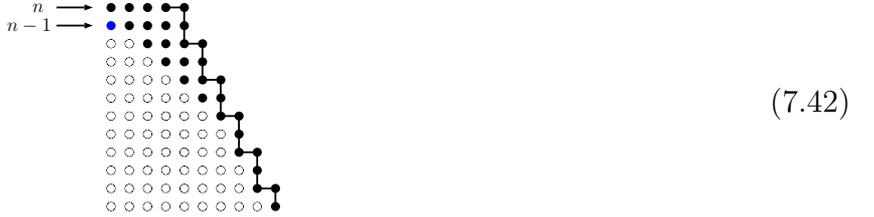
\begin{equation}\label{picture odd case odd}
	\scalebox{.6}{ \nodd }
\end{equation}

\begin{teo}
Let $n$ and $s$ be integers with $2\leq s \leq n$. Then,
\begin{equation*} \label{equation dns-Catalan}
			d_n^s = \left\{   \begin{array}{l}
	\displaystyle	\sum_{k=0}^{s-2}  \left[C_{n,2k}( C_{n, 2(s-1-k)}-C_{n+1,2(s-k)-3})
		 + \frac{1}{4} C_{n+1,2k+1} C_{n+1, 2(s-k)-3}\right]    
		    + C_{n,2(s-1)}C_{n,0}     ,
			  \\
\displaystyle\sum_{k=0}^{s-2} C_{n,2(s-k)-3}(C_{n,2k+1}-C_{n+1,2k})+\frac{1}{4}\sum_{k=0}^{s-1}C_{n+1,2k}C_{n+1,2(s-1-k)}   ,
		\end{array}  \right.
	\end{equation*}
for $n$ even and odd, respectively.	
\end{teo}

\begin{demo}
The result follows by a direct application of  Lemma \ref{teo the last one}  and Lemma \ref{lemma number of words i and j with a affine length given A}. 
\end{demo}

\medskip
\noindent
Table \ref{tabla valores dns} and Table \ref{tabla valores bns} show some  values of  $d_n^s$ and $b_n^s$, respectively.  There,  rows correspond to $n$ and  columns to $s$. We recall that $b_n^s=a_n^s-d_n^s$ and that the value of $a_n^s$ is given by Theorem \ref{teo closed formula for Cs} and Theorem \ref{teo enumeration normal words}. We collect the information obtained in this section in a polynomial, $P_n(v)\in \mathbb{N}[v]$, defined by 
$ P_n(v)= \sum_{s=0}^n b_n^s v^s$. Of course, $p_n:=P_n(1) = |\blobbed |$. Theorem \ref{theorem monomial basis for symplectic} tells us that $ \dim \simp = p_n$. The first nine terms of the sequence $(p_n)$ are  $5,\, 19,\, 84,\, 335,\, 1428,\, 5748,\, 24104,\, 97287,\, 404148$.

\begin{table}[h]
\begin{equation*}
	\scalebox{.75}{\begin{tabular}{|c||c|c|c|c|c|c|c|c|c|c|} \hline
$d_n^s $  & 0 & 1  & 2 & 3 & 4& 5 & 6 & 7 & 8  & 9  \\ \hline \hline
1 & 0 & 1  & 4 & 4 & 4 & 4 & 4 & 4 & 4 & 4\\ \hline
2 & 0 & 4 &  13 &  16	 & 16  & 16 & 16 & 16 &16 &16\\ \hline
3 &  0 & 9   &  42 &  61 &  64 & 64 & 64 & 64 & 64 & 64 \\  \hline
4 &  0 &  36  & 148  & 228  & 253 & 256 & 256 & 256  & 256  & 256  \\ \hline
5 &  0 &  100  & 500  & 845  & 990 & 1021 & 1024 &  1024 & 1024 & 1024\\ \hline
6 &  0 &  400  & 1825  & 3160  & 3846 & 4056 & 4093 & 4096 & 4096 & 4096\\ \hline
7 &  0 &  1225  & 6370  & 11711  & 14868 & 16051 & 16338 & 16381 & 16384  & 16384 \\ \hline
8 &  0 &  4900  & 23716  & 44100   & 57428  & 63308 & 65108 & 65484 & 65533 & 65536 \\ \hline
9 &  0 &  15876  & 84672  & 164304  & 221004 & 249012 & 259008 & 261609 & 262086 & 262141 \\ \hline
\end{tabular}}
\end{equation*}
\caption{Values of $d_n^s$}
\label{tabla valores dns}
\end{table}

\begin{table}[h]
\begin{equation*}
\scalebox{.8}{\begin{tabular}{|c||c|c|c|c|c|c|c|c|c|c|c|} \hline
$b_n^s$ & 0 & 1  & 2 & 3 & 4 &5&6&7&8&9 \\ \hline \hline
1 & 2 & 3  & 0 & 0 & 0 & 0 & 0 & 0 & 0 & 0  \\ \hline
2 & 6 & 10 &  3 &  0	 & 0  & 0 & 0 & 0 & 0 & 0 \\ \hline
3 &  20 & 41   &  20 &  3 &  0 & 0 & 0 & 0 & 0 & 0  \\  \hline
4 &  70 &  146  &  90 & 26  & 3 & 0 & 0 & 0 & 0 & 0  \\ \hline
5 &  252 &  572  & 412  & 157   & 32 & 3 & 0 & 0 & 0 & 0  \\ \hline
6 &  924 &  2108  &  1673 & 778  & 224 & 38 & 3 & 0 & 0 & 0  \\ \hline
7 &  3432 & 8213 & 7072     & 3733   &   1304 & 303  & 44 &  3 & 0 & 0 \\ \hline
8 &   12870 & 30850   & 28050  & 16402   & 6714  & 1954 & 394 &  50 & 3 & 0 \\ \hline
9 &  48620 & 120260   & 115112  &  72608 & 33044 & 11156 & 2792 & 497 & 56 & 3 \\ \hline
\end{tabular}}
\end{equation*}
\caption{Values of $b_n^s$}
\label{tabla valores bns}
\end{table}

\newpage

\section*{References}
\bibliographystyle{myalpha} 
\bibliography{mybibfile}

\end{document}